\documentclass[11pt]{amsart}

\usepackage{amsmath,amssymb,amsfonts,amsthm,graphicx}

\usepackage{pinlabel}

\usepackage[usenames,dvipsnames]{color}

\usepackage[usenames,dvipsnames]{color}

\newtheorem{thm}{Theorem}[section]
\newtheorem{prop}[thm]{Proposition}

\newcommand{\alg}{\mathcal{A}}

\newcommand{\dd}{\partial}

\newcommand{\Z}{\mathbb{Z}}
\newcommand{\R}{\mathbb{R}}
\newcommand{\A}{\mathcal{A}}

\def\ms{}
\def\dr{}

\newtheorem{theorem}{Theorem}[section]
\newtheorem{lemma}{Lemma}[section]
\newtheorem{corollary}{Corollary}[section]
\newtheorem{definition}{Definition}[section]
\newtheorem{proposition}{Proposition}[section]

\theoremstyle{definition}

\newtheorem{example}[theorem]{Example}
\newtheorem{remark}[theorem]{Remark}

\begin{document}

\title[Cellular LCH for surfaces, I]{Cellular Legendrian contact homology for surfaces, part I}

\thanks{We thank Tobias Ekholm and Lenny Ng for useful conversations.  The second author is supported by grant 317469 from the Simons Foundation. He thanks the Centre de Recherches Mathematiques for hosting him while some of this work was done.}

\author{Dan Rutherford}
\address{Ball State University}
\email{rutherford@bsu.edu}

\author{Michael Sullivan}
\address{University of Massachusetts, Amherst}
\email{sullivan@math.umass.edu}

\begin{abstract}
We give a computation of the Legendrian contact homology (LCH) DGA for an arbitrary generic Legendrian surface $L$ in the $1$-jet space of a surface.  As input we require a suitable cellular decomposition of the base projection of $L$.  A collection of generators is associated to each cell, and the differential is given by explicit matrix formulas.  In the present article, we prove that the equivalence class of this cellular DGA does not depend on the choice of decomposition, and in the sequel \cite{RuSu2} we use this result to show that the cellular DGA is equivalent to the usual Legendrian contact homology DGA defined via holomorphic curves.   
 Extensions are made to allow Legendrians with non-generic cone-point singularities.  We apply our approach to compute the LCH DGA for several examples including an infinite family, and to give general formulas for DGAs of front spinnings allowing for the axis of symmetry to intersect $L$.  
\end{abstract}

\maketitle

{\small \tableofcontents}

\section{Introduction}


The Legendrian contact homology (abbrv. LCH) algebra of a Legendrian submanifold, $L$, in the $1$-jet space, $J^1M$, of a manifold $M$ is a differential graded algebra (abbrv. DGA) with generators determined by the double points of the Lagrangian projection of $L$ and with differential defined via counts of suitable spaces of pseudo-holomorphic disks.  The LCH algebra was outlined in a more general setting as part of the wider class of symplectic field theory invariants in \cite{EliashbergGiventalHofer00} and, for Legendrian submanifolds in $1$-jet spaces, was first rigorously constructed in \cite{EkholmEtnyreSullivan05b, EkholmEtnyreSullivan07}.  


Aside from being able to distinguish Legendrian isotopy classes, 
  the LCH algebra is useful in identifying the border between symplectic flexibility and rigidity.  For example,  the existence of an augmentation of the LCH algebra of a displaceable Legendrian $L \subset J^1 M$ implies
that $L$ satisfies the Arnold conjecture for exact Lagrangian immersions--roughly speaking, the projection of $L$ to $T^*M$ must have enough double points to generate half of the homology of $L$ \cite{EkholmEtnyreSullivan05c, EESabloff}.  In contrast, Legendrian submanifolds whose LCH algebras have vanishing homology have been constructed to have only $1$ double point \cite{EEMurphySmith}.  
A sample of other applications of LCH algebras include (i)  surgery exact triangles explaining the effect of Legendrian surgery/symplectic handle attachment on the symplectic and contact homology of Stein manifolds and their contact boundaries \cite{BourgeoisEkholmEliashberg12};  (ii)  a complete invariant for smooth knots in $\R^3$ \cite{Ng08, EkholmEtnyreNgSullivan11, GordonLidman15, EkholmNgShende16}  (iii)  connections between the LCH algebra of $1$-dimensional Legendrian knots and  topological invariants of knots such as the Kauffman and HOMFLY polynomial \cite{Rutherford06, HenryRutherford15};  (iv) results concerning Lagrangian cobordisms in symplectizations \cite{EkholmKalmanHonda12, DimitroglouRizell12,  Chantraine15}.    








For $1$-dimensional Legendrians (in a $1$-jet space), the 
LCH algebra is combinatorially computable, as
the pseudo-holomorphic disks 
used to define the differential may be identified using the Riemann mapping theorem \cite{Chekanov02}.   
Moreover, after placing the front diagram of $L$ into a standard form, the count of such disks can be carried out in an algorithmic manner \cite{Ng03}.  As a result, the LCH algebra is an effective tool for distinguishing Legendrian knots with reasonably sized diagrams, and is useful in distinguishing many pairs of Legendrian knots in the current tables.



In contrast, the explicit computation of Legendrian contact homology for higher dimensional Legendrians is much more challenging as direct counts of the relevant moduli of pseudo-holomorphic disks are not feasible.  For Legendrians in $1$-jet spaces, a major step towards explicit computation was carried out in \cite{Ekholm07} where it was shown that the count of pseudo-holomorphic disks can be replaced by a count of certain rigid gradient flow trees (abbrv. GFTs).   
This replaces the need to identify pseudo-holomorphic disks with a problem phrased entirely in terms of  ODEs.  However, GFTs are themselves complicated global objects.  Consequently, explicit computations of the full DGA have been limited to several restricted classes of 2-dimensional Legendrians: conormal lifts of smooth knots in $\R^3$ \cite{EkholmEtnyreNgSullivan11}; two-sheeted Legendrians surfaces \cite{DimitroglouRizell11}; and Legendrian tori constructed from $1$-dimensional Legendrian knots via front and isotopy spinnings \cite{EkholmKalman08}.  
There have also been a number of partial DGA computations in higher dimensions: front spinning
\cite{EkholmEtnyreSullivan05a, DimitroglouRizellGolovko14}; 
and DGA maps induced from Lagrangian cobordims \cite{DimitroglouRizell12, Lowell15}.




In this paper and its companion \cite{RuSu2}, we give a computation of the LCH algebra with $\Z/2$-coefficients for all $2$-dimensional Legendrians  in the $1$-jet space of a surface.  For a surface $S$ and a Legendrian $L \subset J^1S$, we define a cellular DGA,  $(\mathcal{A},\partial)$, that requires as input a suitable polygonal decomposition of the base projection (to $S$) of $L$.  To each $0$-, $1$-, and $2$-cell in the decomposition we associate a collection of generators with the precise number of generators determined by the appearance of the front projection of $L$ above the cell.  After placing the generators associated to a particular cell and its boundary cells into matrices, the differential of $(\mathcal{A},\partial)$ is characterized by explicit formulas.  See Figure \ref{fig:2cell} for a brief summary.  

The standard notion of equivalence used when considering Legendrian contact homology DGAs is known as stable tame isomorphism. 
This is at least as strong as the notion 
of homotopy equivalence for DGAs. 
 Our main result is:
\begin{theorem} \label{thm:FirstStatement} For any $2$-dimensional $L \subset J^1S$, the cellular DGA $(\mathcal{A},\partial)$ is equivalent to the Legendrian contact homology DGA of $L$ with coefficients in $\Z/2$.  In particular, up to stable tame isomorphism, $(\A, \partial)$ does not depend on the choice of polygonal decomposition, and $(\A,\partial)$ is a Legendrian isotopy invariant.
\end{theorem}

In the present article, we define the cellular DGA, and prove independence of the choice of polygonal decomposition.  
In \cite{RuSu2}, we construct a stable tame isomorphism between the Legendrian contact homology DGA of $L$ and the cellular DGA of $L$ obtained from a particular choice of polygonal decomposition.  
Here, we use the reformulation of the differential in terms of GFTs \cite{Ekholm07}.  
We produce relatively explicit coordinate models for Legendrian surfaces to make our computations of LCH.

The differential in the cellular DGA is local in the sense that it preserves sub-algebras corresponding to generators associated to the closure of any particular  cell.  Starting with the splash construction \cite{Fu}, similar localization techniques have been developed in the case of $1$-dimensional Legendrians where they have been applied to study augmentations of the LCH algebra and to relate the LCH algebra with invariants defined via the techniques of generating families and constructible sheaves \cite{FuRu, Sivek11, NRSSZ}.  
The localizing splash construction has since been generalized to higher dimensional Legendrians \cite{HarperSullivan, Lowell15}.
We expect that the cellular DGA may also provide an approach for extending the other applications to the $2$-dimensional case.  In particular, in the upcoming article \cite{RuSu3} we apply Theorem \ref{thm:FirstStatement} to give a necessary and sufficient condition for the existence of augmentations in terms of a $2$-dimensional generalization of the Morse Complex Sequences studied in \cite{HenryRutherford13}, which exists in the presence of a generating family for the Legendrian.

Although the number of generators of $(\A,\partial)$ can be quite large, it is often possible to reduce the number of generators to a much smaller size while retaining an explicit formula for the differential.  
The idea is to cancel generators in pairs via a suitable stable tame isomorphisms.
Throughout Section \ref{sec:Examples}, we illustrate this process with explicit examples.  In that section we also show how to the relax conditions on the cell decomposition (for example, allowing cone points which arise naturally in \cite{EkholmEtnyreNgSullivan11, DimitroglouRizell11}), which further reduces the set of generators. See Sections \ref{ssec:ConePoints} and \ref{ssec:MultipleCrossings}.





In Section \ref{sec:Background}, we recall background material on DGAs and Legendrian surfaces in $1$-jet spaces.  
The definition of the cellular DGA is given in Section \ref{sec:DefDGA}, where it is verified that  
differential $\partial$ has degree $-1$ and satisfies $\partial^2 =0$;    see Theorem \ref{thm:Summary}.  
Theorem \ref{thm:IndCellD} of Section \ref{sec:Ind} shows that for any fixed Legendrian $L \subset J^1S$, the stable tame isomorphism type of the cellular DGA is independent of the choice of polygonal decomposition and of other auxiliary choices.   In Section \ref{ssec:ExampleSphere} and \ref{ssec:ExampleST}, we compute the LCH of two Legendrian spheres.
In Section \ref{ssec:FrontSpinning} we compute the LCH for front-spun knots, expressing it as a natural DGA operation.
Some of these computations agree with pre-existing ones. 
In Section \ref{sec:Lnsigma}, we provide a family of spheres many of which have identical Linearized homology groups, but are distinguished by product operations that would be difficult to identify without explicit formulas for differentials. 



We remark that knowledge of LCH is neither assumed nor required in the current article.  The required background on LCH and GFTs is postponed until \cite{RuSu2}.
\section{Background}  \label{sec:Background}

\subsection{Differential graded algebras}

In this article, we consider Differential Graded Algebras (abbrv.  {\bf DGAs}) with $\Z/2$-coefficients and $\Z/m$-grading for some non-negative, integer $m$.  The differential $\partial: \mathcal{A} \rightarrow \alg$ of a DGA $(\A,\partial)$ has degree $-1$ and satisfies the Liebniz rule $\partial(x\cdot y) = (\partial x) y + (-1)^{|x|}x (\partial y)$.

A {\bf based DGA}   is a DGA $(\mathcal{A}, \partial)$ with a chosen subset $\mathcal{B} = \{q_1, \ldots, q_r\} \subset \A$ such that  $\mathcal{A}$ is freely generated by $B$, as an associative (non-commutative) $\Z/2$-algebra with identity.  That is, $\mathcal{A} \cong (\Z/2)\langle q_1, \ldots, q_r \rangle$ so that  
elements of $\mathcal{A}$ can be written uniquely as $\Z/2$-linear combinations of words in the elements of $\mathcal{B}$.  Furthermore, we require that the generating set $\mathcal{B}$ consists of homogeneous elements, so that the  $\Z/m$-grading, $\alg = \bigoplus_{i \in \Z/m} \alg_i$, is determined by the assignment of degrees to generators, $|\cdot| : \mathcal{B} \rightarrow \Z/m$, and the requirement that $|x \cdot y| = |x| + |y|$ when $x$ and $y$ are homogeneous elements.  

We say that the differential $\partial$ is {\bf triangular} with respect to a total ordering of the generating set given by $\mathcal{B} =\{q_1, \ldots, q_n\}$ if for all $ 1\leq i \leq n$, $\dd q_i \in F_{i-1}\A$ where  $F_0 \A= \Z/2$ and $F_i\A \subset \A$
 denotes the subalgebra generated by $q_1, \ldots, q_{i}$.


\subsubsection{Stable tame isomorphism}
\label{ssec:STI}

An {\bf elementary automorphism} of $\A$ is a graded algebra map $\phi: \A \rightarrow \A$ such that for some $q_i \in \mathcal{B}$, we have $\phi(q_j) =q_j$ when $j\neq i$, and $\phi(q_i) = q_i + v$ where $v$ belongs to the subalgebra generated by $\mathcal{B} \setminus \{q_i\}$.  A {\bf tame isomorphism} between based DGAs  is an isomorphism of DGAs  
$\psi:(\A, \partial) \rightarrow (\A',\partial')$, i.e. a graded algebra isomorphism satisfying $\psi\circ \partial = \partial' \circ \psi$, that 
is a composition of elementary automorphisms of $\A$ followed by an isomorphism $\A \rightarrow \A'$ that extends some bijection of generating
sets $\mathcal{B} \cong \mathcal{B}'$.

The {\bf degree $i$ stabilization} of a based DGA, $(\A,\partial)$, with generating set, $\mathcal{B}$, is the based DGA, $(S\A, \partial')$, with generating set, $S\mathcal{B}= \mathcal{B}\cup \{a,b\}$, where the new generators are assigned degrees $|a| = i$, $|b|= i-1$, and $\partial'$ satisfies 
\[
\partial' a= b; \quad \partial'b= 0; \quad \mbox{and} \quad \partial'|_{\mathcal{A}} = \partial.
\]

\begin{definition}  Two based DGAs are {\bf stable tame isomorphic} if after stabilizing each of the DGAs some (possibly different) number of times they become isomorphic by a tame isomorphism.
\end{definition}
It can be checked that stable tame isomorphism satisfies the properties of an equivalence relation.  We occasionally say that two based DGAs are {\bf equivalent} to mean that they are stable tame isomorphic.

The following theorem will serve as our primary method for producing stable tame isomorphisms.

\begin{theorem} \label{thm:Alg} Let $(\A, \dd)$ be a based DGA such that $\partial$ is triangular with respect to the ordered generating set $\mathcal{B} = \{q_1, \ldots, q_n\}$.  Suppose that 
\[
\mbox{$\dd q_k = q_l + w$ where $w \in F_{l-1}\A$,}
\]
 and let $I=I(q_k,\partial q_k)$ denote the $2$-sided ideal generated by $q_k$ and $\partial q_k$.  Then, $(\A/I,\partial)$ is free and triangular with respect to the ordered generating set $\mathcal{B}'=\{[q_1], \ldots, \widehat{[q_l]}, \ldots, \widehat{[q_k]}, \ldots, [q_n]\}$.  Moreover, with these generating sets, $(\A,\partial)$ and $(\A/I, \partial)$ are stable tame isomorphic. 
\end{theorem}

Typically we use the same notation for an element of $\A$ and its equivalence class in a quotient of $\A$.  Note that to write the differential $\partial q_i$ in $\A/I$ in terms of the generating set, $\mathcal{B}'$, we simply replace all occurrences of $q_k$ (resp. $q_l$) in $\partial q_i \in \A$ by $0$ (resp. by $w$).

\begin{proof}
This is essentially due to \cite[Section 8.4]{Chekanov02}.

If necessary, we can reorder $\{q_1, \ldots, q_n\}$ as $\{q_1, \ldots, q_l, q_k, q_{l+1}, \ldots, \widehat{q_k}, \ldots, q_n\}$ to arrange that $k$ and $l$ are consecutive while preserving the triangular property of $\dd$.  Thus, we can suppose $\partial q_{k} = q_{k-1} + w$ with $w \in F_{k-2} \A$, and show $(\A/I, \partial)$ is stable tame isomorphic to $(\A,\partial)$ where $I := I(q_{k}, \partial q_{k})$.  
It is clear that $\A/I$ is indeed freely generated by $\mathcal{B}'$ since we have the vector space decomposition, $\A = \A' \oplus I$ where $\A' \subset \A$ is the subalgebra generated by $\mathcal{B} \setminus \{q_k,q_{k-1}\}$.
We next need the following.


\begin{lemma} \label{lem:DGAmapf}
There exists a DGA map $f: \A/I \rightarrow \A$ such that
for any $[q_i] \in \mathcal{B}'$, ($i \neq k,k-1$), $f([q_i]) = q_i + w_i$ with $w_i \in F_{i-1}\A$.
\end{lemma}
\begin{proof}  The projection $p: \A \rightarrow  \A/I$ is a DGA map.  We will construct $f$ to be a homotopy inverse to $p$.

Note that $(\A/I,\partial)$ inherits a filtration from $(\A,\partial)$, where $F_i( \A/I) = p(F_i\A)$ is the sub-algebra generated by $[q_1],\ldots, [q_i]$.  In particular, $F_{k}(\A/I) = F_{k-1}(\A/I) = F_{k-2}(\A/I)$ are all freely generated by $[q_i]$ with $1 \leq i \leq k-2$.  
Begin by defining $f_0: F_{k} \A/I \rightarrow F_k\A$, and $H_0:F_k \A \rightarrow F_k\A$ on generators by 
\[
f_0([q_i]) = q_i,  \quad \mbox{for $1\leq i\leq k-2$;} \quad H_0(q_{i}) = \left\{ \begin{array}{cr} q_{k}, & \mbox{if $i=k-1$,} \\
0 & \mbox{else}. \end{array} \right.
\] 
Then, extend $f_0$ as an algebra homomorphism, and extend $H_0$ to be a $(f_0 \circ p, \mathit{id}_{F_k\A})$-derivation, meaning that $H_0(x\cdot y) = H_0(x) \cdot \mathit{id}_{\A_0}(y) + (f_0\circ p)(x) \cdot H_0(y)$.   A direct verification shows that 
\[
f_0 \circ \partial_{\A/I} = \partial_\A \circ f_0;  \quad \mbox{and} \quad f_0 \circ p - \mathit{id}_{\A_0} = \partial_\A \circ H_0 + H_0 \circ \partial_\A
\]
hold when applied to generators of $F_{k} (\A/I)$ and $F_k \A$, respectively.   This implies that the identities in fact hold  on all of $F_k (\A/I)$ and $F_k \A$.  [In general, if $g_1,g_2: (\mathcal{A}_1, \partial_1) \rightarrow (\mathcal{A}_2, \partial_2)$ are DGA maps, and $H: \A_1 \rightarrow \A_2$ is a $(g_1,g_2)$-derivation, then $g_1 -g_2 = \partial_{\A_2}\circ H + H \circ \partial_{\A_1}$ holds if and only if it holds on a generating set for $\mathcal{A}_1$.]

Now, for $i\geq 0$, we inductively define maps $f_i: (F_{k+i} \A/I, \partial) \rightarrow (F_{k+i} \A,\partial)$ and $H_i: (F_{k+i} \A,\partial) \rightarrow (F_{k+i} \A,\partial)$ to satisfy 
\begin{enumerate}
\item $f_i$ is a chain map, 
\item $H_i$ is a $(f_i \circ p, \mathit{id}_{F_{k+i}\A})$-derivation, and  
\item $f_i \circ p - \mathit{id}_{F_{k+i}\A} = \partial_\A \circ H_i + H_i \circ \partial_\A$.
\end{enumerate} 
Supposing $f_i$ and $H_i$ have been defined, we extend them to maps $f_{i+1}: (F_{k+i+1} \A/I, \partial) \rightarrow (F_{k+i+1} \A,\partial)$ and $H_{i+1}: (F_{k+i+1} \A,\partial) \rightarrow (F_{k+i+1} \A,\partial)$.  Begin by setting
\[
f_{i+1}([q_{i+1}]) = q_{i+1} + H_i( \partial_\A( q_{i+1}) ).
\]
Note that $\partial_\A(q_{i+1}) \in F_{i}\A$, since $\partial_A$ is triangular.  Thus, this definition makes sense, and we indeed have $f_{i+1}([q_{i+1}]) = q_{i+1} + w_{i+1}$ with $w_{i+1} \in F_i \A$.
Next, extend $f_{i+1}$ as an algebra homomorphism, put $H_{i+1}(q_{i+1}) =0$, and then extend $H_{i+1}$ to $F_{k+i+1}\A$ as a $(f_{i+1} \circ p, \mathit{id}_{F_{k+i+1}\A})$-derivation.  We check that $f_{i+1}$ and $H_{i+1}$ again satisfy (1)-(3).  Note that (2) holds by construction.  Once  (1) is known, (3) follows from the definition of $f_{i+1}([q_{i+1}])$ since the identity holds when applied to generators of $F_{k+i+1}\A$. 

To verify (1), compute:
\[
f_{i+1}\circ \dd_{\A/I} ([q_{i+1}]) = f_{i+1} \circ p(\partial_\A q_{i+1}) = f_i \circ p (\partial_\A q_{i+1}) = \partial_\A q_{i+1} + (\partial_\A \circ H_i + H_i \circ \dd_\A)(\dd_\A q_{i+1}) =
\]
\[
\partial_\A q_{i+1} + \dd_\A \circ H_i( \dd_A q_{i+1}) = \partial_A( q_{i+1} + H_i( \dd_A q_{i+1}) ) = \partial_A\circ f_{i+1}([q_{i+1}])
\]
where at the second equality we used that $\partial_\A q_{i+1} \in F_i\A$ since $\partial$ is triangular.

Thus, the construction proceeds inductively with $f= f_n$ 
satisfying the stated requirements of the Lemma.
\end{proof}

We now construct a stable tame isomorphism.  Denote the generators of the stabilization, $(S(\A/I), \partial')$, as  $\overline{q}_1, \ldots, \overline{q}_n$ where, for $i \notin \{ k,k-1\}$, $\overline{q}_i = [q_i]$, and $\overline{q}_k$ and $\overline{q}_{k-1}$ are the two new generators from stabilization with $\partial' \overline{q}_{k}= \overline{q}_{k-1}$; the degree of the stabilization is chosen so that $|\overline{q}_{k}| = |q_k|$ and $|\overline{q}_{k-1}| = |q_{k-1}|$.  Define an algebra map $F:(S(\A/I), \dd') \rightarrow (\A,\dd)$ by 
\[
F(\overline{q_i}) = \left\{ \begin{array}{cr}  f(\overline{q_i}), & \mbox{for $i\notin\{k,k-1\}$}, \\
q_k,  & \mbox{for $i=k$}, \\
 q_{k-1}+w, & \mbox{for $i=k-1$}.
  \end{array} \right.
\]

\medskip

\noindent {\bf $F$ is a chain map:}  It suffices to verify that $F\circ \dd'= \dd \circ F$ holds when applied to generators.  Compute 
\begin{align*}
& F\circ \dd' (\overline{q}_{k}) = F(\overline{q}_{k-1}) = q_{k-1} + w = \partial q_k = \partial \circ F(\overline{q}_{k}); \\
&
F\circ \dd' (\overline{q}_{k-1}) = 0 = \partial^2 q_{k} = \partial \circ F(\overline{q}_{k-1}).
\end{align*}
For $i \notin \{k,k-1\}$, $\overline{q}_i \in \A/I$, so we have
\[
(F \circ \dd')(\overline{q}_i) = f ( \partial_{\A/I}([q_i])) = \partial_A( f([q_i]) ) = (\partial \circ F)(\overline{q}_i). 
\]


\medskip

\noindent {\bf $F$ is a tame isomorphism:}  Note that 
for all $1 \leq i \leq n$, $F(\overline{q}_i) = q_i  + w_i$ where $w_i \in F_{i-1} \A$.   
Thus, we can write
\[
F =  f_n \circ \cdots \circ f_1 \circ \sigma
\]
where $f_{i}(q_j) = \left\{ \begin{array}{cr} q_i + w_i & \mbox{for $i=j$} \\ q_j & \mbox{for $j\neq i$},  \end{array} \right.$ and $\sigma(\overline{q}_i) = q_i$.





 
\end{proof}

\begin{remark}  \label{rem:variation}
\begin{enumerate}
\item[(i)] The notion of stable tame isomorphism appears in \cite{Chekanov02} and is now prevalent in the literature surrounding Legendrian contact homology.
\item[(ii)] 
A variation on stable tame isomorphism arises from requiring that the word $w$ that appears in $\phi(q_i)$ in the definition of elementary isomorphism belongs to $F_{i-1}\A$ with respect to some ordering of generators for which $\partial$ is triangular.  We do not know whether the equivalence on based algebras generated by such isomorphisms is distinct from stable tame isomorphism.  However, we note in passing that the proof of Theorem \ref{thm:IndCellD} shows that the DGAs corresponding to distinct compatible cell decompositions of $L$ are related by this stronger equivalence.
\end{enumerate}
\end{remark}



\subsection{Legendrian surfaces}  Let $S$ be a surface.  The $1$-jet space, $J^1S = T^*S \times \R$, has a standard contact structure given by the kernel of the contact form $dz - \lambda$ where $\lambda$ is the Louiville $1$-form on $T^*S$ and $z$ is the coordinate on the $\R$ factor.  Any choice of local coordinates $(x_1,x_2)$ on $U \subset S$ leads to coordinates $(x_1,x_2,y_1,y_2,z)$ on $J^1 U \subset J^1S$ with respect to which the contact form appears as 
\[
dz - \lambda = dz - \sum_{i=1}^2 y_i \, dx_i.
\]
We denote by 
\[
\pi_{xz}:  J^1S \rightarrow S \times \R;  \quad \mbox{and} \quad \pi_x: J^1S \rightarrow S
\]
the projections which are known respectively as the {\bf front projection} and {\bf base projection}.

A surface $L \subset J^1S$ is {\bf Legendrian} if $(dz-\lambda)|_L = 0$.  
The singularities of front and base projections of Legendrian submanifolds are well studied, cf.  \cite{ArnoldGuseinZadeVarchenko}.  In this $2$-dimensional setting, after possibly perturbing $L$ by a small Legendrian isotopy, we may assume that $L = \bigsqcup_{i=0}^2L_i$ such that $L_1 \sqcup L_2 \subset L$ is an embedded $1$-manifold and $L_2$ is a discrete set; near each $L_i$ the front projection has the following standard forms.


\begin{itemize}
\item 
Any $p \in L_0$ has a neighborhood, $N \subset L$, whose front projection is the graph of a smooth function defined on a neighborhood $U$ of $\pi_x(p)$, i.e.
\[
\exists\,\, f:U \rightarrow \R,  \quad \mbox{such that} \quad \pi_{xz}(N) = \{(x,z) \,|\, z= f(x) \}.
\]
In particular, the restriction of $\pi_{xz}$ to $L_0$ is an immersion.  We often refer to connected subsets of $L_0$ and/or their front projections as {\bf sheets} of $L$.

\item Any point in $L_1$ has a neighborhood whose front projection is diffeomorphic to a standard {\bf cusp edge}, i.e. a semi-cubical cusp crossed with an interval, see Figure \ref{fig:Codim1Both}.  Locally cusp edges belong to the closure of two sheets of $L$ that we refer to as the upper and lower sheets of the cusp edge.

\item Any point in $L_2$ has a neighborhood whose front projection is diffeomorphic to a standard {\bf swallow tail} singularity.  In fact, there are two types of swallow tail singularities that we refer to as {\bf upward} and {\bf downward} swallow tails.  
The Legendrian, $L_F$, in $J^1\R^2$ defined by the generating family,
\[
F:\R^2 \times \R \rightarrow \R, \quad F(x_1,x_2;e)= e^4-x_1 e^2+ x_2 e,
\]
has a standard upward swallow tail singularity when $x_1=x_2=0$.  The front projection of $L_F$ is given by 
\[
\pi_{xz}(L_F) = \{(x_1,x_2, f_{x_1,x_2}(e)) \, |\, (f_{x_1,x_2})'(e) =0\} \quad \mbox{where $f_{x_1,x_2}(e) = F(x_1,x_2;e)$.}
\]
The standard downward swallow tail singularity is obtained by negating the $z$-coordinate.
See \cite[p. 47]{ArnoldGuseinZadeVarchenko}.

A swallow tail point, $s_0$, lies in the closure of two cusp edges of $\pi_{xz}(L)$; 
the base projections of these two cusp edges form a semi-cubical cusp curve in $S$ with the cusp point at  $\pi_x(s_0)$. 
The two cusp edges that meet at a swallowtail point have either their upper or lower sheet in common.  (The other two sheets associated with these cusp edges meet at the crossing arc that terminates at $s_0$.) When the swallowtail is \emph{upward} (resp. \emph{downward}) the common sheet appears above (resp. below) the two sheets that cross in the front projection.
\end{itemize}

As long as $L \subset J^1S$ is an \emph{embedded} Legendrian, the front projection of $L_0$ must be transverse to itself.  [Since the $y_0$ and $y_1$ coordinates are recovered by $y_i = \partial z/\partial x_i$.]  Again, after a small perturbation, we can assume 
$L_0$ is also transverse to $L_1$ and $L_2$ as well as to the self intersection set of $L_0$ to arrange that the self intersections of $\pi_{xz}(L)$ are as follows:
\begin{itemize}
\item Along {\bf crossing arcs} two sheets of $L$ intersect transversally.
\item At isolated {\bf triple points}  three sheets of $L$ meet in a manner that is pairwise transverse, and so that the crossing arc between any two of the sheets is transverse to the third sheet.  In the base projection, three crossing arcs meet at a single point, but are pairwise transverse.  
\item At isolated points, $x$, a {\bf cusp-sheet intersection} occurs where a cusp edge cuts transversally through a sheet of $L$.  The crossing arcs between this sheet and the upper and lower sheets of the cusp edge both end at $x$.  The base projections of these two crossing arcs meet at a semi-cubical cusp point on the base projection of the cusp edge. 
\end{itemize}

We refer to the union of the crossing arcs, cusp edges, and swallowtail points in either the front or base projection (depending on context) as the {\bf singular set} of $L$.  We denote the base projection of the singular set as $\Sigma$.  
When $L$ has generic base projection, we write
\[
\Sigma = \Sigma_1 \sqcup \Sigma_2
\]
where points in $\Sigma_1$ are in the image of a single crossing are or cusp edge, and $\Sigma_2 \subset S$ is a finite set containing the image of codimension $2$ parts of the front projection as well as points in the transverse intersection of the image of two distinct cusp edges and/or crossing arcs.  Often, we will refer to the base projections of cusp edges, and crossing arcs as the {\bf cusp locus} and the {\bf crossing locus}.  
See Figures \ref{fig:Codim1Both} and \ref{fig:Codim2Both} for the local appearance of the singular set.

\begin{figure}
\labellist
\small
\pinlabel $x_1$ [l] at 38 48
\pinlabel $x_2$ [b] at 2 86
\pinlabel $x_1$ [l] at 38 200
\pinlabel $x_2$ [l] at 22 226
\pinlabel $z$ [b] at 2 238
\pinlabel \textbf{Cusp~Edge} [t] at 202 -4
\pinlabel \textbf{Crossing~Arc} [t] at 394 -4
\endlabellist
\centerline{\includegraphics[scale=.6]{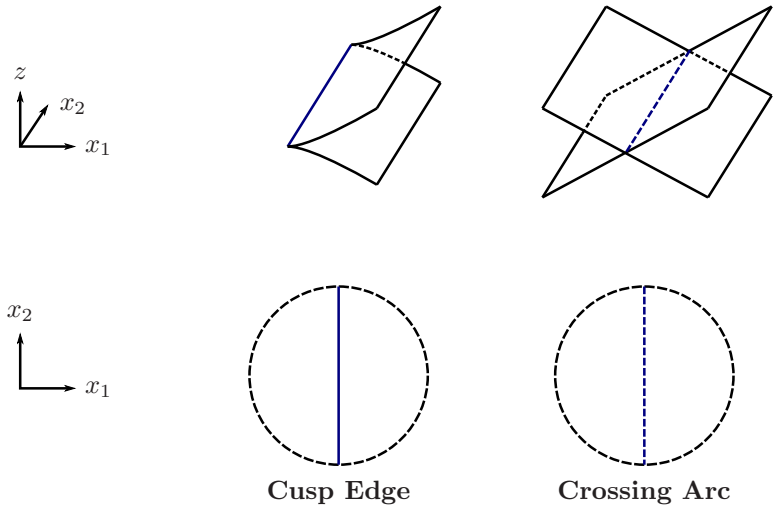}} 

\quad

\quad

\caption{Codimension 1 parts of the singular set pictured in the front and base projection.  In the base projection we picture crossing arcs (resp. cusp edges) as dotted (resp. solid) lines.}
\label{fig:Codim1Both}
\end{figure}

\begin{figure}
\labellist
\small
\pinlabel $x_1$ [l] at 38 48
\pinlabel $x_2$ [b] at 2 86
\pinlabel $x_1$ [l] at 38 200
\pinlabel $x_2$ [l] at 22 226
\pinlabel $z$ [b] at 2 238
\pinlabel \textbf{Triple~Point} [t] at 178 -4
\pinlabel \textbf{Cusp-Sheet~Intersection} [t] at 354 -4
\pinlabel \textbf{Swallow Tail} [t] at 530 -4
\endlabellist
\centerline{\includegraphics[scale=.6]{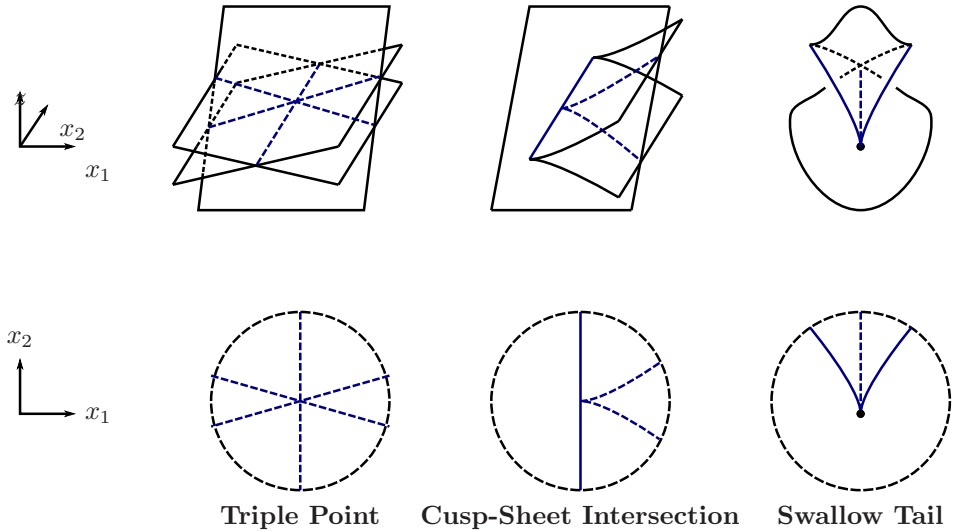}} 

\quad

\quad

\caption{Codimension 2 parts of the singular set pictured in the front and base projection.  Note that the pictured swallow tail is \emph{upward}.  The front projection of a downward swallow tail is obtained by reflecting the $z$-direction.  Additional codimension $2$ points exist in the base projection where two crossing and/or cusp arcs intersect transversally. }
\label{fig:Codim2Both}
\end{figure}

Conversely, any surface in $S \times \R$ matching the above description and without vertical tangent planes is the front projection of a unique Legendrian $L \subset J^1 S$.
See Figure \ref{fig:2DExample}  for an example of a Legendrian surface in $J^1(\R^2)$ pictured in its front and base projection.

\subsubsection{Maslov potentials}
A generic loop $\gamma$ in $L$ is disjoint from swallowtail points and crosses cusp edges transversally.  Let $m(\gamma) = D(\gamma) - U(\gamma)$ where $D(\gamma)$ (resp. $U(\gamma)$) is the number of points where $\gamma$ crosses from the upper to lower sheet (resp. lower to upper sheet) at a cusp edge.  The (minimal) {\bf Maslov number} for $L$,  $m(L) \in \Z_{\geq 0}$ is the smallest positive integer value taken by $m(\gamma)$, or $0$ if all loops have $m(\gamma) =0$.

A {\bf Maslov potential} is a locally constant function
\[
\mu:  L_0 \rightarrow \Z/m(L)
\]  
whose value increases by $1$ when passing from the lower sheet to the upper sheet at a cusp edge.  When $L$ is connected, any two Maslov potentials differ by the overall addition of a constant.







\section{Definition of the cellular DGA $(\A,\partial)$} \label{sec:DefDGA}

In this section we associate a DGA to a Legendrian surface,  $L \subset J^1S$, using an appropriately chosen cell decomposition of $S$.   Initially we assume that the front projection of $L$ does not contain any swallowtail singularities as this simplifies the definition.  In Sections \ref{ssec:defsw}-\ref{sssec:sw2cell}, we complete the definition by addressing the case when swallowtail points are present.

\subsection{Compatible cell decompositions} \label{sec:polygonal}
Let $S$ be a surface.  By a {\bf polygonal decomposition} of $S$, we mean a CW-complex decomposition of $S$, 
\[
S = \sqcup_{i=0}^2 \sqcup_\alpha e^i_\alpha
\]
with characteristic maps $c^i_\alpha: D^i \rightarrow S$ satisfying:
\begin{enumerate}
\item[(i)]  The characteristic maps for $1$-cells are smooth.
\item[(ii)] For any two cell $e_\alpha^2$, preimages of $0$-cells divide the boundary of $D^2$ into intervals that are mapped homeomorphically to $1$-cells by $c_\alpha^2$.   
\end{enumerate}
Note that we allow that the same $1$-cell may appear multiple times in the boundary of a given $2$-cell.  Compare with the decomposition pictured in Figure \ref{fig:mathcalE}.

Let $L \subset J^1S$ be a Legendrian submanifold with generic front projection, and denote by $\Sigma= \Sigma_1\sqcup \Sigma_2 \subset S$ the base projection of the singular set of $L$, where  
$\Sigma_1$ and $\Sigma_2$ are the singularities of codimension $1$ and $2$ respectively.  
We say that a polygonal decomposition $\mathcal{E} = \{e^i_\alpha\}$ of $S$ is {\bf compatible} with $L$ or $L$-compatible if $\Sigma$ is contained in the $1$-skeleton of $\mathcal{E}$.  From the nature of codimension $2$ singularities, it follows that $\Sigma_2$ must be contained in the $0$-skeleton.

\begin{figure}
\labellist
\small
\pinlabel $x_2$ [t] at 62 26
\pinlabel $x_1$ [tr] at -3 3
\pinlabel $z$ [r] at 12 78
\endlabellist
\centerline{\includegraphics[scale=.5]{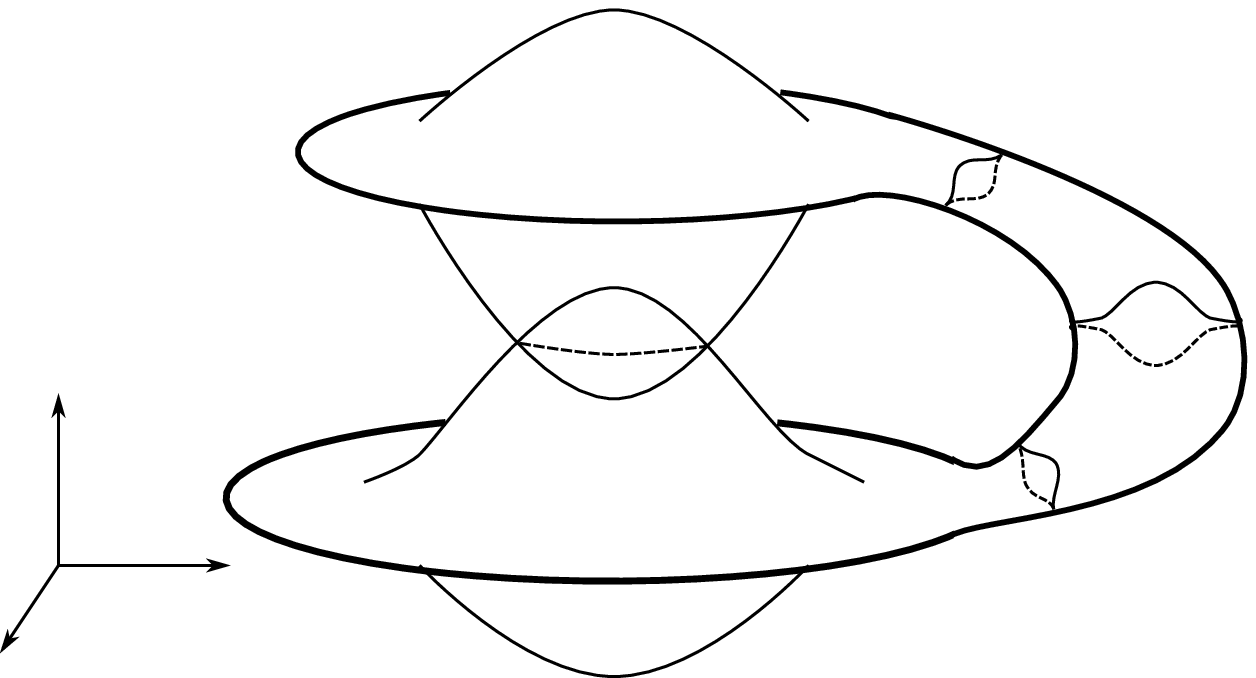} \quad \quad \quad 
\labellist
\small
\pinlabel $x_2$ [t] at 50 282
\pinlabel $x_1$ [r] at -2 240
\endlabellist
\includegraphics[scale=.4]{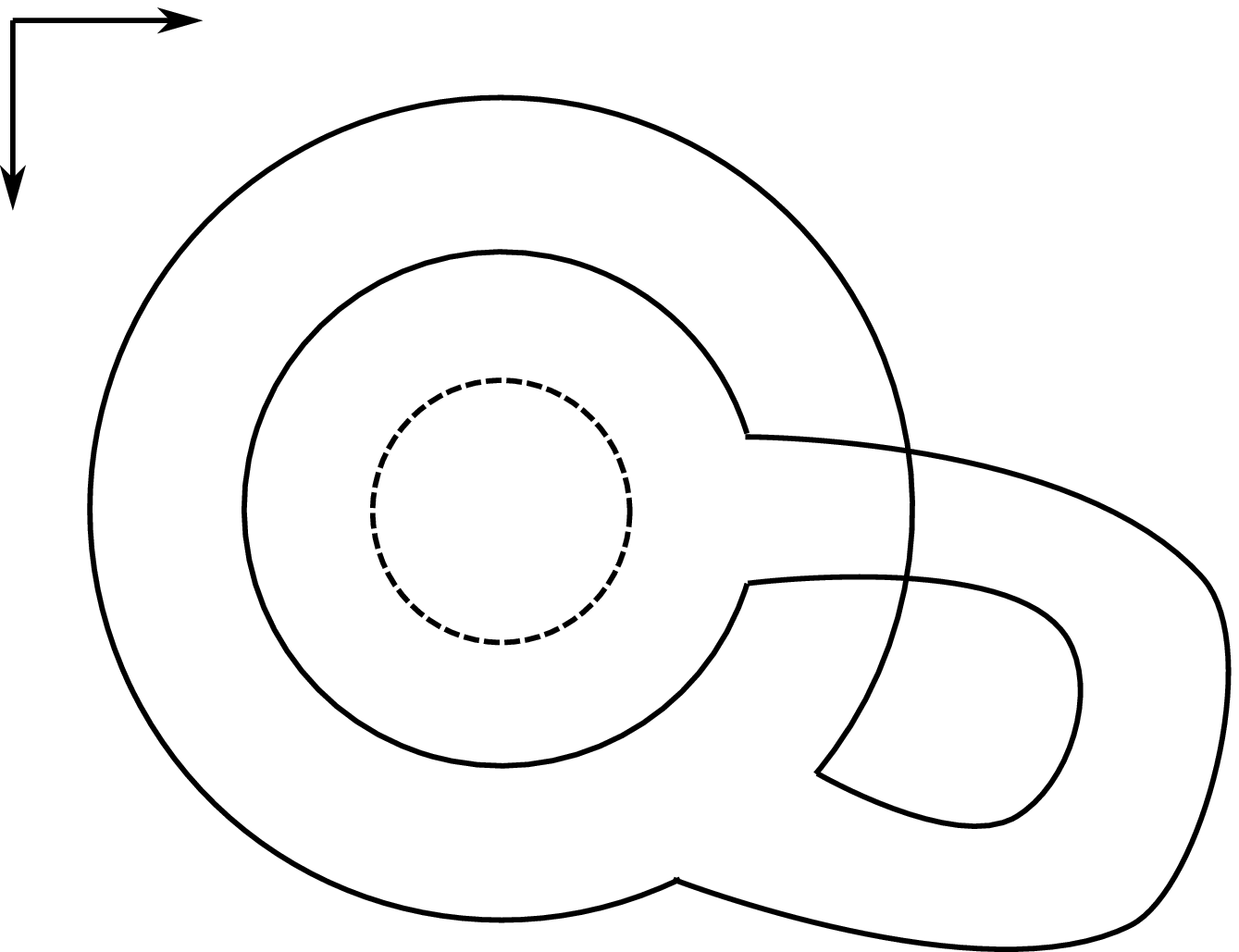}} 
\caption{Example of a Legendrian surface in $J^1(\R^2)$ and its base projection.}
\label{fig:2DExample}
\end{figure}

\subsection{Definition of the cellular DGA $(\A, \partial)$}  \label{sec:CellularDef}
Assume that $\mathcal{E}$ is a compatible polygonal decomposition for a  Legendrian $L$  whose front projection is without swallowtail points.  We now associate a differential graded algebra, $(\mathcal{A},\dd)$, to $L$ using $\mathcal{E}$.  We will refer to $(\mathcal{A}, \dd)$ as the {\bf cellular DGA} of $L$, and in Section \ref{ssec:defsw} we extend the definition to allow swallowtail points.  

The definition of $(\mathcal{A},\dd)$ requires the following additional data associated to $\mathcal{E}$.  For each $2$-cell, $e^2_\alpha$, we choose an  {\bf initial and terminal vertex} from points of $S^1$ that are mapped to $0$-cells by $c^2_\alpha|_{S^1}$, and label these points $v_0^\alpha$ and $v_1^\alpha$.  We allow that $v_0^\alpha = v_1^\alpha$, but in this case we must also declare  a direction for the path around the circle from $v_0^\alpha$ to $v_1^\alpha$.

\subsection{The algebra $\A$} When referring to {\bf sheets} of $L$ above a cell $e^i_\alpha \in \mathcal{E}$ we mean those components of  $L \cap \pi_x^{-1}(e^i_\alpha)$ that are not contained in a cusp edge. 
 We denote the set of sheets of $L$ above $e^i_\alpha$ as $L(e^i_\alpha) = \{S_p \, | \, {p \in I_\alpha}\}$ where $ I_\alpha$ is an indexing set.
 The  sets $L(e^i_\alpha)$ 
are partially ordered by decreasing $z$-coordinate, and we write  $S_p \prec S_q$  if $z(S_p) > z(S_q)$ above $e^i_\alpha$.  Two sheets are incomparible if and only if  they meet in a crossing arc above $e^i_\alpha$ in the front projection of $L$.  

 The algebra $\mathcal{A}$ is the free unital associative (non-commutative) $\Z/2$-algebra with generating set arising from the cells of $\mathcal{E}$ as follows.  For each cell $e_\alpha^i$ we associate one generator for each pair of sheets $S_p$, $S_q \in L(e^i_\alpha)$ satisfying $S_p \prec S_q$.  We denote these generators as $a^\alpha_{p,q}$, $b^\alpha_{p,q}$, or $c^\alpha_{p,q}$ in the case of a $0$-cell, $1$-cell, or $2$-cell respectively. The superscript $\alpha$ will sometimes be omitted from notation.
 

\subsection{The grading on $\A$}  A choice of Maslov potential, $\mu$, for $L$ allows us to assign a $\Z/m(L)$-grading  to $\mathcal{A}$ as follows (where $m(L)$ is the Maslov number of $L$).  Each of the generators is a homogeneous element with degree given by
\begin{equation} \label{eq:graddef}
|c^\alpha_{p,q}| = \mu(S_p) -\mu(S_q) +1;  \quad |b^\alpha_{p,q}| = \mu(S_p) -\mu(S_q);  \quad \mbox{and  }|a^\alpha_{p,q}| = \mu(S_p) -\mu(S_q) -1.
\end{equation} 
When $L$ is connected the grading is independent of the choice of $\mu$.

\subsection{Algebra in  $\mathit{Mat}(n, \A)$} \label{sec:MatrixAlg}  
In the remainder of the article, we often make computations in the ring, $\mathit{Mat}(n, \A)$, of $n\times n$ matrices with entries in $\mathcal{A}$.  Any linear map $f: \mathcal{A} \rightarrow \mathcal{A}$ extends to a linear map $f: \mathit{Mat}(n, \A) \rightarrow \mathit{Mat}(n,\A)$ by applying $f$ entry-by-entry.  Moreover, if $f$ is an algebra homomorphism (resp. a derivation), then the resulting extension is also an algebra homomorphism (resp. a derivation).  Note that a derivation of $\A$ with $\partial(1)=0$ will annihilate any  matrix of constants in $\mathit{Mat}(n, \Z/2) \subset \mathit{Mat}(n,\A)$, so that for instance we can compute $\partial (P B P^{-1}) =P(\partial B) P^{-1}$ if $P \in GL(n, \Z/2)$.  

In our notation, we will use $E_{i,j}$ to denote 
a matrix with all entries $0$ except for the $(i,j)$-entry which is $1$.

\subsection{Defining the differential}  \label{ssec:diff}  We define the differential, $\partial$, by requiring $\partial( 1) = 0$; specifying values on the generators of $\A$; and then extending $\partial$ as a derivation.  Generators that correspond to cells of dimension $0$, $1$ and $2$ are considered separately in the definition.  


\subsubsection{$0$-cells} 
\label{sssec:0cells}
For a zero cell, $e_\alpha^0$, extend the partial ordering $\prec$ to a total linear order via a bijection $\iota: \{1, \ldots, n\}\rightarrow I_\alpha$, so that the (non-cusp) sheets above $e_\alpha^0$ are labeled as $S_{\iota(1)}, \ldots, S_{\iota(n)}$ with $z(S_{\iota(1)}) \geq \cdots \geq z(S_{\iota (n)})$.  Using this ordering, we arrange the generators corresponding to $e_\alpha^0$ into a strictly upper triangular $n \times n$ matrix $A$ with $(i,j)$-entry given by $a_{\iota(i),\iota(j)}^\alpha$ if $S_{\iota(i)} \prec S_{\iota(j)}$ and $0$ otherwise.
The differential is then defined so that the  matrix equation
\begin{equation} \label{eq:Adef}
\partial A = A^2
\end{equation}
holds with $\partial$ applied entry-by-entry.

\begin{lemma} \label{lem:Awd}
There is a unique way to define $\partial a^\alpha_{p,q}$ so that (\ref{eq:Adef}) holds.  Moreover, $\partial a^\alpha_{p,q}$ is independent of the choice of extension of $\prec$ to a total order, and  $\partial^2 a^\alpha_{p,q}=0$.
\end{lemma}
\begin{proof}
Uniqueness is clear since $\partial A$ contains each $\partial a^\alpha_{p,q}$ as an entry.  That such a definition is possible requires that all non-zero entries of $A^2$ correspond to non-zero entries of $A$.  Using $A = (x_{i,j})$ for the entries of $A$, the $(i,j)$-entry of $A^2$ is $\sum_{k}x_{i,k} x_{k,j}= \sum a_{\iota(i),\iota(k)} a_{\iota(k),\iota(j)}$ where the latter sum is over those $k$ such that $S_{\iota(i)} \prec S_{\iota(k)} \prec S_{\iota(j)}$ and hence vanishes unless $S_{\iota(i)} \prec S_{\iota(j)}$.  It is also clear that the sum depends only on $\prec$ and not the choice of extension to a total order.

To see that $\partial^2 a^\alpha_{p,q}=0$, compute
\[
\partial^2A = \partial(A^2) = (\partial A)A + A(\partial A) = A^3 + A^3 =0.
\]
\end{proof}

\subsubsection{$1$-cells} \label{sec:ddef1cell} For generators corresponding to a $1$-cell, $e_\alpha^1$, extend $\prec$ to a total order via a bijection $\iota: \{1, \ldots, n\}\rightarrow I_\alpha$ and form an $n\times n$ matrix $B = (x_{i,j})$ with $x_{i,j} = \left\{ \begin{array}{cr} b_{\iota(i),\iota(j)}^\alpha & \mbox{if $S_{\iota(i)} \prec S_{\iota(j)}$,} \\ 0 &  \mbox{else} \end{array} \right.$.  The characteristic map for $e_\alpha^1$ allows us to refer to the $0$-cells at the boundary of $e_\alpha^1$ as initial and terminal vertices, and we denote these cells as $e_-$ and $e_+$.  (It may be the case that $e_- = e_+$.)  The values, $\partial b^\alpha_{\iota(i),\iota(j)}$, are determined by the matrix equation
\begin{equation} \label{eq:Bdef}
\partial B = A_+(I+B) + (I+B) A_-
\end{equation}
where $A_\pm$ are $n\times n$ matrices formed from the generators corresponding to $e_\pm$ in a manner that will be described presently.

For convenience of notation, we only define $A_+$ as the definition of $A_-$ is identical.    Each sheet above $e_{+}$ belongs to the closure of a unique sheet of $L(e^1_{\alpha})$.  Since we have already chosen a bijection, $\{1,\ldots,n\} \cong L(e^1_\alpha)$, this produces an order preserving injection $ \kappa: L(e_+) \hookrightarrow \{1, \ldots, n\}$.  Those sheets of $e^1_{\alpha}$ not in the image of $\kappa$ meet in pairs at cusp points above $e_{+}$.  We form $A_+$ so that the $(i,j)$ entry is $a^+_{p,q}$ if $S_p \prec S_q$ and $\kappa(p) =i$, $\kappa(q)=j$; the $(k,k+1)$ entry is $1$ if the $k$ and $k+1$ sheets of $L(e^1_{\alpha})$ meet at a cusp above $e_+$; and all other entries are $0$.  Alternatively, one can think of using the total ordering arising from $\kappa$ to form a matrix out of the $a^+_{p,q}$ and then inserting $2\times2$ blocks of the form $\left[ \begin{array}{cc} 0 & 1 \\ 0 & 0 \end{array}\right]$ along the diagonal for each pair of sheets of $L(e^1_{\alpha})$ that meet at a cusp above $e_+$.  For example, if there are $4$ sheets above $e^1_\alpha$ and the top two meet at a cusp at $e_+$, then $A_+ = \left[ \begin{array}{cccc} 0 & 1 & 0 & 0 \\ 0 & 0 & 0 & 0 \\ 0 & 0 & 0 & a^+_{12} \\ 0& 0& 0 &0  \end{array}\right]$.  

\begin{lemma} \label{lem:Bwd}
There is a unique way to define $\partial b^\alpha_{p,q}$ so that (\ref{eq:Bdef}) hold.  Moreover, $\partial b^\alpha_{p,q}$ is independent of the extension of $\prec$ to a total order, and  $\partial^2 b^\alpha_{p,q}=0$.
\end{lemma}
\begin{proof}
In order to be able to define $\partial b^\alpha_{p,q}$ so that (\ref{eq:Bdef}) will hold, we need to know that for any $0$ entries of $B$, the corresponding entry of the right hand side of (\ref{eq:Bdef}) is $0$.  This is only an issue when $e^1_\alpha$ is in the base projection of a crossing arc as otherwise $B$ is a full strictly  upper triangular matrix, and in general the matrices $A_{\pm}$ (resp. $(I+B)$) are strictly upper triangular (resp. upper triangular).  Assuming the $k$ and $k+1$ sheets cross above $e^1_\alpha$, it suffices to check that $A_\pm$ is strictly upper triangular with the $(k,k+1)$-entry equal to $0$.  This is clear since the sheets of $L(e_{\pm})$ that correspond to the $k$ and $k+1$ sheets of $L(e_{\alpha}^1)$ must also have the same $z$-coordinate above $e_+$.

To check independence of $\partial b^\alpha_{p,q}$ from the choice of total order, we again only need to consider the case of a crossing arc above $e^1_{\alpha}$.  The two choices of total order lead to $B$ matrices $B'$ and $B''$ that are related by conjugation by the permutation matrix $Q$ associated to the transposition of the two sheets that cross.  The corresponding $A_{\pm}$ matrices $A_{\pm}'$ and $A_{\pm}''$ are related in the same manner.  Therefore, the equations are equivalent since
\[
\partial B' = A_+'(I+B') + (I+B') A_-' \,\, \Leftrightarrow Q(\partial B')Q = Q[A_+'(I+B') + (I+B') A_-']Q  \,\, \Leftrightarrow 
\]
\[
\partial B'' = A_+''(I+B'') + (I+B'') A_-''
\]
where we use the observation from Section \ref{sec:MatrixAlg} in the $2$-nd equivalence.

To verify $\partial^2 B=0$, observe that in all cases $\partial A_{\pm} = A_{\pm}^2$ because of (\ref{eq:Adef}) combined with the block nature of the matrix $A_\pm$ and the computation $\partial\left[ \begin{array}{cc} 0 & 1 \\ 0 & 0 \end{array}\right]= 0 = \left[ \begin{array}{cc} 0 & 1 \\ 0 & 0 \end{array}\right]^2$.  We compute
\[
\partial^2(I+ B) = \partial[A_+[I+B] + [I+B] A_-] = (\partial A_+)[I+B] +    A_+[\partial(I+B)] + [\partial(I+B)] A_-+[I+B](\partial A_-) =
\]
\[
 (A^2_+)[I+B] + A_+ [A_+(I+B) + (I+B) A_-] + [A_+(I+B) + (I+B) A_-] A_- + [I+B](A_-)^2=0.
\]
\end{proof}

\subsubsection{$2$-cells} \label{sec:Cdef}  For a $2$-cell, $e^2_\alpha$, the partial ordering of sheets in $L(e^2_\alpha)$ is a total ordering, so we take $\{1,\ldots, n\}$ for the indexing set $I_\alpha$ and label sheets as $S_1, \ldots, S_n$ with $z(S_1)> \ldots> z(S_n)$.  Using this ordering of sheets above $e^2_\alpha$, we identify the sheets above all of the edges  and vertices that appear along the boundary of $e^2_{\alpha}$ with subsets of $\{1,\ldots, n\}$.  Then, for each such edge (resp. vertex) we can place the corresponding generators $b^\alpha_{p,q}$ (resp. $a^\alpha_{p,q}$) into an $n \times n$ matrices with $2\times 2$ blocks of the form $\left[ \begin{array}{cc} 0 & 0 \\ 0 & 0 \end{array} \right]$ (resp.  $\left[ \begin{array}{cc} 0 & 1 \\ 0 & 0 \end{array} \right]$) inserted in the $k, k+1$ position along the diagonal whenever sheets $S_k, S_{k+1} \in L(e^2_\alpha)$ meet at a cusp above the edge (resp. vertex).

 Recall that for each $2$-cell we have chosen initial and terminal vertices, $v^\alpha_0$ and $v^\alpha_1$, along the boundary of $D^2$ that are mapped to $0$-cells of $\mathcal{E}$ via the characteristic map, $c^2_\alpha$, for the $2$-cell.  Let $\gamma_+$ and $\gamma_-$ denote paths in $S^1$ that proceed counter-clockwise and clockwise respectively from $v^\alpha_0$ to $v^\alpha_1$.  (If $v^\alpha_0=v^\alpha_1$, then we choose one of these paths to be constant and the other to be the entire circle as mentioned in Section \ref{sec:CellularDef}.) We let $B_1, \ldots, B_{j}$ (resp. $B_{j+1}, \ldots, B_{m}$) denote the matrices associated (as in the previous paragraph) to the successive edges of $\mathcal{E}$ that appear in the image of $c^2_\alpha \circ \gamma_+$ (resp. $c^2_\alpha \circ \gamma_-$).  In addition, we let $A_{v_0}$ and $A_{v_1}$ be the matrices associated (as in the previous paragraph) to the initial and terminal vertices $v_0^\alpha$ and $v_1^\alpha$.  
Collecting the generators corresponding to $e^2_\alpha$ into the strictly upper triangular matrix $C$, we define $\partial c^\alpha_{i,j}$ so that
\begin{equation} \label{eq:Cdef}
\partial C = A_{v_1}C + C A_{v_0} + (I+B_{j})^{\eta_j} \cdots (I+B_1)^{\eta_1} + (I+B_{m})^{\eta_m} \cdots (I+B_{j+1})^{\eta_{j+1}}
\end{equation}
where the exponent, $\eta_i$, is $1$ (resp. $-1$) if the parametrization of the edge corresponding to $B_i$ given by $c^2_\alpha\circ \gamma_{\pm}$ has the same (resp. opposite) orientation as the characteristic map of the $1$-cell.  See Figure \ref{fig:2cell}.


\begin{figure}

\quad

\begin{tabular}{ccc}
$0$-cell:   & \labellist
\small
\pinlabel $e^\alpha_0$ [t] at 4 -4
\pinlabel $A=(a^\alpha_{i,j})$ [bl] at 2 12
\endlabellist
\includegraphics[scale=.5]{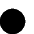}  &  $\partial A = A^2$ \\
 & & \\
 & & \\
  & & \\
  & & \\ 
$1$-cell:  & \labellist
\small
\pinlabel $A_-$ [b] at 4 14
\pinlabel $A_+$ [b] at 116 14
\pinlabel $B$ [b] at 58 18
\endlabellist
\includegraphics[scale=.5]{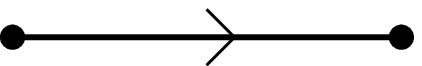}   &  $\partial B = A_+ (I+B) + (I+B) A_-$ \\
 & & \\
 & & \\
$2$-cell: \quad \quad & \labellist
\small
\pinlabel $B_2$ [l] at 156 40
\pinlabel $B_1$ [t] at 52 0
\pinlabel $B_3$ [l] at 198 133
\pinlabel $A_{v_0}$ [tr] at -3 3
\pinlabel $B_5$ [r] at 118 168
\pinlabel $B_4$ [r] at 20 76
\pinlabel $A_{v_1}$ [bl] at 202 184
\pinlabel $C$ [bl] at 102 84
\endlabellist
\raisebox{-1.5cm}{\includegraphics[scale=.5]{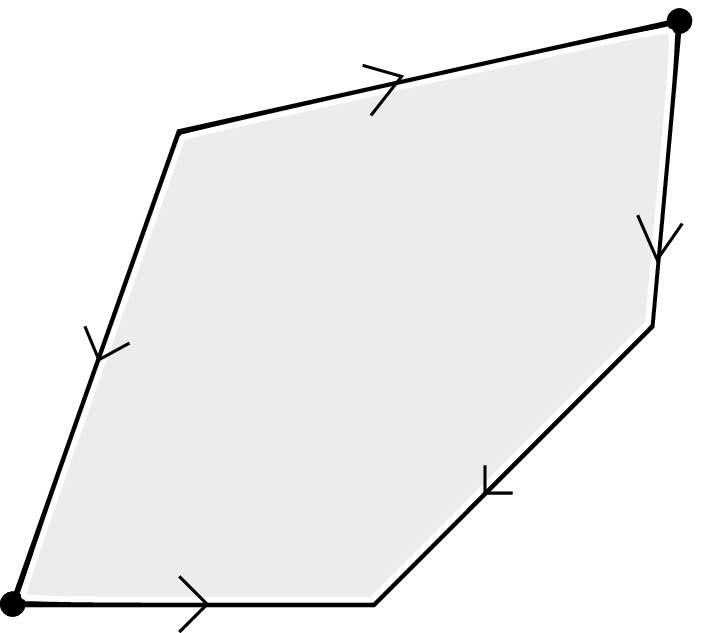}} 
 \quad \quad \quad & $\begin{array}{rc} \partial C = &  A_{v_1}C + CA_{v_0}+ \\ & \\ & (I+B_3)^{-1}(I+B_2)^{-1}(I+B_1) + \\ & \\ & (I+B_5)(I+B_4)^{-1} \end{array}$
\end{tabular}
\medskip
\caption{Summary of differentials in $(\A,\partial)$.}
\label{fig:2cell}
\end{figure}


\begin{lemma}  \label{lem:cwd}
We can uniquely define $\partial c^\alpha_{i,j}$ so that (\ref{eq:Cdef}) holds.  Moreover, $\dd^2 C = 0$.
\end{lemma}
\begin{proof}
First, observe that since the matrices $B_{i}$ are strictly uppertriangular, and hence nilpotent in $\mathit{Mat}(n, \A)$, the inverse of $(I+B_i)$ is given by the finite geometric series
\[
(I+B_i)^{-1} = I+B_i+ B_i^2 + \ldots.
\]
That the equation (\ref{eq:Cdef}) may be obtained relies on the right hand side being strictly upper triangular.  This follows since the first two terms are strictly upper triangular and the $3$-rd and $4$-th terms both have the form $I + X$ where $X$ is strictly upper triangular.

Next, each $B_i$ corresponds to a $1$-cell, $e_i^1$, in $\mathcal{E}$.  Moreover, $e_i^1$ has initial and terminal vertices (using the orientation of $e_i^1$ to distinguish initial from terminal).   Denote by $A_i^+$ and $A_i^-$ the $n\times n$ matrices associated to these vertices using the total ordering and insertion of $2\times2$ blocks specified by the $2$-cell $e^2_\alpha$.  We claim that
\begin{equation} \label{eq:dbi}
\partial B_i = A_i^+(I+B_i) + (I+B_i) A_i^-.
\end{equation}
Indeed, as long as there are not two sheets of $L(e^2_\alpha)$ that meet at a cusp above $e_i^1$, this is just (\ref{eq:Bdef}) with matrices formed using the total order from $L(e^2_\alpha)$.  If it does happen that $S_k,S_{k+1} \in L(e^2_\alpha)$ share a cusp edge in their closure above $e_i^1$, then the matrix $B_i$, (resp. $A_i^\pm$) is obtained from the corresponding matrix in (\ref{eq:Bdef}) by forming a block matrix with $\left[ \begin{array}{cc} 0 & 0 \\ 0 & 0 \end{array} \right]$  (resp. $N= \left[ \begin{array}{cc} 0 & 1 \\ 0 & 0 \end{array} \right]$) inserted at the $k,k+1$ part of the diagonal of $B_i$ (resp. $A_i^\pm$).  Then, using block matrix calculations, (\ref{eq:dbi}) follows from (\ref{eq:Bdef}) and the calculation for $2\times 2$ matrices
\[
\partial \, 0 = N(I +0) + (I+0)N.
\]

From (\ref{eq:dbi}), we can deduce that
\begin{equation} \label{eq:dbinv}
\partial[(I+B_i)^{-1}] = A^-_i(I+B_i)^{-1} + (I+B_i)^{-1}A^+_i
\end{equation}
by applying the Liebniz rule to $0 = \partial(I) = \partial[ (I+B)(I+B)^{-1}]$ and solving for $\partial[ (I+B_i)^{-1}]$.

This gives us that $\partial (I+B_i)^{\eta_i} = X_i^+ (I+B_i)^{\eta_i} + (I+B_i)^{\eta_i}X_i^+$ where the $X_i^\pm$ correspond to the initial and terminal vertex of $e_i^1$ with respect to the orientation given by $\gamma_+$ or $\gamma_-$ rather than by the orientation of $e_i^1$ itself.  That is, $X_i^\pm = A^\pm_i$ if $\eta_i = 1$, and $X_i^\pm = A^\mp_i$ if $\eta_i =-1$.  
Since $X_i^+ = X_{i+1}^-$ for $1\leq i \leq j-1$ and $j+1 \leq i \leq m$, the sums arising from expanding $\partial\left[(I+B_j)^{\eta_j} \cdots (I+B_1)^{\eta_1}\right]$ and $\partial\left[(I+B_{m})^{\eta_{m}} \cdots (I+B_{j+1})^{\eta_{j+1}}\right]$ using the Liebniz rule telescope to give
\[
\partial\left[(I+B_j)^{\eta_j} \cdots (I+B_1)^{\eta_1}\right] = X^+_j(I+B_j)^{\eta_j} \cdots (I+B_1)^{\eta_1} + (I+B_j)^{\eta_j} \cdots (I+B_1)^{\eta_1}X^-_1=
\]
\[
 A_{v_1}(I+B_j)^{\eta_j} \cdots (I+B_1)^{\eta_1} + (I+B_j)^{\eta_j} \cdots (I+B_1)^{\eta_1}A_{v_0},
\]
and
\[
\partial\left[(I+B_{m})^{\eta_{m}} \cdots (I+B_{j+1})^{\eta_{j+1}}\right] = A_{v_1}(I+B_{m})^{\eta_m} \cdots (I+B_{j+1})^{\eta_{j+1}} + (I+B_{m})^{\eta_m} \cdots (I+B_{j+1})^{\eta_{j+1}}A_{v_0}.
\] 
With these formulas in hand, $\partial^2C = 0$ follows in a straightforward manner from (\ref{eq:Cdef}).
\end{proof}


Explicit examples computing this differential appear in Section \ref{sec:Examples}.

\subsection{Extending the definition to allow swallowtail points}  \label{ssec:defsw}

For a Legendrian $L$ with swallowtail points and compatible polygonal decomposition $\mathcal{E}$, the DGA $(\A,\partial)$ is defined in the same general manner.  
For a cell whose closure is disjoint from the base projections of swallowtail points, generators and differentials are defined as in Section \ref{sec:CellularDef}. 
The appropriate modifications of the definition for $0$-cells, $1$-cells, and $2$-cells that border swallow tail points are given below in \ref{sssec:sw0cell}, \ref{sec:BdefST}, and \ref{sssec:sw2cell}.

\subsection{Local appearance near a swallowtail point} Suppose that $e^0_\alpha$ is the base projection of a swallowtail point of $L$.  The singular locus $\Sigma \subset S$ near $e^0_\alpha$ consists of $3$ arcs that share a common end point at $e^0_\alpha$.  Two of these arcs are projections of cusp edges that meet in a semi-cubical cusp point at $e^0_\alpha$, and the third is a crossing arc that lies between the cusp arcs.  See Figure \ref{fig:LocalSwallow}.  In a neighborhood of $e^0_\alpha$, we refer to the region between the two cusp arcs (right of the cusp point in Figure \ref{fig:LocalSwallow}) as the {\bf swallowtail region}.  Note that there are $2$ more sheets above the swallowtail region  than there are above
its complement.

\begin{figure}
\centerline{\includegraphics[scale=.5]{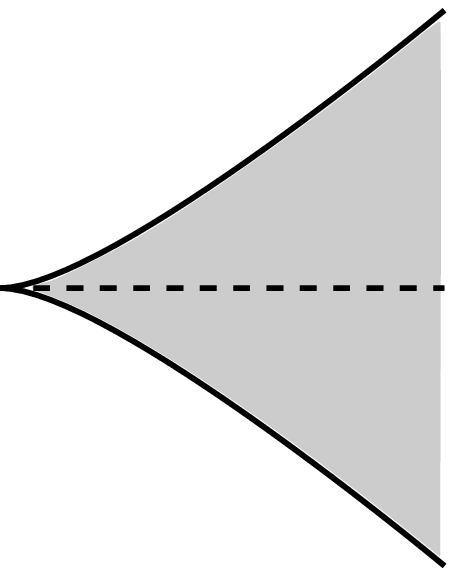}  \quad \quad \quad \quad
\labellist
\small
\pinlabel $S$ [c] at 58 92
\pinlabel $T$ [c] at 58 66
\pinlabel $T$ [c] at 266 92
\pinlabel $S$ [c] at 266 66
\endlabellist
\includegraphics[scale=.5]{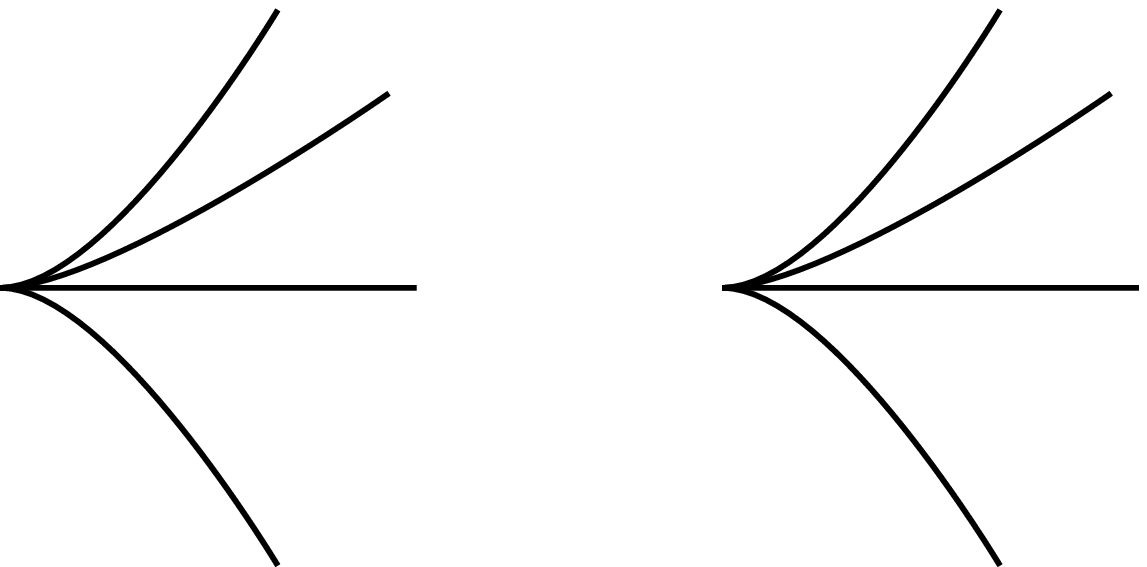}} 
\caption{(left) The base projection of the singular set at a swallowtail point.  Solid curves denote cusp arcs, and the dotted curve is a crossing arc.  The swallowtail region is shaded.  (right)  An example of an $L$-compatible cellular decomposition near a swallowtail point with the two possible labelings by $S$ and $T.$ Here one of the $1$-cells is not on the singular set.}
\label{fig:LocalSwallow}
\end{figure}

Recall that the appearance of the front projection of $L$ near a swallowtail point $s_0$ 
has one of  two distinct types depending on if the $s_0$ is an \emph{upward} or \emph{downward} swallow tail. 
We use the convention of indexing the $3$ sheets that meet at an upward (resp. downward) swallowtail point as $k$, $k+1$, $k+2$ (resp. $l-2$, $l-1$, $l$), so that $k$ (resp. $l$) denotes the upper (resp. lower) sheet and the $k+1$ and $k+2$ (resp. $l-2$ and $l-1$) sheets cross near  $s_0$.  The single sheet outside the swallowtail region that contains $s_0$ in its closure is then labeled $k$ (resp. $l-2$).


\subsection{Decorations at swallowtail points} \label{sec:decor} Defining $(\A,\partial)$ in the presence of swallowtail points requires some additional data.  Within a neighborhood of each swallowtail point, $s \in \Sigma_2$, we choose a labeling of the two regions (subsets of $2$-cells of $\mathcal{E}$) that border the crossing arc that ends at $s$;  label one of the regions with an $S$ and the other with a $T$.  See Figure \ref{fig:LocalSwallow}.


\subsection{$0$-cells} \label{sssec:sw0cell}  In case the $0$-cell, $e^0_\alpha$, is the projection of a swallowtail point the definition of the algebra does not require serious modification.  Just note that we do consider the swallowtail point itself to be a sheet of $L(e^0_\alpha)$ so that there are the same number of  generators $a^\alpha_{i,j}$ as if $e^0_\alpha$ were located in the complement of the swallowtail region.  The differential is defined by (\ref{eq:Adef}) as usual.

\subsection{Matrices associated to a swallowtail point}  \label{sec:matrices} For each swallow tail point, $s \in \Sigma_2$,  we will use several matrices formed from the generators, $a_{i,j}$, associated to $s$ (which is necessarily a $0$-cell of $\mathcal{E}$).  Near $s$, we suppose that there are $n$ sheets above the swallow tail region, and $n-2$ in the complement of the swallow tail region. 
This is partly inspired by the swallowtail discussion of Cerf theory in \cite{HatcherWagoner73}. 

The $n-2$ sheets above $s$ itself are totally ordered, and we let $A$ denote the $(n-2) \times (n-2)$ matrix with non-zero entries $a_{i,j}$.  For $1 \leq m_1 < m_2 \leq n$, we let $\widehat{A}_{m_1,m_2}$ denote the $n \times n$  matrix obtained from $A$ by inserting the block  $N= \left[ \begin{array}{cc} 0 & 1 \\ 0 & 0 \end{array} \right]$ along the main diagonal at the (possibly non-consecutive) $m_1$ and $m_2$ rows and columns with the rest of the entries in these rows and columns equal to $0$.  

In the case of an {\it upward} swallowtail point that involves sheets $k,k+1,$ and $k+2$, we use the following matrix notations:
\begin{equation}
\begin{array}{ll}
A_S = [I + E_{k+2,k+1}] \widehat{A}_{k,k+2} [I + E_{k+2,k+1}];  &  A_T = [I+E_{k+1,k+2}]\widehat{A}_{k,k+1} [I+E_{k+1,k+2}];  \\ & \\
S = I + \widehat{A}_{k,k+1} E_{k+2,k} + E_{k+1,k+2} &  T = I +E_{k+1,k+2}. \\
 \,\,\, = I + \sum_{i<k} a_{i,k} E_{i,k} + E_{k+1,k+2}; &
\end{array}
\end{equation}

In the case of a {\it downward} swallowtail point that involves sheets $l-2,l-1,$ and $l$, we use
\begin{equation}  \label{eq:Atdef}
\begin{array}{ll}
 A_S = [I+E_{l-1,l-2}]\widehat{A}_{l-2,l} [I+E_{l-1,l-2}];  &  A_T = [I+E_{l-2,l-1}]\widehat{A}_{l-1,l} [I+E_{l-2,l-1}];  \\
 &
\\
 S = I +  E_{l,l-2} \widehat{A}_{l-1,l} + E_{l-2,l-1} & 
T= I+ E_{l-2,l-1}. \\
\,\,\,=I + \sum_{l-2<j} a_{l-2,j} E_{l,j+2} + E_{l-2,l-1};
\end{array}
\end{equation}

For instance, if there are $4$-sheets above the swallow tail region of an upward swallow tail, $s$, that involves sheets $2,3$, and $4$, then 
\[A = \left[ \begin{array}{cc}0 & a_{12} \\ 0 & 0 \end{array} \right];  \quad \widehat{A}_{k,k+1}= \left[ \begin{array}{cccc}0 & 0& 0 & a_{12} \\ 0 & 0 & 1 & 0 \\ 0 & 0 & 0 & 0 \\ 0 & 0 & 0 & 0 \end{array} \right]; \quad \widehat{A}_{k,k+2}= \left[ \begin{array}{cccc}0 & 0& a_{12} & 0 \\ 0 & 0 & 0 & 1 \\ 0 & 0 & 0 & 0 \\ 0 & 0 & 0 & 0 \end{array} \right];
\] 
\[A_{S}= \left[ \begin{array}{cccc}0 & 0& a_{12} & 0 \\ 0 & 0 & 1 & 1 \\ 0 & 0 & 0 & 0 \\ 0 & 0 & 0 & 0 \end{array} \right]; \quad A_{T}= \left[ \begin{array}{cccc}0 & 0& 0 & a_{12} \\ 0 & 0 & 1 & 1 \\ 0 & 0 & 0 & 0 \\ 0 & 0 & 0 & 0 \end{array} \right];
\] 
\[S= \left[ \begin{array}{cccc}1 & a_{12} & 0 & 0 \\ 0 & 1 & 0 & 0 \\ 0 & 0 & 1 & 1 \\ 0 & 0 & 0 & 1 \end{array} \right]; \quad T= \left[ \begin{array}{cccc}1 & 0& 0 & 0 \\ 0 & 1 & 0 & 0 \\ 0 & 0 & 1& 1 \\ 0 & 0 & 0 & 1 \end{array} \right].
\] 

\begin{remark}
These somewhat mysterious matrices may be interpreted in a nice way using generating families.  When $L$ is defined by a generating family, $F:S \times \R^N \rightarrow \R$, the sheets of $L$ above $x \in S$ correspond to the critical points of $F(x,\cdot) : \R^N \rightarrow \R$.  Suppose that the entries, $a_{i,j}$, of $A$  are replaced with the mod $2$ count of gradient trajectories connecting sheet $S_i$ to $S_j$ near the swallow tail point, but outside of the swallowtail region.  Then, the matrices $S$ and $T$ are continuation maps determined by the collection of handleslides that end at the swallowpoint as specified in \cite{HatcherWagoner73}.  The matrices $A_{S}$ and $A_{T}$ represent the differential in the Morse complex of $F_x$ above a point $x$ belonging to the crossing arc that ends at the swallowtail point, with respect to the ordering of basis vectors (i.e. sheets of $L$) as they appear above the regions decorated with $S$ and $T$.   See \cite{RuSu3} for more details.     
\end{remark}

\begin{lemma} \label{lem:ASAT} We have 
\[
\partial A_S = (A_S)^2,  \quad \partial A_T = (A_T)^2.
\]
Moreover, for an upward swallowtail
\begin{equation} \label{eq:partialS}
\partial S = \widehat{A}_{k,k+1} S + S A_S= S\widehat{A}_{k,k+1}  + A_S S, \quad \mbox{and} \quad \partial T = \widehat{A}_{k,k+1} T + T A_T;
\end{equation}
for a downward swallowtail
\begin{equation}  \label{eq:partialS2}
\partial S = \widehat{A}_{l-1,l} S + S A_S= S\widehat{A}_{l-1,l}  +  A_S S, \quad \mbox{and} \quad \partial T = \widehat{A}_{l-1,l} T + T A_T.
\end{equation}
\end{lemma}

\begin{proof}
We verify these formulas in the case of an upward swallowtail.  The case of a downward swallowtail is similar.  

Compute
\begin{align*}
\partial A_S = &  [I + E_{k+2,k+1}] \partial \widehat{A}_{k,k+2} [I + E_{k+2,k+1}] = [I + E_{k+2,k+1}] (\widehat{A}_{k,k+2})^2 [I + E_{k+2,k+1}] =
\\ &  [I + E_{k+2,k+1}] (\widehat{A}_{k,k+2}) [I + E_{k+2,k+1}] [I + E_{k+2,k+1}] (\widehat{A}_{k,k+2}) [I + E_{k+2,k+1}] = A_S^2
\end{align*}
where the second equality follows since $\partial A = A^2$ and $\partial N = N^2$.  That $\partial A_T = (A_T)^2$ is similar.

For the second equation of (\ref{eq:partialS}), we compute from definitions
\[
\widehat{A}_{k,k+1} T + T A_T = \widehat{A}_{k,k+1} T + T(T \widehat{A}_{k,k+1} T) = 0 = \partial T. 
\] 

To establish the first equation of (\ref{eq:partialS}), begin by observing that, for $Q$ the permutation matrix of the transposition $(k+1,k+2)$, we have 
\begin{equation} \label{eq:SQEk}
S= [Q+E_{k+2,k+2}] F  \mbox{   where $F= I + E_{k+2,k+1} + \widehat{A}_{k,k+1}E_{k+2,k}$.}
\end{equation}
Indeed,  $[Q+E_{k+2,k+2}]\widehat{A}_{k,k+1} E_{k+2,k} =\widehat{A}_{k,k+1} E_{k+2,k}$ because all entries below the $k-1$ row of  $AE_{k+2,k}$ are $0$, and 
$[Q+E_{k+2,k+2}][I + E_{k+2,k+1}] = I+E_{k+1,k+2}$.
Next, note that
\begin{equation} \label{eq:WAIE}
F \widehat{A}_{k,k+1} = [I + E_{k+2,k+1} + \widehat{A}_{k,k+1} E_{k+2,k}] \widehat{A}_{k,k+1} = \widehat{A}_{k,k+1}+ 0 + \widehat{A}_{k,k+1}E_{k+2,k+1} = \widehat{A}_{k,k+1}F +\widehat{A}_{k,k+1}^2E_{k+2,k} 
\end{equation} 
where the second equality is due to the $N$ block in the $k,k+1$ rows of $\widehat{A}_{k,k+1}$. 
Finally, using (\ref{eq:SQEk}) and (\ref{eq:WAIE}) compute
\[
S\widehat{A}_{k,k+1} = [Q+E_{k+2,k+2}]F \widehat{A}_{k,k+1} = [Q+E_{k+2,k+2}] \widehat{A}_{k,k+1} F + [Q+E_{k+2,k+2}] \widehat{A}_{k,k+1}^2E_{k+2,k}= 
\]
\[[I+E_{k+2,k+1}]Q\widehat{A}_{k,k+1}\left(Q[I+E_{k+2,k+1}][I+E_{k+2,k+1}]Q\right)F +\partial S
=
\]
\[
[I+E_{k+2,k+1}]Q\widehat{A}_{k,k+1}Q[I+E_{k+2,k+1}] S +\partial S = A_S S + \partial S.
\]
(In the $3$-rd equality we used that $Q$ and $I+E_{k+2,k+1}$ are both self inverse.)

This gives $\partial S = S\widehat{A}_{k,k+1} + A_S S$.  Since $S^2=I$, multiplying on left and right by $S$ gives $S (\partial S) S = \widehat{A}_{k,k+1} S + S A_S$.  One checks that $S(\partial S) S = \partial S$, so this completes the proof.

\end{proof}



\subsection{$1$-cells}  \label{sec:BdefST} Suppose that a $1$-cell $e^1_\alpha$ has its initial  vertex, $e^0_-$, or terminal vertex, $e^0_+$, at a swallowtail point.  The generators,  $b^\alpha_{p,q}$, associated to $e^1_\alpha$ arise from the partially ordered set of sheets $L(e^1_\alpha)$ as usual, and the differential is again characterized by (\ref{eq:Bdef}).   However, some adjustment may be required in forming the matrices $A_\pm$.  To simplify notation, we define $A_-$ in case  $e^0_-$ is a swallowtail point, $s$, and we note that an identical definition applies for $A_+$ if $e^0_+$ is a swallowtail point.

\begin{itemize}
\item Suppose $e^1_\alpha$ is outside the interior of the swallowtail region near $e^0_-$, including if $e^1_\alpha$ lies in either of the cusp arcs ending at $e^0_-$.  We then form the matrices $B$ and  $A_-$ as in \ref{sec:ddef1cell}.  (No issues arise as the sheets of $L(e^1_\alpha)$ are then totally ordered by $\prec$, and the swallow tail point in $L(e^0_-)$ is in the closure of a unique sheet of $L(e^1_\alpha)$.)

\item Suppose $e^1_\alpha$ is in the interior of the swallowtail region, and $e^1_\alpha$ is not contained in the crossing arc that terminates at $e^0_-$.  Then, $L(e^1_\alpha)$ is totally ordered, and we form $B$ as in \ref{sec:ddef1cell}.  In addition, we set 
\[
\mbox{$A_- = \widehat{A}_{k,k+1}$   (resp. $A_- = \widehat{A}_{l-1,l}$)}
\] if $e^0_-$ is an upward (resp. downward) swallow tail point.  (That is, in the upward (resp. downward) case we form $A_-$ as if the upper (resp. lower) $2$ of the $3$ swallowtail sheets above $e^1_\alpha$ meet at a cusp above $e^0_-$.) 

\item Suppose $e^1_{\alpha}$ is contained in the crossing arc that ends at the swallowtail point above $e^0_{-}$.  In forming the matrix, $B$, there are two possible choices for extending the ordering of $L(e^1_{\alpha})$ to a total order.  If the total ordering used agrees with the ordering of sheets on the side of the swallow tail decorated with an $S$, then we set 
\begin{equation} \label{eq:st1} 
A_- = A_S.
\end{equation}
If instead the ordering agrees with the side of the swallow tail decorated with a $T$, then
\begin{equation} \label{eq:st2}
A_- =  A_T.
\end{equation}   
\end{itemize}

\begin{lemma}  With $B, A_-,$ and $A_+$ formed in this manner, equation (\ref{eq:Bdef}) leads to a well defined definition of $\partial b^\alpha_{p,q}$, and $\partial^2b^\alpha_{p,q}=0$.
\end{lemma}
\begin{proof}
First, we verify well-definedness.  The total ordering of $L(e^1_\alpha)$ is uniquely determined except in the case where $e^1_\alpha$ is contained in a crossing locus.  Suppose that this is the case and that $e_-^0$ is a swallowtail point.  (When $e_+^0$ is a swallowtail point the same argument applies.)  Changing the choice of total ordering conjugates $B$ by the permutation matrix, $Q$, of the transposition $(k+1 \,\, k+2).$
The corresponding definitions of $A_-$ from (\ref{eq:st1}) and (\ref{eq:st2}) are related in the same manner since
\[
QA_TQ = Q[I+E_{k+1,k+2}]\widehat{A}_{k,k+1}[I+E_{k+1,k+2}]Q = Q[I+E_{k+1,k+2}]Q(Q\widehat{A}_{k,k+1}Q)Q[I+E_{k+1,k+2}]Q =
\]
\[
 [I+E_{k+2,k+1}]\widehat{A}_{k,k+2} [I+E_{k+2,k+1}]= A_S.
\]
Thus,
 as in the proof of Lemma \ref{lem:Bwd}, the two versions of (\ref{eq:Bdef}) that arise from different choices of ordering of $L(e^1_\alpha)$ are equivalent.

Since the $k+1$ and $k+2$ sheets cross above $e^1_\alpha$, it is also necessary to check that the $(k+1,k+2)$ entry of the right hand side of (\ref{eq:Bdef}) is $0$. Here, it is enough to check that $A_-$ is strictly upper triangular with the $(k+1,k+2)$-entry equal to $0$.  (This implies that $(I+B)A_-$ has the desired property, and a similar argument or the argument of Lemma \ref{lem:Bwd} applies to the $A_+(I+B)$ term.)  Since either ordering produces an equivalent equation, we can assume that the total ordering for sheets of $L(e^1_\alpha)$ coincides with the order of sheets in the $T$ region.  Then, we can compute the $3\times3$ diagonal block at rows $k,k+1,k+2$ of $A_- = [I+E_{k+1,k+2}]\widehat{A}_{-,k,k+1}[I+E_{k+1,k+2}]$ to be
\[
\left[ \begin{array}{ccc} 1& 0 & 0 \\ 0 & 1 & 1 \\ 0 & 0 & 1 \end{array}\right]\,\left[ \begin{array}{ccc} 0& 1 & 0 \\ 0 & 0 & 0 \\ 0 & 0 & 0 \end{array}\right]\, \left[ \begin{array}{ccc} 1& 0 & 0 \\ 0 & 1 & 1 \\ 0 & 0 & 1 \end{array}\right] = \left[ \begin{array}{ccc} 0& 1 & 1 \\ 0 & 0 & \mathbf{0} \\ 0 & 0 & 0 \end{array}\right].
\]

The computation $\partial^2 B = 0$ goes through as before since, according to Lemma \ref{lem:ASAT}, it is still the case that $\dd A_\pm = A_\pm^2$.
\end{proof}

\subsection{$2$-cells}  \label{sssec:sw2cell}  For a $2$-cell, $e^2_\alpha$, the boundary of the domain of the characteristic map $c^2_\alpha: D^2 \rightarrow \overline{e^2_\alpha}$ can be viewed as a polygon, $P$, with edges mapping to $1$-cells and vertices mapping to $0$-cells of $\mathcal{E}$.  Some vertices may have their sufficiently small  neighborhoods in $D^2$ mapping to a region near a swallowtail point that is labeled with an $S$ or $T$.  We indicate this pictorially by placing an $S$ or $T$ near the corresponding corner of $P$, and we form a new polygon $Q$ by adding an extra edge that cuts out each such corner.  Thus, the edges of $Q$ each correspond to either a $1$-cell of $\mathcal{E}$ or an $S$ or $T$ region at a swallow tail point.  Those edges corresponding to $1$-cells receive orientations according to the orientation of the $1$-cells (provided by characteristic maps).  We orient the new $S$ and $T$ edges so that the orientation points away from the endpoint shared with the crossing arc that terminates at the swallow tail point.   See Figure \ref{fig:Ex2Base}.

\begin{figure}
\centerline{
\labellist
\small
\pinlabel $x_1$ [r] at -4 113
\pinlabel $x_2$ [b] at 60 142
\pinlabel $z$ [l] at 21 185
\endlabellist
\includegraphics[scale=.6]{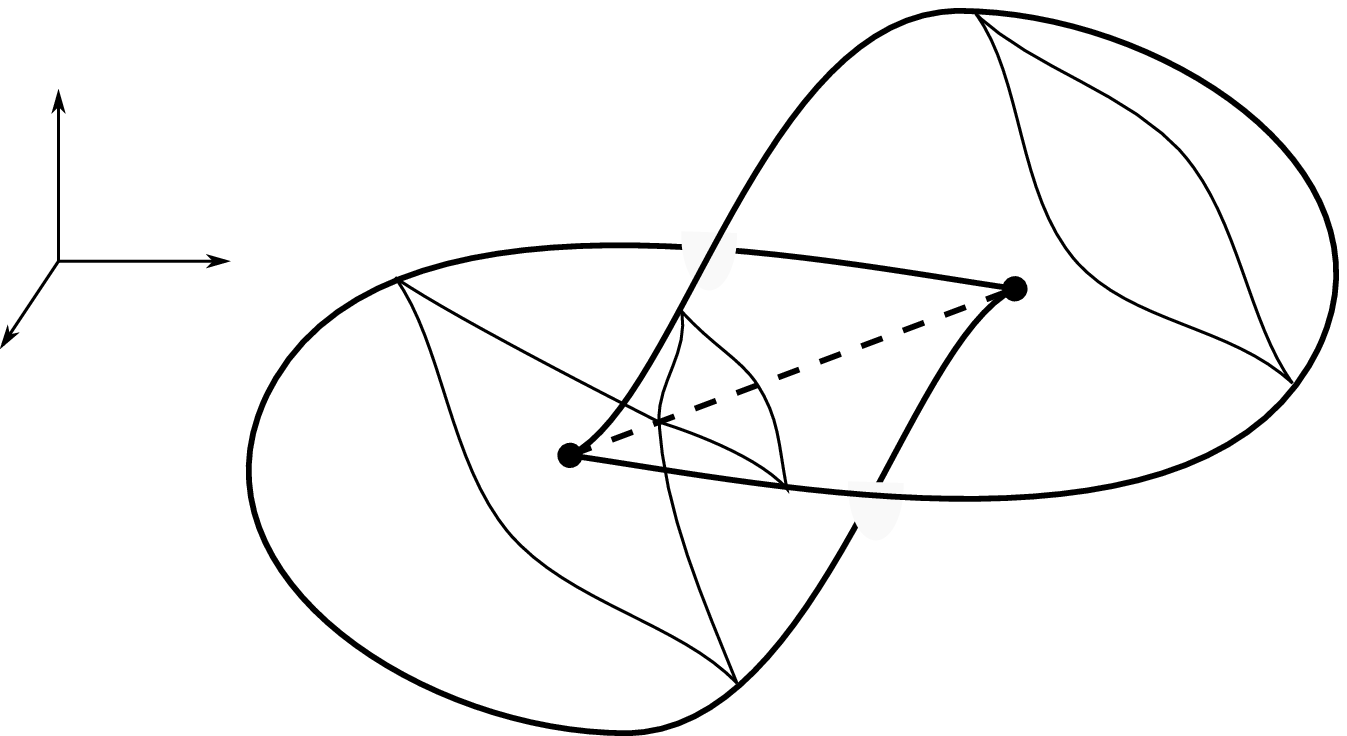}} 

\medskip

\medskip

\centerline{
\labellist
\small
\pinlabel $T_1$ [t] at 112 82
\pinlabel $S_1$ [b] at 112 88
\pinlabel $T_2$ [t] at 206 82
\pinlabel $S_2$ [b] at 206 88
\pinlabel $C_3$ [c] at 156 122
\pinlabel $C_4$ [c] at 156 47
\pinlabel $C_1$ [c] at 49 128
\pinlabel $C_2$ [c] at 253 128
\pinlabel $A_1$ [r] at 66 84
\pinlabel $A_2$ [l] at 248 84
\pinlabel $B_1$ [tr] at 108 56
\pinlabel $B_2$ [br] at 106 113
\pinlabel $B_3$ [tl] at 209 58
\pinlabel $B_4$ [bl] at 189 129
\pinlabel $B_5$ [b] at 170 86
\endlabellist
\includegraphics[scale=.8]{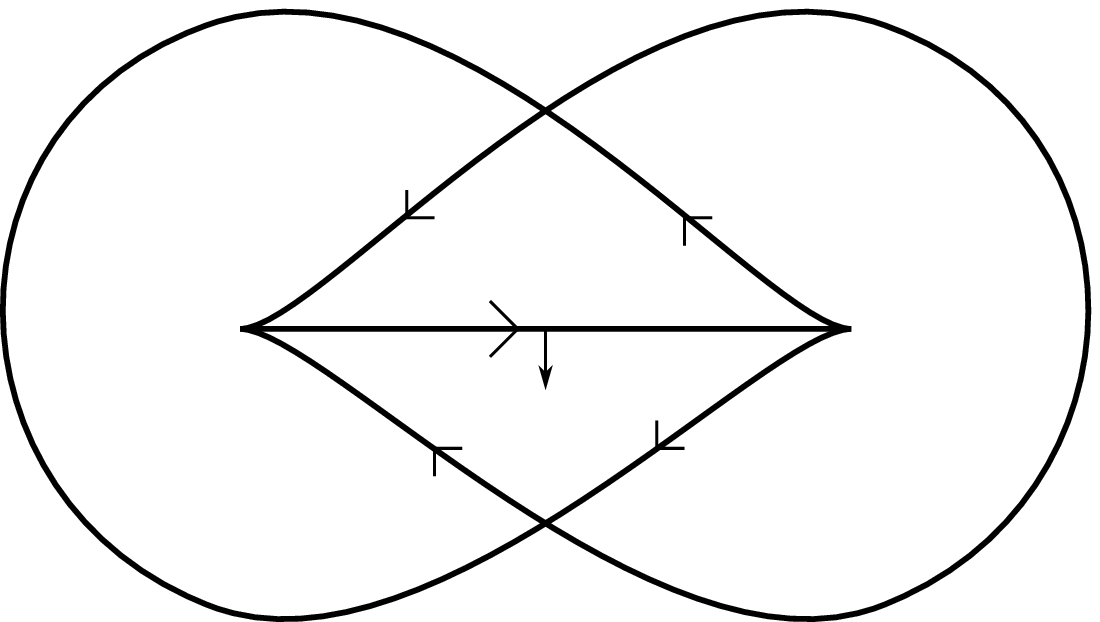}} 

\medskip

\medskip

\centerline{
\labellist
\small
\pinlabel $S_1$ [b] at 30 12
\pinlabel $S_2$ [b] at 134 12
\pinlabel $B_2$ [br] at 50 86
\pinlabel $B_4$ [bl] at 112 86
\pinlabel $B_5$ [t] at 96 6
\pinlabel $S_1$ [tr] at 232 40
\pinlabel $S_2$ [tl] at 378 40
\pinlabel $B_2$ [br] at 270 102
\pinlabel $B_4$ [bl] at 338 102
\pinlabel $B_5$ [t] at 318 6
\pinlabel $P$  at 81 54
\pinlabel $Q$  at 308 54
\endlabellist
\includegraphics[scale=.6]{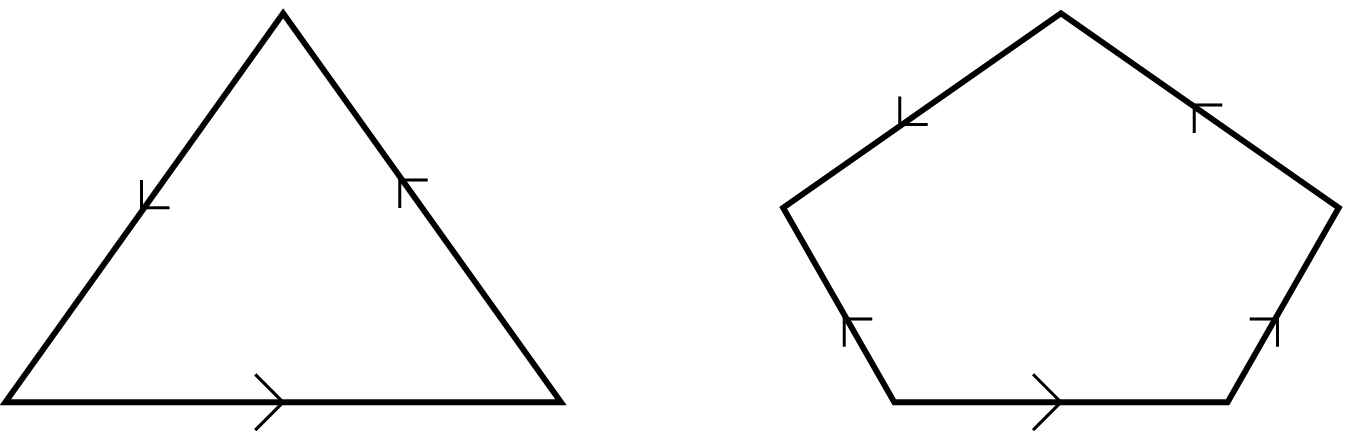}} 
\caption{(top) The front projection of a Legendrian, $L$, with two swallowtail points.  The swallowtail point on the left (resp. right) is upward (resp. downward).
(middle) A compatible cell decomposition of the base projection of $L$ with choices of $S$ and $T$ decorations made at swallowtail points. (bottom)  The polygon $Q$ used in defining $\partial C_3$ is formed by adding additional edges at all $S$ and $T$ corners of $P$. 
}
\label{fig:Ex2Base}
\end{figure}

Next, we assign a matrix $A_v$ to each vertex $v$ of $Q$.  We identify each such vertex with a $0$-cell in $\mathcal{E}$ by using the characteristic map of $e^2_\alpha$ and declaring that both endpoints of $S$ and $T$ edges are identified with the swallowtail point of the corresponding $S$ or $T$ region.
As before, we use the total ordering of sheets above $e^2_\alpha$ and the nature of the singular set above $v$ to form a matrix $A_{v}$ from the generators of the $0$-cell $v$ with the following modifications made when $v$ is a swallowtail point.  

\begin{itemize}
\item Suppose $v$ is the initial vertex of an $S$ or $T$ edge with respect to the orientation of that edge. Then, we respectively set $A_{v}=A_S$ or $A_{v} = A_T$. 

\item  Suppose $v$ is the terminal vertex of an $S$ or $T$ edge \underline{or} that small neighborhoods of $v$ are mapped within the swallow tail region, but not to one of the regions $S$ or $T$.  Then, if the swallowtail is upward $A_{v} = \widehat{A}_{k,k+1}$, and if the swallowtail is downward $A_{v} = \widehat{A}_{l-1,l}$.  (That is, we form $A_{v}$ as if $v$ were located on the portion of the cusp edge at the swallowtail point that is in the same half of the swallow tail region as the image of a neighborhood of $v$.)

\item  If the $2$-cell is in the complement of the swallowtail region, then no adjustment is required to define $A_{v}$.  (That is, $A_v = A$.)  
\end{itemize}

Next, we assign a matrix $Y_e$ to each edge of $Q$ as follows.
\begin{itemize}
\item  Suppose $e$ is an edge corresponding to a $1$-cell of $\mathcal{E}$. Then, we use the total ordering of the sheets above $e^2_\alpha$ and the nature of the singular set above $e$ to place the generators $b_{i,j}$ associated to $e$ into an $n \times n$ matrix $B_e$.  This is precisely as in Section \ref{sec:Cdef}.  We then set 
\[
Y_e = I+B_e.
\]

\item  Suppose $e$ is an edge corresponding to an $S$ region at a swallowtail point.  Then, we take $Y_e=S$.

\item  Suppose $e$ is an edge corresponding to a $T$ region.  Then, we take $Y_e=T$.
\end{itemize}

We form an upper triangular matrix $C$ from the generators $c^\alpha_{i,j}$ associated to $e^2_\alpha$.  To define $\partial C$, we make a choice of an initial and terminal vertex $v_0$ and $v_1$ from the vertices of $Q$.  [Once again, if $v_0=v_1$, then a direction needs to be chosen for the path around $\partial Q$ from $v_0$ to $v_1$.]  We let $\gamma_+$ and $\gamma_-$ denote paths around $\partial Q$ that respectively proceed counter-clockwise and clockwise from $v_0$ to $v_1$.  [If $v_0 = v_1$, one of these paths is constant as specified by the choice of direction from $v_0$ to $v_1$.]

Let $Y_1, \ldots, Y_{j}$ (resp. $Y_{j+1}, \ldots, Y_{m}$) denote the matrices $Y_e$ associated to successive edges of $Q$ that appear along $\gamma_+$ (resp. $\gamma_-$), and let $A_{v_0}$ and $A_{v_1}$ be the matrices associated to the vertices $v_0$ and $v_1$.  We define $\partial c^\alpha_{i,j}$ so that
\begin{equation} \label{eq:CdefSwallowtail}
\partial C = A_{v_1}C + C A_{v_0} + Y_j^{\eta_j} \cdots Y_1^{\eta_1} + Y_m^{\eta_m} \cdots Y_{j+1}^{\eta_{j+1}}
\end{equation}
where the exponent, $\eta_i$, is $1$ (resp. $-1$) if the orientation of the corresponding edge of $Q$ agrees (resp. disagrees) with the orientation of $\gamma_{\pm}$.  Here, it is useful to note that $S=S^{-1}$ and $T= T^{-1}$.

\begin{lemma} \label{lem:SwallowTail} Equation (\ref{eq:Cdef}) leads to a well defined definition of $\partial c^\alpha_{p,q}$, and $\partial^2c^\alpha_{p,q}=0$.
\end{lemma}
\begin{proof}
Equation (\ref{eq:CdefSwallowtail}) may be used to define $\partial c^\alpha_{i,j}$, since the matrices $A_{v_i}$ and $C$ are strictly upper triangular and the $Y_i$ are upper triangular with $1$'s on the diagonal.  [Thus, both products $Y_j^{\eta_j} \cdots Y_1^{\eta_1}$ $Y_m^{\eta_m} \cdots Y_{j+1}^{\eta_{j+1}}$ have the form $I + X$ with $X$ strictly upper triangular, so that the entries on the main diagonal cancel.]  

Recall that 
matrices have been assigned to all edges and vertices of $Q$.  For an edge, $e$, of $Q$ let $Y_e$, $A_{e^-}$ and $A_{e^+}$ denote the matrices so assigned to $e$ and the initial and terminal vertices $e^-$ and $e^+$ of $e$ (with respect to the orientation of $e$). 
We claim that 
\begin{equation} \label{eq:BiA}
\partial Y_e = A_{e^+} Y_e + Y_e A_{e^-}.
\end{equation}
Note that $\partial^2 C=0$ can then be checked using an argument similar to the proof of Lemma \ref{lem:cwd} with (\ref{eq:BiA}) used in place of (\ref{eq:dbi}).  In the case where $e$ corresponds to a $1$-cell of $\mathcal{E}$, (\ref{eq:BiA}) follows from observing  the relation between the matrices $Y_e$ and $A_{e^\pm}$, formed when viewing $e$ and $e^{\pm}$ as belonging to the boundary of $Q$, and the matrices $I+B$ and $A_\pm$ from (\ref{eq:Bdef}) that are used in defining the differential of the generators associated to $e$.  As in the proof of Lemma \ref{lem:cwd}, if it is not the case that two sheets of $e^2_{\alpha}$ meet at a cusp  above $e$, then $Y_e = I +B$ and $A_{e^\pm} = A_{\pm}$ provided that, in forming $I+B$ and $A_\pm$, we use the ordering of sheets above $e^2_{\alpha}$ and follow the provisions of \ref{sec:BdefST}.  

Finally, if $e$ is an edge of $Q$ labeled with an $S$ or $T$, then (\ref{eq:BiA}) is one of the equations in (\ref{eq:partialS}) (resp. in (\ref{eq:partialS2})) from Lemma \ref{lem:ASAT} if the swallowtail is upward (resp. downward). 
\end{proof}

An explicit example computing this differential with swallowtails appears in Section \ref{ssec:ExampleST}.

\medskip

We summarize the results of Section \ref{sec:DefDGA}.
\begin{theorem}  \label{thm:Summary}
The cellular DGA $(\mathcal{A}, \partial)$ satisfies $\partial^2 =0$.  A choice of Maslov potential, $\mu$, on $L$ provides a $\Z/m(L)$-grading on $\mathcal{A}$ for which $\partial$ has degree $-1$.   
\end{theorem}
\begin{proof}
That $\partial^2=0$ has been verified during the definition of $(\mathcal{A},\partial)$.  Using equation (\ref{eq:graddef}) it is straightforward to verify that $\partial$ has degree $-1$. 
\end{proof}

\section{Independence of cell decomposition}  \label{sec:Ind}

In this section we prove the following: 

\begin{theorem} \label{thm:IndCellD}  The stable tame isomorphism type of the cellular DGA $(\A,\partial)$ is independent of the choice of cell decomposition $\mathcal{E}$ and additional data.
\end{theorem}

This requires showing independence of the following items:
\begin{enumerate}
\item  The orientation of $1$-cells.
\item The choice of initial and terminal vertex for each $2$-cell. 
\item  The choice of decorations at swallow tail points.
\item  The choice of cell decomposition $\mathcal{E}$.
\end{enumerate}
These results are obtained in Corollary \ref{cor:IndOrient},  Corollary \ref{cor:IndVer}, and Theorem \ref{thm:refine} below.  Before embarking upon their proof we collect some algebraic preliminaries.


\subsection{Ordering of generators}  Our main tool for producing stable tame isomorphisms will be Theorem \ref{thm:Alg} whose application requires a DGA to have its generating set ordered so that the differential becomes triangular.  Such orderings are provided for the cellular DGA in the following Lemma \ref{lem:precA}.

The cellular DGA $(\A,\partial)$ was constructed in Section \ref{sec:DefDGA} with a specific genererating set.  Define a partial ordering of this generating set, $\prec_\A$, by declaring that $y \preceq_\A x$ if
\begin{enumerate}
\item The cell corresponding to $x$ has larger  dimension than the cell corresponding to $y$, or 
\item The same cell, $e^i_\alpha$, corresponds to both $x$ and $y$, and  subscripts $p_1,q_1$ and $p_2,q_2$ for $x$ and $y$ are such that $S_{p_1} \preceq S_{p_2}$ and $S_{q_2} \preceq S_{q_1}$ holds in $L(e^i_\alpha)$.  
\end{enumerate}

\begin{lemma}  \label{lem:precA}
The differential of $(\A,\partial)$ is triangular with respect to any ordering of the generating set that extends $\prec_\A$.
\end{lemma}
\begin{proof}
As in the proof of Lemma \ref{lem:Awd}, we have $\partial\,a^\alpha_{p,q} = \sum a^\alpha_{pr}a^\alpha_{rq}$ with the sum over those sheets $S_r \in L(e^0_\alpha)$ such that $S_p \prec S_r \prec S_q$.  Therefore, all of these generators satisfy $ a^\alpha_{pr},a^\alpha_{rq}\prec_\A a^\alpha_{p,q} $. 

In the formula (\ref{eq:Bdef}) that characterizes $\partial b^\alpha_{p,q}$, only $A_+B$ and $BA_-$ give rise to terms that do not correspond to cells of lower dimension.  
Recall that a total ordering  $\iota: \{1,\ldots, n\} \rightarrow L(e^1_\alpha)$ is used to form the matrices  $A_\pm$ and $B$, so that 
the $A_+B$ term corresponds to a sum $\sum_{i<k<j} x_{i,k}b_{\iota{(k)},\iota{(j)}}$ in $\partial b_{\iota{(i)},\iota{(j)}}$ where $A_+ = (x_{i,j})$.  
Since $i <k <j$, $b_{\iota(k), \iota(j)} \prec_\A b_{\iota(i),\iota(j)}$ unless the sheets $S_{\iota(i)}$ and $S_{\iota(k)}$ are incomparible in $L(e_\alpha^1)$. However, if this is the case then the entry $x_{i,k}$ was checked to be $0$ in Lemma \ref{lem:Bwd} and Lemma \ref{lem:SwallowTail}.  The $BA_-$ term is handled similarly.

For a generator $c_{i,j}$ corresponding to a $2$-cell, only the terms $A_{v_1}C$ and $C A_{v_0}$ are relevant.  Since the sheets above a $2$-cell are totally ordered, we have $c_{i,k},c_{k,j} \prec_\A c_{i,j}$ for $i < k < j$.
\end{proof}

\subsection{Elementary modifications to a compatible polygonal decomposition}  To prove Theorem \ref{thm:IndCellD}, we begin by showing that the stable tame isomorphism type of $(\mathcal{A},\partial)$ is invariant under certain local changes to the defining data.  

\subsubsection{Subdividing a $1$-cell}
We say that compatible cell decompositions $\mathcal{E}$ and $\mathcal{E}'$ for $L$ are related by {\bf subdividing a $1$-cell} if $\mathcal{E}'$ is obtained from $\mathcal{E}$ by dividing a $1$-cell, $e^1_\alpha$, into two pieces, $e^1_x$ and $e^1_y$, by placing a new vertex $e^0$ somewhere along $e^1_\alpha$.  Let $e^0_x$ and $e^0_y$ denote the endpoints of $e^1_{\alpha}$ in $\mathcal{E}$ that become endpoints of $e^1_x$ and $e^1_y$ respectively in $\mathcal{E}'$.  See Figure \ref{fig:Sub1Cell} (a). 

\begin{figure}
\labellist
\small
\pinlabel $e^1_\alpha$ [b] at 60 86
\pinlabel $e^1_x$ [b] at 248 86
\pinlabel $e^1_y$ [b] at 302 86
\pinlabel $e^0_x$ [t] at 3 76
\pinlabel $e^0_y$ [t] at 116 76
\pinlabel $e^0_x$ [t] at 220 76
\pinlabel $e^0$ [t] at 276 76
\pinlabel $e^0_y$ [t] at 332 76
\pinlabel $(a)$ [c] at -26 76
\pinlabel $(b)$ [c] at -26 8
\pinlabel $\cong$ [c] at 136 8
\pinlabel $\cong$ [c] at 288 8
\endlabellist
\centerline{\includegraphics[scale=.8]{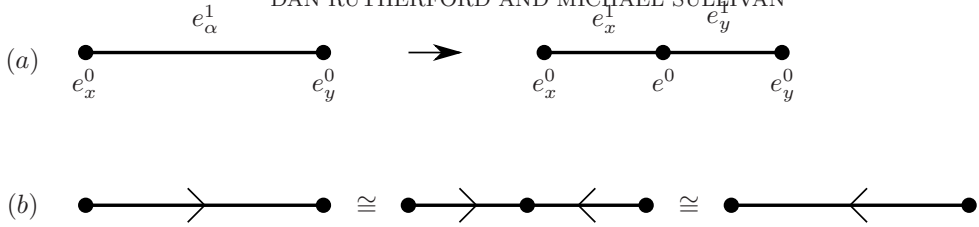}} 

\caption{(a) Subdiving a $1$-cell. (b) Reversing the orientation of a $1$-cell.  
}
\label{fig:Sub1Cell}
\end{figure}

\begin{theorem} \label{thm:sub1cell}
Suppose that $\mathcal{E}$ and $\mathcal{E}'$ are related by subdividing a $1$-cell.  Moreover assume that:
\begin{enumerate} 
\item The orientation of $e^1_x$ agrees with the orientation of $e^1_\alpha$.
\item The $S$ and $T$ regions at swallow tail points, as well as initial and terminal vertices for $2$-cells of $\mathcal{E}$ and $\mathcal{E}'$ are chosen in an identical way. 
\end{enumerate} 
Then, the corresponding cellular DGAs are stable tame isomorphic.
\end{theorem} 

\begin{proof}
Let $(\A,\partial)$ and $(\A',\dd')$ denote the DGAs associated to $\mathcal{E}$ and $\mathcal{E}'$ respectively.  We fix a total ordering of the sheets in  $L(e^1_\alpha)$, and use this choice to produce total orderings of $L(e^1_x), L(e^1_y), L(e^0), L(e^0_x)$ and $L(e^0_y)$.  We can then collect corresponding generators of $\A'$ (resp. $\A$) into matrices $B_x,B_y,A_0,A_x,A_y$ (resp. $B, A_x,A_y$) where, in the case that $e^0_x$ or $e^0_y$ is a swallow tail point, we follow the instructions from \ref{sec:BdefST} when forming $A_x$ or $A_y$. 

Assume the orientation of $e^1_\alpha$ is from $e^0_x$ to $e^0_y$, as the argument for the reverse orientation is similar.  Then, we have
\[
\dd B = A_y (I +B) + (I+B) A_x;  
\]
\[
\dd' B_x = A_0 (I +B_x) + (I+B_x) A_x; 
\] 
and
\[
\dd' B_y = \left\{ \begin{array}{cr} A_y (I +B_y) + (I+B_y) A_0 & \mbox{if $e^0_y$ and $e^0_\alpha$ have the same orientation,} \\
A_0(I+B_y) + (I+B_y)A_y & \mbox{else}. 
\end{array}\right.
\]

  We can extend the partial ordering, $\prec_{\A'}$,   to a total ordering so that all of the generators corresponding to $e^0$ are greater than the generators corresponding to $e^0_y$.  Working in increasing order of the $b^y_{p,q}$, we then apply Theorem \ref{thm:Alg} inductively to 
cancel the generators $b^y_{p,q}$ in pairs with the $a^0_{p,q}$. 
[Note that the sheets above $e^0$ are in bijection with the sheets of $e^1_y$ and have the same partial ordering.  Therefore, the generators $b^y_{p,q}$ and $a^0_{p,q}$ are in bijective correspondence.] In the resulting quotient, the entries of $A_0$ will all be replaced by the corresponding entries of $A_y$.  This is because of the order that we cancel the $b^y_{p,q}$;  at the inductive step, any of the $B_y$ terms that would appear in $\partial b^y_{p,q}$ have already been cancelled, and we indeed have $\partial b^y_{p,q}= a^0_{p,q} +w$ where $w$ is the corresponding entry in $A_y$ (which is less than $a^0_{p,q}$ as required in Theorem \ref{thm:Alg}).  

Thus, the resulting stable tame isomorphic quotient is obtained from $(\A',\dd')$ by replacing $B_y$ with $0$, and replacing all occurrences of $A_0$ in $\dd'$ with $A_y$.  [Note that this includes in $\partial C$ if $C$ is a $2$-cell bordering $e^1_\alpha$.  When compared with $\dd C$ there is an extra $(I+B_i)^{\eta_i}$ factor in $\dd' C$ corresponding to the edge $e^1_y$.  This term is replaced with $I^{\eta_i}$ in the quotient.] 

\end{proof} 

\begin{corollary} \label{cor:IndOrient} The stable tame isomorphism type of $(\A,\dd)$ is independent of the orientation of $1$-cells of $\mathcal{E}$.
\end{corollary}
\begin{proof}  
To reverse the orientation of a $1$-cell, $e^1_\alpha$, of $\mathcal{E}$, apply Theorem \ref{thm:sub1cell} twice.  First, subdivide into $e^1_x$ and $e^1_y$ so that $e^1_y$ has opposite orientation to $e^1_\alpha$.   This is just as well a subdivision of $\mathcal{E}$ with the orientation of $e^1_\alpha$ reversed if the roles of $e^1_x$ and $e^1_y$ are interchanged.  See Figure \ref{fig:Sub1Cell} (b).
\end{proof}

\subsubsection{Subdividing a $2$-cell}  
Suppose that $\mathcal{E}$ and $\mathcal{E}'$ are cell decompositions of $S$ that are compatible with $L$ and that $\mathcal{E}'$ is obtained from $\mathcal{E}$ by removing a $1$-cell, $e^1_\alpha$, that borders two distinct $2$-cells, $e^2_x$ and $e^2_y$, of $\mathcal{E}$.  Thus, $\mathcal{E}'$ has a single $2$-cell, $e_z^2$, satisfying 
\[
e^2_z = e^2_x \sqcup e^1_\alpha \sqcup e^2_y.
\]
Note that (because $\mathcal{E}'$ is compatible with $L$) it is necessarily the case that $e^1_\alpha$ is disjoint from $\Sigma$.  We say that $\mathcal{E}$ and $\mathcal{E}'$ are related by {\bf subdividing a $2$-cell} of $\mathcal{E}'$.  See Figure \ref{fig:Sub2Cell}.

\begin{theorem} \label{thm:sub2cell}  Suppose $\mathcal{E}$ and $\mathcal{E}'$ are related by subdividing a $2$-cell.  Suppose in addition that:
\begin{enumerate}
\item Initial and terminal vertices are chosen for $2$-cells of $\mathcal{E}$ and $\mathcal{E}'$ so that the initial and terminal vertices of $e^2_z$ are chosen to coincide with those of $e^2_x$, and the  choice made for all remaining $2$-cells of $\mathcal{E}'$ coincides with the choice for the corresponding $2$-cells of $\mathcal{E}$.
\item The $S$ and $T$ sides of the crossing locus are assigned near each swallowtail point in an identical manner for $\mathcal{E}$ and $\mathcal{E}'$.  Here, we allow for the possibility that $e^1_\alpha$ borders $S$ and $T$ regions at either of its endpoints.   
\end{enumerate}
 Then, the associated cellular DGAs are stable tame isomorphic.
\end{theorem}
\begin{proof}
As in Section \ref{sssec:sw2cell}, let  $Q_x$, $Q_y$, and $Q_z$ denote the polygons associated to the $2$-cells $e^2_x$, $e^2_y$ of $\mathcal{E}$ and the $2$-cell $e^2_z$ respectively.    In addition, for $w \in \{x,y,z\}$, let $\gamma_{w,\pm}$ denote the paths around $\partial Q_w$ from the initial vertex to the terminal vertex.  Note that the edge $e^1_\alpha$ appears precisely once along either $\gamma_{x,+}$ or $\gamma_{x,-}$ and once along either $\gamma_{y,+}$ or $\gamma_{y,-}$.  Since interchanging the notations of $\gamma_+$ and $\gamma_-$ has no effect on the definition of the differential from  (\ref{eq:CdefSwallowtail}), we can assume without loss of generality that $\gamma_{x,+}$ and $\gamma_{y,+}$ each contain $e^1_\alpha$ exactly once.  Moreover, in view of Corollary \ref{cor:IndOrient} we may assume that the orientation of $e^1_\alpha$ agrees with the orientation of $\gamma_{x,+}$.  
Using $*$ for concatenation of paths, we can then write 
\[
\gamma_{x,+} = \gamma_{x,a} * e^1_\alpha * \gamma_{x,b} \quad \mbox{and} \quad  
\gamma_{y,+} = \gamma_{y,a} * (e^1_\alpha)^\nu * \gamma_{y,b}\]
where $\nu = \pm 1$ and some of these paths may be constant.  Note that
\begin{equation} \label{eq:gzgx}
\gamma_{z,-} = \gamma_{x,-} \quad \mbox{and} \quad \gamma_{z,+} = \gamma_{x,a} * ((\gamma_{y,a})^{-1} * \gamma_{y,-}* (\gamma_{y,b})^{-1})^{\nu}* \gamma_{x,b}.
\end{equation}
 See Figure \ref{fig:Sub2Cell}.
\begin{figure}
\labellist
\small
\pinlabel $e^2_x$ [c] at 132 132
\pinlabel $e^2_y$ [c] at 308 124
\pinlabel $e^2_z$ [c] at 722 132
\pinlabel $e^1_\alpha$ [l] at 230 132
\pinlabel $\gamma_{x,-}$ [tr] at -2 122
\pinlabel $v_0$ [t] at 90 2
\pinlabel $v_1$ [b] at 98 264
\pinlabel $\gamma_{x,a}$ [t] at 162 8
\pinlabel $\gamma_{x,b}$ [bl] at 170 236
\pinlabel $v_0$ [bl] at 390 224
\pinlabel $v_1$ [tl] at 390 44
\pinlabel $\gamma_{y,-}$ [l] at 404 138
\pinlabel $\gamma_{y,a}$ [br] at 306 216
\pinlabel $\gamma_{y,b}$ [t] at 306 28
\pinlabel $\gamma_{z,+}$ [t] at 824 28
\pinlabel $\gamma_{z,-}$ [tr] at 544 122
\pinlabel $v_1$ [b] at 634 264
\pinlabel $v_0$ [t] at 626 2

\endlabellist
\centerline{\includegraphics[scale=.4]{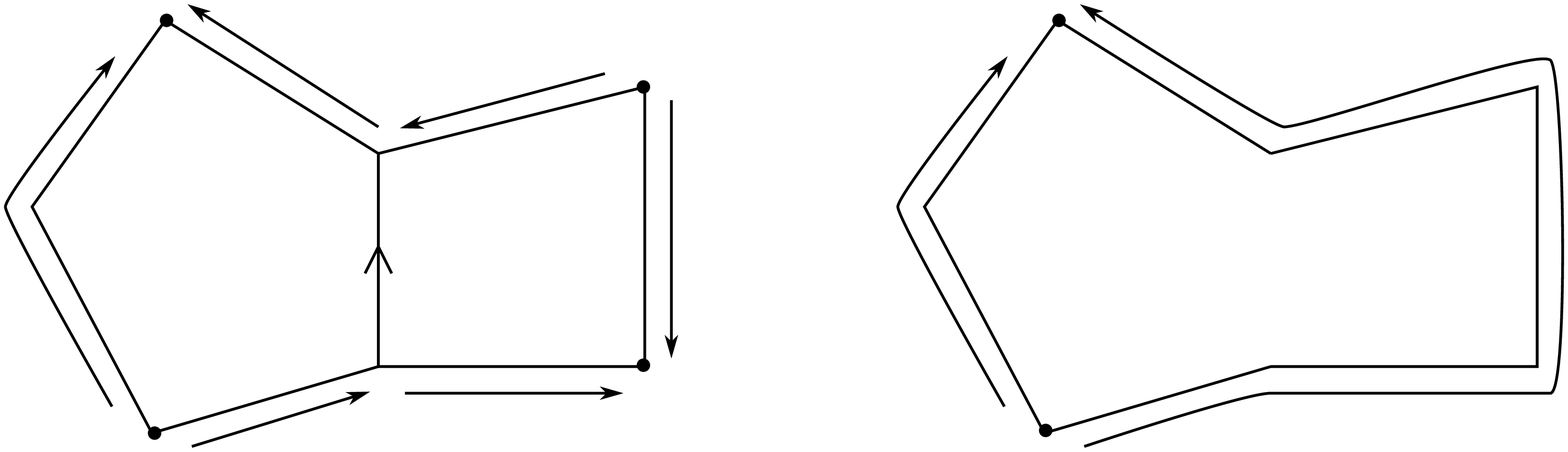}} 

\caption{The cell decompositions $\mathcal{E}$ (left) and $\mathcal{E}'$ (right).  Note that the exponent $\nu = -1$. 
}
\label{fig:Sub2Cell}
\end{figure}
In view of Corollary \ref{cor:IndOrient}, we may assume the orientation of $e^1_\alpha$ is from the initial vertex to the terminal vertex.  
Using Theorem \ref{thm:Alg}, we now show that the DGAs $(\A, \partial)$ and $(\A,\partial')$ are stable tame isomorphic.  Collect generators of $\A$ (resp. $\A'$) associated to $e^2_x$ and $e^2_y$ (resp. to $e^2_z$) into matrices $C_x$ and $C_y$ (resp. $C_z$).   We choose an extension of the ordering of generators of $A$ from $\prec_\A$ to a total order so that the generators corresponding to $e^1_\alpha$ are larger than any other $1$-cell generators.  Using this ordering we can inductively cancel the $c^y_{i,j}$ with the $b^\alpha_{i,j}$ using Theorem \ref{thm:Alg}.  [Indeed, expanding the product of edges along $\gamma_{y,+}$ that appears in equation (\ref{eq:CdefSwallowtail}) for $\partial C_y$ allows us to write
\begin{equation}  \label{eq:CyAv1}
\partial C_y = A_{v_1} C_y + C_y A_{v_0} + B_{\alpha} + X
\end{equation}
where the matrix $X$ is strictly upper triangular with and has its $i,j$ entry in the sub-algebra generated by generators from $0$-cells and $1$-cells different from $e^1_\alpha$ and also those $b^\alpha_{i',j'}$ with $i \leq i' < j' \leq j$ such that at least one of the first and last inequalities is strict, i.e. those $b^\alpha_{i',j'}$ with $b^\alpha_{i',j'} \prec_\A b^\alpha_{i,j}$.  Therefore,  we can apply Theorem \ref{thm:Alg} and inductively quotient by ideals generated by $c_{i,j}$ and $\partial c_{i,j}$ according to the increasing ordering of these generators.  For the inductive step, it is important to observe that, as long as $c_{i',j'}= 0$ for $i',j'$ such that $c_{i',j'} \prec_A c_{i,j}$, equation  (\ref{eq:CyAv1}) shows that $\partial c_{i,j} = b^\alpha_{i,j} + w$ where $w$ belongs to the subalgebra generated by those $x$ with $x \prec_A b^\alpha_{i,j}$.]

The resulting quotient has the $c^y_{i,j}$ and $b^\alpha_{i,j}$ removed from the generating set.  Moreover, the relations
\begin{align*}
& 0 = C_y, \quad \mbox{and} \\
& 0 = \partial C_y = A_{v_1} C_y +C_y A_{v_0} + \underbrace{Y^{\eta_j}_{j} \cdots Y^{\eta_1}_1}_{\gamma_{y,+}} + \underbrace{Y^{\eta_m}_{m} \cdots Y^{\eta_{j+1}}_{j+1}}_{\gamma_{y,-}} = \\
& \quad \underbrace{Y^{\eta_j}_{j} \cdots Y^{\eta_{l+1}}_{l+1}}_{\gamma_{y,b}} (I+B_{\alpha})^{\nu}  \underbrace{Y^{\eta_{l-1}}_{l-1} \cdots Y^{\eta_{1}}_{1}}_{\gamma_{y,a}}+ \underbrace{Y^{\eta_m}_{m} \cdots Y^{\eta_{j+1}}_{j+1}}_{\gamma_{y,-}}
\end{align*}
allow us to find 
\begin{equation} \label{eq:IBalpha}
(I+B_{\alpha}) = \big( [\underbrace{Y^{\eta_j}_{j} \cdots Y^{\eta_{l+1}}_{l+1}}_{\gamma_{y,b}}]^{-1}  \underbrace{Y^{\eta_m}_{m} \cdots Y^{\eta_{j+1}}_{j+1}}_{\gamma_{y,-}}  [\underbrace{Y^{\eta_{l-1}}_{l-1} \cdots Y^{\eta_{1}}_{1}}_{\gamma_{y,a}}]^{-1}\big)^\nu.
\end{equation}
(We have indicated which parts of the paths $\gamma_{y,-}$ and $\gamma_{y,+}$ the various portions of the product correspond to. The $Y_i$ matrices are as in Section \ref{sssec:sw2cell}, and some of them may be $S$ or $T$ matrices from corners decorated at swallowtail points.)  

Finally, note that $\partial C^x$ becomes identical to $\partial' C^z$ after making the substitution (\ref{eq:IBalpha}).  [Compare with (\ref{eq:gzgx}).]  
It follows that the identification of $C_x$ with $C_z$ provides an isomorphism between this quotient of $(\A,\partial)$ and $(\A', \partial')$.

\end{proof}

\begin{corollary} \label{cor:IndVer} The stable tame isomorphism type of $(\A,\dd)$ is independent of the choices of initial and terminal vertices for $2$-cells.
\end{corollary} 
\begin{proof}
Let $\mathcal{E}$ be an $L$-compatible polygonal decomposition of $S$, and let $e^2_\alpha$ be any $2$-cell of  $\mathcal{E}$.  Let $(\A,\dd)$ denote a cellular DGA formed from using $v_0$ and $v_1$ as the initial and terminal vertices for $e^2_\alpha$.

Pick a $0$-cell, $e^0$, appearing in the boundary of $e^2_\alpha$ and modify the cell decomposition $\mathcal{E}$ to $\mathcal{E}'$ by adding a loop edge, $e^1$, from $e^0$ to itself that is contained in $e^2_\alpha$.  This subdivides $e^2_\alpha$ into two pieces $e^2_x$ and $e^2_y$ where $e^2_x$ is exterior to $e^1$ and $e^2_y$ is interior to $e^1$.  Choose initial and terminal vertices of $e^2_x$ to be $v_0$ and $v_1$, and choose both initial and terminal vertices of $e^2_y$ to be $e^0$.  Keep all other choices the same. Then according to Theorem \ref{thm:sub2cell}  the DGA $(\A',\dd')$ associated to $\mathcal{E}'$ is stable tame isomorphic to $(\A,\dd)$.  On the other hand, if we reverse the role of $e^2_x$ and $e^2_y$, we see that the DGA associated to $\mathcal{E}$ modified so that the initial and terminal vertices $v_0$ and $v_1$ for $e^2_{\alpha}$ are both replaced with $e^0$ is also stable tame isomorphic to $(\A',\dd')$.  Since stable tame isomorphism satisfies the properties of an equivalence relation, it follows that the stable tame isomorphism type of $(\A,\dd)$ is independent of the choice of $v_0$ and $v_1$.  See Figure \ref{fig:Sub2Proof}.
\end{proof}

\begin{figure}
\labellist
\small
\pinlabel $v_0$ [t] at 74 -2
\pinlabel $v_1$ [b] at 82 260
\pinlabel $v_0$ [t] at 384 -2
\pinlabel $v_1$ [b] at 402 258
\pinlabel $e^0$ [tl] at 534 38
\pinlabel $e^0$ [tl] at 854 38
\pinlabel $\cong$ [c] at 268 124
\pinlabel $\cong$ [c] at 590 124
\endlabellist
\centerline{\includegraphics[scale=.4]{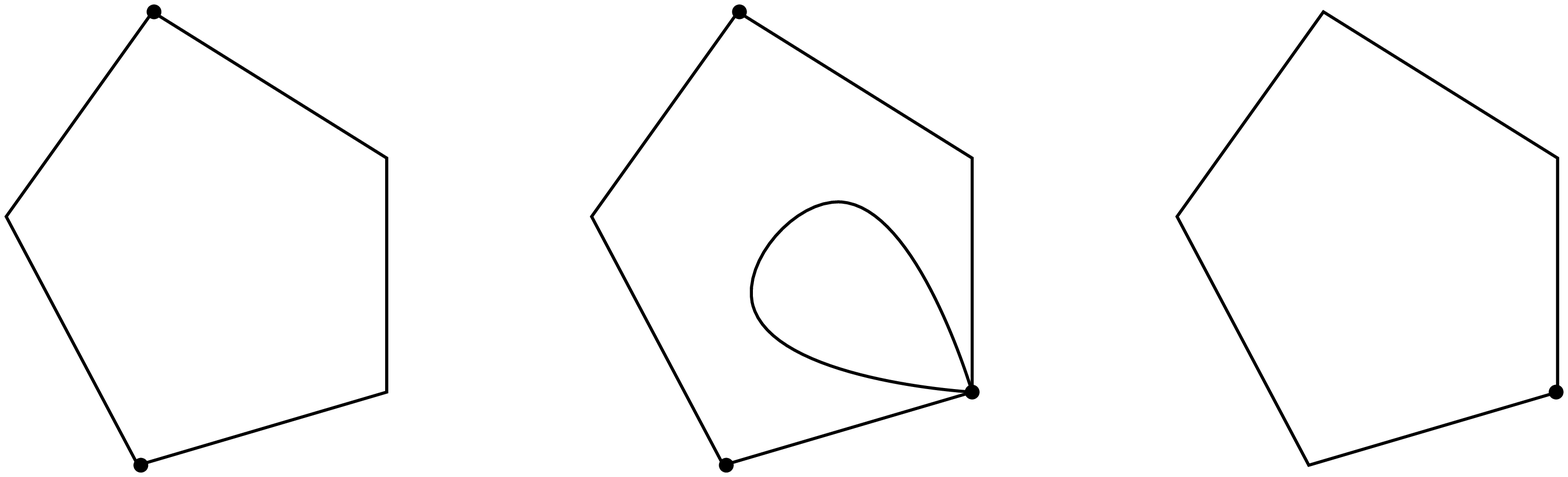}} 

\quad

\caption{The cell decompositions and initial and terminal vertices used in the proof of Corollary \ref{cor:IndVer}. 
}
\label{fig:Sub2Proof}
\end{figure}

\subsubsection{Deleting an edge with $1$-valent vertex} 
Suppose that an edge of $\mathcal{E}$ has an endpoint at a $1$-valent vertex.  The edge and vertex must be disjoint from the singular set $\Sigma \subset S$, and therefore we can delete them to produce another $L$-compatible cell decomposition of $S$ which we denote as $\mathcal{E}'$.
\begin{theorem} \label{thm:1val} For $\mathcal{E}$ and $\mathcal{E}'$ related by deleting an edge with $1$-valent vertex, the associated cellular DGAs for $L$ are stable tame isomorphic.
\end{theorem}
\begin{proof}
Cancel the generators of the $1$-cell with the generators of the univalent vertex as in the proof of Theorem \ref{thm:sub1cell}.
\end{proof}

\subsection{Independence of $S$ and $T$ decorations at swallow tail points}

\begin{theorem}  \label{thm:IndDecorate}
The stable tame isomorphism type of the cellular DGA is independent of the choice of $S$ and $T$ regions at swallowtail points.
\end{theorem}

\begin{proof}  It suffices to establish stable tame isomorphism between cellular DGAs $(\A,\partial_1)$ and $(\A, \partial_2)$ arising from a common $L$-compatible decomposition, $\mathcal{E}$, such that the choice of $S$ and $T$ regions is opposite at a single swallowtail point, $s \in \Sigma$, as pictured in Figure \ref{fig:DecorateProof}.  We give the proof only in the case of an upward swallowtail with $n$ sheets above the swallow tail region and so that sheets $k, k+1,$ and $k+2$ meet at the swallowtail point.  The case of a downward swallowtail is similar.

Using Theorems \ref{thm:sub1cell} and \ref{thm:sub2cell} as well as Corollaries \ref{cor:IndOrient} and \ref{cor:IndVer},  we may assume that
\begin{itemize}
\item  The $1$-cell, $e^1_\alpha$, that has an endpoint at $s$ and is contained in the crossing arc is  oriented away from $s$.   
\item  The $2$-cells containing the $S$ and $T$ regions are distinct.  
\item  The initial and terminal vertices of these $2$-cells are disjoint from $s$, and the orientations of the paths $\gamma_\pm$ around the boundaries of the $2$-cells agree with the orientation of $B$ as they pass around the corner of the $S$ and $T$ regions. 
\end{itemize}

Fixing an orientation of $S$ (the base surface) near the swallowtail point, we collect the generators for the $2$-cell that sits to the left (resp. right) of $e^1_\alpha$ into matrices $C_1$ and $C_2$.  We form a matrix $B$ from the generators for $e^1_\alpha$ by using the total ordering of sheets as it appears to the left of $e^1_\alpha$, so that rows and columns of $B$ and $C_1$ correspond to the same sheets.  The generators associated to the swallowtail point itself are placed in an $(n-2) \times (n-2)$ matrix $A$, and we have $n\times n$ matrices $A_S$, $A_T$, $S$, and $T$ as defined in Section \ref{sec:matrices}.  We suppose that for $(\A,\partial_1)$ (resp. $(\A,\partial_2)$), the $S$ region is to the right (resp. left) of $e^1_\alpha$ and the $T$ region is to the left (resp. right) of $e^1_\alpha$.  See Figure \ref{fig:DecorateProof}.

\begin{figure}
\[
\begin{array}{ccc}
\labellist
\small
\pinlabel $T$ [b] at 50 86
\pinlabel $S$ [t] at 50 76
\pinlabel $B$ [b] at 132 84
\pinlabel $C_1$ [c] at 110 134
\pinlabel $C_2$ [t] at 110 30
\endlabellist
\includegraphics[scale=.5]{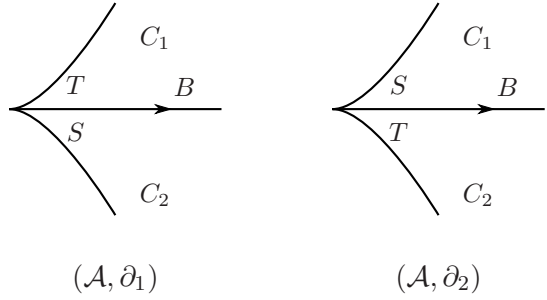}  & \quad \quad & 
\labellist
\small
\pinlabel $S$ [b] at 50 86
\pinlabel $T$ [t] at 50 76
\pinlabel $B$ [b] at 132 84
\pinlabel $C_1$ [c] at 110 134
\pinlabel $C_2$ [t] at 110 30
\endlabellist
\includegraphics[scale=.5]{images/DecorateProof} \\  &  &  \\
 (\A,\partial_1) & & (\A, \dd_2) 
\end{array}
\]
\caption{Generators and decorations near the swallow tail point $s$. 
}
\label{fig:DecorateProof}
\end{figure}

We define an algebra morphism $\psi: (\A, \partial_1) \rightarrow (\A,\partial_2)$ by requiring that when applied entry by entry
\begin{equation} \label{def:psiIso}
\psi(I + B) = (I+B)(S T),
\end{equation}
and $\psi(x) = x$ for any generator that is not associated to the $1$-cell $e^1_\alpha$.  To verify that (\ref{def:psiIso}) can be obtained by a unique assignment of values $\psi(b^\alpha_{i,j})$ it is necessary to note that $(I+B)(ST)= I+X$ where $X$ is strictly upper-triangular with $k+1,k+2$ entry is equal to $0$.  The latter claim holds since it is true for $B$ and also for
\begin{equation}  \label{eq:TSIAKK1}
S T = T S = I + \widehat{A}_{k,k+1}E_{k+2,k}.
\end{equation}

We claim that $\psi$ is in fact a tame isomorphism from $\A$ to itself.  To verify, note that for each generator $b_{i,j}^\alpha$,  since  the $\widehat{A}_{k,k+1}E_{k+2,k}$ term from (\ref{eq:TSIAKK1}) is strictly upper triangular, it follows from (\ref{def:psiIso})  
  that
\[
\psi(b_{i,j}^\alpha) = b_{i,j}^\alpha + w_{i,j}
\]
where $w_{i,j}$ belongs to the sub-algebra generated by generators $x$ with $x \prec_A b^\alpha_{i,j}$.  Define $\psi_{i,j}: \A \rightarrow \A$ so that
$\psi_{i,j}(b_{i,j}^\alpha) = b_{i,j}^\alpha + w_{i,j}$ and $\psi_{i,j}(x) =x$ for any generator not equal to $b_{i,j}^\alpha$.  The $\psi_{i,j}$ are elementary isomorphisms\footnote{In fact, the $\psi_{i,j}$ satisfy the stronger requirement discussed in Remark \ref{rem:variation}.
}.  Moreover, it is straightforward to check that $\psi$ can be written as the composition of all of the $\psi_{i,j}$ provided that we compose in such a way that the subscripts increase from right to left, with respect to a total ordering of the $b_{i,j}^\alpha$ that extends $\prec_A$.  Thus, $\psi$ is indeed a tame isomorphism.

To complete the proof, we check that $\psi \circ \partial_1 = \partial_2 \circ \psi$.  Note that it is enough to verify this equality when the two sides are applied entry-by-entry to the matrices $C_1$, $C_2$, and $I+B$.  This is because these matrices contain the only generators for which the entries of $B$ can appear in the differential.  To this end, we compute:

\noindent \textbf{1.}
\begin{align*}
\psi \circ \partial_1( C_1 )= & \psi( \cdots (I+B)T \cdots + \cdots) = \\
 &  \cdots (I+B)(ST)T \cdots + \cdots = \cdots (I+B)S \cdots + \cdots = \partial_2 \circ \psi (C_1)
\end{align*}
where we used that $T$ is self inverse.

\noindent \textbf{2.}  Using $Q$ for the permutation matrix of the transposition $(k+1 \,\, k+2)$,
\begin{align*}
\psi \circ \partial_1( C_2 )= & \psi( \cdots Q(I+B)Q S \cdots + \cdots) = \\
 &  \cdots Q(I+B)(TS)Q S \cdots + \cdots = \cdots Q(I+B)Q(TS)S \cdots + \cdots = \\
&  \cdots Q(I+B)QT \cdots + \cdots=  \partial_2 \circ \psi (C_1)
\end{align*}
where we used (\ref{eq:TSIAKK1}), that $(TS)Q= Q(TS)$ since the $k+1$ and $k+2$ columns and rows of $TS$ agree with those columns of the identity matrix, and that $S^2 =I$.  (The $Q$'s appear in $\partial_i(C_2)$ because the matrix $B$ was formed using the ordering of sheets above $C_1$ which has the $k+1$ and $k+2$ sheets in the opposite order that they appear in above $C_2$.)  

\noindent \textbf{3.}  With $A_+$ denoting a matrix corresponding to the terminal end point of $e^1_\alpha$, we compute
\begin{align*}  
\psi \circ \partial_1( I+B)= & \psi[ A_+ (I+B) + (I+B)A_T] = \\
 &  A_+ (I+B) ST + (I+B)STA_T = A_+ (I+B) ST + (I+B) S \widehat{A}_{k,k+1} T = \\
 &  A_+(I+B) ST + (I+B) (A_S S + \partial_2 S) T = \\
&  [A_+(I+B) + (I+B) A_S] ST + (I+B)(\partial_2 S)T = \\
 &  \partial_2(I+B) \cdot ST + (I+B) \cdot \partial_2(ST) = \partial_2 \circ\psi(I+B).
\end{align*}
Here, the 3rd and 4th
 equalities used identities from Lemma \ref{lem:ASAT}.
\end{proof}

\subsection{Common refinements for $L$-compatible cell decompositions}
Let $\mathcal{E}_1$ and $\mathcal{E}_2$ be any cell decompositions for $S$ that are compatible with $L$.  After modifying $\mathcal{E}_2$ by an ambient isotopy of $S$, preserving the sets $\Sigma_2$ and $\Sigma_1$, we can assume that the cells of $\mathcal{E}_1$ and $\mathcal{E}_2$ intersect transversally in each of the strata $\Sigma_2$, $\Sigma_1$, and $\Sigma_0:=S\setminus\Sigma$.  That is, if $e^d_{\alpha} \in \mathcal{E}_1$ and $e^{d'}_\beta \in\mathcal{E}_2$ satisfy $e^d_\alpha, e^{d'}_\beta \subset \Sigma_i$, then they intersect transversally when viewed as subsets of the $2-i$-dimensional manifold $\Sigma_i$. [This just amounts to requiring that the only common $0$-cells of $\mathcal{E}_1$ and $\mathcal{E}_2$ are in $\Sigma_2$, and that all $1$-cells and $0$-cells of $\mathcal{E}_1$ and $\mathcal{E}_2$ are transverse in $S\setminus \Sigma.$]  
Such a modification of $\mathcal{E}_2$ does not affect the cellular DGA.
\begin{theorem}  \label{thm:refine}
  With the above transversality assumption, we can transform $\mathcal{E}_1$ into $\mathcal{E}_2$ via a sequence of the local modifications  appearing in Theorems \ref{thm:sub1cell}, \ref{thm:sub2cell}, and \ref{thm:1val}. 
\end{theorem}

Theorem \ref{thm:refine}, together with Corollaries \ref{cor:IndOrient},  \ref{cor:IndVer}, and Theorem \ref{thm:IndDecorate},  establish Theorem \ref{thm:IndCellD}

\begin{proof}
By subdividing edges and $2$-cells, we can assume that:
\begin{equation}  \label{eq:char}
 \mbox{All characteristic maps for both $\mathcal{E}_i$ are embeddings of closed disks into $S$.}
\end{equation}

 We construct a $3$-rd $L$-compatible cell decomposition $\mathcal{E}$, and then show that $\mathcal{E}$ is related to both $\mathcal{E}_1$ and $\mathcal{E}_2$ in the required manner.

\paragraph{{\bf Defining $\mathcal{E}$}}  We start by defining $\mathcal{E}$ on $\Sigma$.  Here, the $0$-cells of $\mathcal{E}$ are the union of the $0$-cells of $\mathcal{E}_1$ and $\mathcal{E}_2$ that belong to $\Sigma$.  These two sets are disjoint except for $\Sigma_2$, and to obtain the one cells of $\mathcal{E}$ in $\Sigma$ we just subdivide the $1$-cells of $\mathcal{E}_1$ at any $0$-cells of $\mathcal{E}_2$ not in $\Sigma_2$.  

To extend the construction of $\mathcal{E}$ to all of $\Sigma$, we include the $1$-skeletons of both $\mathcal{E}_1$ and $\mathcal{E}_2$ in the $1$-skeleton of $\mathcal{E}$.  Note, that all $1$-cells and $0$-cells in $S\setminus\Sigma$ intersect transversally, so we triangulate this union of $1$-skeletons by adding new $4$-valent vertices at intersections of $1$-cells of $\mathcal{E}_1$ and $\mathcal{E}_2$ in $S\setminus \Sigma$.  Denote this union of $1$-skeletons as $X$.  

The components of $S \setminus X$ are open surfaces with polygonal boundary in $X$.  They are planar, since they are contained in cells of the $\mathcal{E}_i$, and can be subdivided into disks by adding some extra $1$-cells that connect distinct boundary components.  This completes the construction of $\mathcal{E}$.

It remains to prove that $\mathcal{E}$ is related to both of the $\mathcal{E}_i$ by the $3$ allowable moves:
\begin{enumerate}
\item[(A)]  Adding/deleting an edge with totally ordered sheets that borders two distinct $2$-cells.
\item[(B)]  Adding/deleting a $0$-cell that subdivides a $1$-cell.
\item[(C)]  Adding/deleting an edge with a $1$-valent vertex.
\end{enumerate}

For this purpose, let $e$ be a $2$-cell in  $\mathcal{E}_1$ and consider the polygonal decomposition of $D^2$ consisting of preimages of cells of $\mathcal{E}$.  

\paragraph{{\bf Step 1.}}  Remove all $0$ and $1$-cells from the interior of $e$, by repeated application of (A)-(C).  This is done as follows.  If there is more than one $2$-cell in $e$, then there must be some edge that borders two distinct $2$-cells which we delete to decrease the number of $2$-cells.  [The sheets of $L$ are totally ordered above the edge since it is in the interior of $e$ and the singular set $\Sigma$ is contained in the $1$-skeleton of $\mathcal{E}_1$.]  Once there is only one $2$-cell in the polygonal decomposition of $e = D^2$, then either the $1$-skeleton is the boundary of $D^2$ or there exists a $1$-valent vertex in the interior.  [To find it, start with any edge in the interior.  Building a path inductively starting with this edge, we can either find some closed loop in the $1$-skeleton that contains this edge, which would contradict their only being one $2$-cell, or we can find a $1$-valent vertex.]  Repeatedly cancelling $1$-valent vertices with their corresponding edges removes all remaining $0$ and $1$-cells from the interior of $D^2$. 

\paragraph{{\bf Step 2.}}  Applying Step 1 to each of the $2$-cells of $\mathcal{E}_1$ leaves an $L$-compatible polygonal decomposition that has the same $2$-cells and same $1$-skeleton as $\mathcal{E}_1$.  The only remaining difference is that some of the $1$-cells of $\mathcal{E}_1$ are subdivided, and we can remove these subdivisions using (A).

\end{proof}






\section{Examples and Extensions}
\label{sec:Examples}

In Sections \ref{ssec:ExampleSphere} and \ref{ssec:ExampleST}, we compute some examples of the cellular DGAs of Legendrian spheres.  In Sections \ref{ssec:ConePoints} and \ref{ssec:MultipleCrossings}, we extend the definition of the DGA to allow for Legendrians with (non-generic) cone point singularities, and to allow the flexibility of having more the one crossing arc above a given $1$-cell.  We then apply this extensions to give an algebraic description of the DGA of surfaces obtained from spinning a $1$-dimensional Legendrian around an axis, allowing for the possibility that the axis intersects the Legendrian.  Finally, in Section \ref{sec:Lnsigma} we compute the DGA of a family of Legendrian spheres.  Many pairs of spheres from this family have linearized contact homology groups with the same ranks but can be distinguished with product operations.  

For some of these examples, the LCH has been (partially or completely) computed before with holomorphic curves consistent with our computations.
So such computations can be viewed as ``empirical evidence" that the cellular DGA is the same as LCH.

\subsection{Legendrian spheres}  
\label{ssec:ExampleSphere}

For these next two examples, the linearized contact homology has already been partially computed and used to show that they form
an infinite family of distinct Legendrian spheres with the same rotation class and Thurston-Bennequin invariant \cite[Proposition 4.10 and Theorem 4.11]{EkholmEtnyreSullivan05a}.

\begin{example}
\label{ex:OneCopy}
Recall the Legendrian $L_1$  pictured in Figure \ref{fig:2DExample}.  
We use the polygonal decomposition, $\mathcal{E}$, of $\pi_x(L_1)$ indicated in Figure \ref{fig:mathcalE}.

\begin{figure}
\labellist
\small
\pinlabel $B_3$ [br] at 108 158
\pinlabel $C_2$ [c] at 128 82
\pinlabel $A$ [l] at 158 124
\pinlabel $B_2$ [br] at 80 198
\pinlabel $B_4$ [b] at 216 126
\pinlabel $C_1$ [c] at 128 34
\pinlabel $B_1$ [l] at 264 132
\pinlabel $C_3$ [c] at 128 140
\endlabellist
\centerline{\includegraphics[scale=.6]{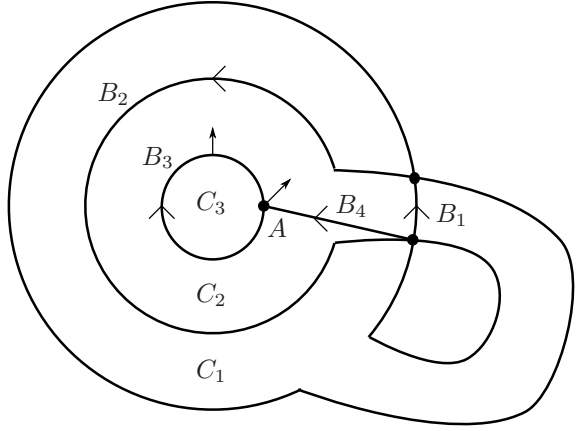}} 
\caption{To arrive at a polygonal decomposition of $\pi_x(L_1)$ that is compatible with $L_1$ we take the $0$-cells to consist of the codimension $2$ singularities of the base projection together with a point on the crossing locus.  The $1$-cells are the resulting pieces of the cusp and crossing locus together with 
an additional $1$-cell that subdivides the annular component of $\pi_x(L_1) \setminus \Sigma$ into a disk.
  Orientations are assigned to $1$-cells as pictured.  The arrows pointing out of $A$ and $B_3$ indicate a preferred ordering of sheets at these cells.}
\label{fig:mathcalE}
\end{figure}

Generators for the cellular DGA are determined once we assign an indexing set to the sheets above each cell of $\mathcal{E}$.  In all cases, if $L(e^i_\alpha)$ consists of $n$ sheets, we use $\{1, \ldots, n\}$ for the indexing set in such a way that the $z$-coordinates are non-increasing, $z(S_i) \geq z(S_{i+1})$.  This uniquely specifies the indexing for cells that are not contained in the crossing locus.  For such cells, it is convenient to specify the indexing of sheets by choosing a bordering $2$-cell, and ordering sheets as they appear above this $2$-cell.  In Figure \ref{fig:mathcalE} the choice is indicated with arrows pointing from the cells in the crossing locus to a neighboring $2$-cell.  

In the definition of the cellular DGA, the generators associated to a given cell are placed into several, possibly distinct matrices depending on the context, eg. generators for a $1$-cell may be placed into different matrices when computing the differentials of the two bordering $2$-cells.  
When working through examples, it is convenient to fix particular matrices containing the generators of each cell and then write all differentials in terms of these initial matrices.  For $L_1$, we form initial $2\times 2$ matrices $B_1, B_2$, and $C_1$, as well as $4\times 4$ matrices $A$, $B_3$, $B_4$, $C_2$ and $C_3$ by ordering rows and columns according to our indexing of sheets.  All of these matrices are strictly upper triangular.  We note that the matrices $A$ and $B_3$ have their $(2,3)$-entry equal to $0$ and that there are no generators associated to cells that appear on the boundary of $\pi_x(L_1)$.
We choose as $v_1 =v_0$  for both $C_1$ and $C_2$ the initial point of $B_4.$

Differentials are then determined by the matrix formulas
\begin{align}
\label{eq:DifferentialEx1}
& \partial A = A^2;  \\ 
\notag
& \partial B_1 = N (I+B_1) + (I+B_1) N = 0; \\
\notag
& \partial B_2  = N (I+B_2) + (I+B_2) N = 0; \\
\notag
& \partial C_1= N C_1 + C_1 N + (I+B_2)(I+B_1) + I = \left[\begin{array}{cc} 0 & b^1_{1,2}+ b^2_{1,2} \\ 0 & 0 \end{array} \right]; \\
\notag
& \partial B_3 = A(I+B_3) + (I+B_3) A; \\
\notag
& \partial B_4 = A ( I+B_4) + (I+B_4) \widehat{N}_{3,4};  \\
\notag
& \partial C_2 = \widehat{N}_{3,4} C_2 + C_2 \widehat{N}_{3,4} + (I+B_4)^{-1}(I+B_3)(I+B_4)(I+\widetilde{B}_{2,(1,2)})(I+\widetilde{B}_{1,(3,4)}) + I; \\
\notag
& \partial C_3 = (Q A Q) C_3 + C_3 (Q A Q) + (I+ Q B_3 Q) + I. 
\end{align}
Here, the notation $\widehat{X}_{k,l}$ (resp. $\widetilde{X}_{k,l}$) indicates the matrix obtained from $X$ by inserting the $2\times 2$ block $N$ (resp. $0$) along the diagonal at row and column $k$ and $l.$
In particular, $\widetilde{B}_{j, (k,l)} = \left(\widetilde{B}_j\right)_{k,l}.$
In addition, $Q = \left[ \begin{array} {cccc} 1 & 0 & 0 & 0 \\ 0 & 0 & 1 & 0 \\ 0 & 1 & 0 & 0 \\ 0 & 0 & 0 & 1 \end{array} \right]$ is the permutation matrix associated to the transposition $(2 \,\, 3)$.  Note that $A$ and $B_3$ are conjugated by $Q$ in the formula for $\partial C_3$ because the indexing of sheets above $A$ and $B_3$ was chosen to agree with the ordering above $C_2$ where sheets $2$ and $3$ appear in opposite order than they do above $C_3$.  

The cellular DGA of $L_1$ has $31$ generators.  While the differential of each generator is easily obtained from the above matrix formulas, calculations with this version of the DGA would probably be handled best by a computer.  However, applying Theorem \ref{thm:Alg} allows us to find a much smaller stable tame isomorphic quotient.
We essentially cancel matrices of generators in pairs, following a sequence of simple homotopy equivalences applied to the polygonal decomposition of $\pi_x(L_1).$

\begin{proposition}
\label{prop:OneCopy}
The DGA of $L_1$ is stable tame isomorphic to a DGA, $(\A_1,\partial)$, with $3$ generators $x,y$ and $z$ of degrees
\[
|x| = -1;  \quad |y| = |z| = 2
\]
with differentials
\[
\partial x = \partial y = 0; \quad \partial z = x y + y x.
\]
\end{proposition}
\begin{proof}
We will show how to cancel generators of the cellular DGA $(\A,\partial)$ using Theorem \ref{thm:Alg} in order to arrive at $(\A_1,\partial)$ as a stable tame isomorphic quotient.  Throughout we use the same notation for generators and their equivalence classes in various quotients of $(\A,\partial)$, and we use the symbol $\doteq$ to indicate that equality holds in the currently considered quotient.

Consider the equation
\[
\partial \framebox{$B_4$} = \framebox{$A$}+  A B_4 + (I+B_4) \widehat{N}_{3,4}.
\]
The $(2,3)$-entries of $A$ and $\widehat{N}_{3,4}$ are both $0$, so, using the notation $x := b^4_{2,3}$, we have $\partial x = 0$.  The remaining entries of $B_4$ can be inductively canceled with the entries of $A$, leaving
\begin{equation}
\label{eq:dB4}
B_4 \doteq x E_{2,3}; \quad A \doteq (I+x E_{2,3}) \widehat{N}_{3,4} (I+x E_{2,3}) = \left[ \begin{array} {cccc} 0 & 1 & x & 0 \\ 0 & 0 & 0 & x \\ 0 & 0 & 0 & 1 \\ 0 & 0 & 0 & 0 \end{array} \right].
\end{equation}


Proceeding in a similar manner, the equation
\begin{equation}
\label{eq:dC3}
\partial \framebox{$C_3$} = (Q A Q) C_3 + C_3 (Q A Q) + Q \framebox{$B_3$} Q
\end{equation}
allows us to cancel all entries of $C_3$, except for $y := c^3_{2,3}$, with the corresponding entries of $QB_3 Q$.  In the quotient, we have
\[
C_3 \doteq  y E_{2,3};  \quad 
B_3 \doteq A Q (y E_{2,3})Q + Q(y E_{2,3})Q A \doteq \left[ \begin{array} {cccc} 0 & x y & 0 & 0 \\ 0 & 0 & 0 & 0 \\ 0 & 0 & 0 & y x \\ 0 & 0 & 0 & 0 \end{array} \right]. 
\]

The equation for $\partial C_2$ becomes
\begin{equation}
\label{eq:dC2}
\partial C_2 \doteq \left[ \begin{array} {cccc} 0 & 0 & c^2_{2,3} & c^{2}_{1,3}+c^{2}_{2,4} \\ 0 & 0 & 0 & c^2_{2,3} \\ 0 & 0 & 0 & 0 \\ 0 & 0 & 0 & 0 \end{array} \right] + \left[ \begin{array} {cccc} 1 & xy & xyx & 0 \\ 0 & 1 & 0 & x yx  \\ 0 & 0 & 1 & yx \\ 0 & 0 & 0 & 1 \end{array} \right] \left[ \begin{array} {cccc} 1 & b^1_{1,2} & 0 & 0 \\ 0 & 1 & 0 & 0 \\ 0 & 0 & 1 & b^2_{1,2} \\ 0 & 0 & 0 & 1 \end{array} \right] + I.
\end{equation}
We can thus cancel as indicated in the following
\begin{align*}
\partial \framebox{$c^2_{1,2}$}  \doteq 
x y  +\framebox{$b^1_{1,2}$}, \quad
\partial \framebox{$c^2_{1,3}$} \doteq 
\framebox{$c^2_{2,3}$}+ \cdots, \quad 
\partial \framebox{$c^2_{1,4}$} \doteq 
\framebox{$c^2_{2,4}$} + \cdots, \quad 
\partial \framebox{$c^2_{3,4}$} \doteq 
y x +\framebox{$b^2_{1,2}$}.
\end{align*}
The first and last equations imply $b^1_{1,2} \doteq xy$ and $b^2_{1,2} \doteq yx.$  
Finally,
\begin{equation}
\label{eq:dC1}
{\partial c^1_{1,2}} = {b^1_{1,2}} + b^2_{1,2}.
\end{equation}
We have canceled all generators except for $x,y$ and $z:= c^1_{1,2}.$
Since $\partial x = 0 =\partial y,$ the result follows.

\end{proof}

\end{example}

\begin{example}
\label{ex:kCopies}

Let $L_k$ denote $k$ copies of $L_1$ from Example \ref{ex:OneCopy} cusp-connect summed as in Figure \ref{fig:k2DExample}.
Each cell $A, B_1, B_2, B_3,B_4, C_2,C_3$ in Example \ref{ex:OneCopy} now has $k$ copies which we indicate with a superscript,
such as $B_1^1, \ldots, B_1^k .$
Note that there is only one $C_1$ cell.
We add a corresponding second superscript when labeling the Reeb chords;
for example, the unique non-zero entry in $B_2^j$ is $b_{1,2}^{2,j}.$
Except for $\partial C_1,$ the differentials are exactly as in (\ref{eq:DifferentialEx1}) with appropriate subscripts.
For example, $\partial B_4^j = A^j(I + B_4^j) + (I + B_4^j) \widehat{N}_{3,4}.$
For $C_1,$ if we set $v_0 = v_1$ equal to the start point of cell $B_4^1,$ then
$$
\partial C_1 = \widehat{N}_{3,4} C_1 + C_1 \widehat{N}_{3,4} + (1+B_2^1)(1+B_1^k)(1+B_2^k) \cdots (1+B_1^2)(1+B_2^2)(1+B_1^1) + I.
$$

\begin{proposition}
\label{prop:kCopies}
The DGA of $L_k$ is stable tame isomorphic to a DGA, $(\A_k,\partial)$, with $2k+1$ generators $x_1,y_1, \ldots, x_k,y_k, z$ of degrees
\[
|x_j| = -1;  \quad |y_j| = |z| = 2
\]
with differentials
\[
\partial x_j = \partial y_j = 0; \quad \partial z = \sum_{j=1}^k (x_j y_j + y_j x_j).
\]
\end{proposition}
\begin{proof}
We do the same cancelations as in (\ref{eq:dB4}), (\ref{eq:dC3}) and (\ref{eq:dC2}), labeling the remaining generators $x_j:=b^{4,j}_{2,3}, y_j = c^{3,j}_{2,3}.$
Relations in (\ref{eq:dC2}) imply $b^{1,j}_{1,2} \doteq x_j y_j$ and $b^{2,j}_{1,2} \doteq y_j x_j.$
The $k$-copy version of (\ref{eq:dC1}) is
$\partial c^1_{1,2} = \sum_{j=1}^k (b^{1,j}_{1,2} + b^{2,j}_{1,2}).$
Setting $z:= c^1_{1,2},$ the result follows.

\end{proof}

\end{example}

\begin{figure}
\labellist
\small
\endlabellist
\centerline{\includegraphics[scale=.25]{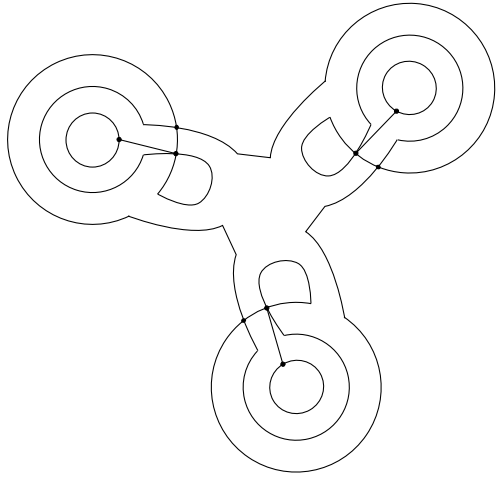}} 
\caption{The Legendrian $L_k$ with $k=3$.}
\label{fig:k2DExample}
\end{figure}

As alluded to before, a partial computation for these examples appears in \cite[Proposition 4.10 and Theorem 4.11]{EkholmEtnyreSullivan05a}.
In particular, using $J$-holomorphic disks, the paper computes that the degree $-1$ subspace of the linearized Legendrian contact homology is $\mbox{Lin}LCH_{-1}(L_k) = (\Z/2)^k.$
[Linearized LCH was introduced in \cite{Chekanov02}. Briefly, an augmentation, if it exists, is a DGA-morphism $\epsilon:(\mathcal{A}, \partial) \rightarrow (\Z_2, 0).$ Define a graded algebra morphism $\phi_\epsilon: \mathcal{A} \rightarrow \mathcal{A}$ on the generators
by $\phi_\epsilon(x) = x + \epsilon(x).$ Define $d_\epsilon(x)$ as the sum of words of length one in $\phi_\epsilon \partial \phi_\epsilon^{-1}(x).$
One can check that $d_\epsilon^2=0.$ The linearized LCH is $\mbox{Ker}(d_\epsilon)/\mbox{Im}(d_\epsilon).$]
This computation agrees with ours. 
To see this consistency, note that since there are no generators of grading 0, there is a unique (graded) augmentation.
Thus, the  degree $-1$ subspace of the linearized cellular homology is freely generated by $x_1, \ldots, x_k,$
and so isomorphic to $(\Z_2)^k.$

\subsection{An example with swallowtail points}  
\label{ssec:ExampleST}

Consider the Legendrian, $L$, which is pictured in Figure \ref{fig:Ex2Base} together with a compatible decomposition of $\pi_x(L)$.  As in the earlier example, we fix upper-triangular matrices $A_1, A_2, B_1, \ldots, B_5, C_1, \ldots, C_4$ by placing generators of $(\mathcal{A}, \partial)$ into rows and columns according to the ordering of sheets above each cell.  For the $1$-cell $B_5$ that contains the crossing locus, we use the ordering of sheets as they appear above $C_4$.  (This choice of ordering is indicated in Figure \ref{fig:Ex2Base} by the arrow pointing from $B_5$ to $C_4$.)  

The differential of any generator can be computed using the matrix formulas of Section \ref{ssec:defsw}.  Here we write out $\partial C_3$ and $\partial B_5$ explicitly.

To compute $\partial C_3$ we would consider the polygon  $Q$ pictured in Figure \ref{fig:Ex2Base} that arises from considering the boundary of $C_3$ along with the location of the $S$ and $T$ decorations.  Note that, for $L$, the matrices $S_1, T_1$ and $S_2, T_2$ associated to the swallowtail points $A_1$ and $A_2$ are all equal to $I + E_{2,3}$.  Taking the upper vertex of $Q$ for $v_0$ and $v_1$, with the path around $Q$ from $v_0$ to $v_1$ chosen to be counter-clockwise, we have
\begin{equation}   \label{eq:pC3}
\partial C_3= \widehat{N}_{3,4} C_3 + C_3 \widehat{N}_{3,4} + I +    (I+\widetilde{B}_{4,(3,4)})S_2(I+Q B_5Q)  S_1(I+\widetilde{B}_{2,(1,2)})
\end{equation}
 where $Q$ is the permutation matrix of the transposition $(2,3)$ which appears because the labeling of sheets $S_2$ and $S_3$ above $B_5$ and $C_3$ is opposite.  (Notations are as in Example \ref{ssec:ExampleSphere}.)

We now consider $\partial B_5$. As the ordering of rows and columns of $B_5$ agrees with the ordering of sheets above the $T_1$ and $T_2$ corners at the swallow tail points, we have
\begin{equation} \label{eq:pB5}
\partial B_5 = \left( A_2 \right)_T (I+B_5) + (I+ B_5) \left( A_1 \right)_T, 
\end{equation}
where
\[
\left( A_1 \right)_T  = (I+E_{2,3}) \left[ \begin{matrix} 0 & 1 & & \\  &  0 & & \\ & &0 & a^1_{1,2} \\ & & & 0 \end{matrix} \right](I+E_{2,3}) \, ; \quad 
\left( A_2 \right)_T = (I+E_{2,3}) \left[ \begin{matrix} 0 & a^2_{1,2} & & \\  &  0 & & \\ & &0 & 1 \\ & & & 0 \end{matrix} \right](I+E_{2,3})
\]

The cellular DGA $(\mathcal{A}, \partial)$ for $L$ has $25$ generators, but once again appropriate applications of Theorem \ref{thm:Alg} produce a much smaller stable tame isomorphic quotient.

\begin{proposition}
The cellular DGA of $L$ is stable tame isomorphic to a DGA with a single generator $x$ with $|x|=2$ and $\partial x = 0$.
\end{proposition}
 
\begin{remark}  This agrees with the DGA of the standard $2$-dimensional Legendrian unknot, $U$.  In fact, it can be shown that $L$ and $U$ are Legendrian isotopic. 
\end{remark}

\begin{proof}
We first compute
\[
\framebox{$\partial b^2_{1,2}$} = \framebox{$a^1_{1,2}$} +1 
\]
and apply Theorem \ref{thm:Alg} to cancel $b^2_{1,2}$ with $a^1_{1,2}$ so that $b^2_{1,2} \doteq 0$ and $a^1_{1,2} \doteq 1$.  (Again, $\doteq$ denotes equality in the presently considered quotient.)
Independent of the choices of $v_0$ and $v_1$ for the $2$-cell $C_1$, we have
\[
\framebox{$\partial c^1_{1,2}$} = b^1_{1,2} + b^2_{1,2} \doteq \framebox{$b^1_{1,2}$}
\]
so that we may cancel $c^1_{1,2}$ with $b^1_{1,2}$.

Next, a parallel sequence of applications of Theorem \ref{thm:Alg} results in
\[
c^2_{1,2} \doteq b^3_{1,2} \doteq b^4_{1,2} \doteq 0 \quad \mbox{and} \quad a^2_{1,2} \doteq 1.
\]

For computing $\partial C_4$, we take $v_0$ and $v_1$ to be the vertex that is the common endpoint of $B_1$ and $B_3$ and choose the path around the boundary of $C_4$ to be clock-wise.  We then have
\[
\partial C_4 = \left[\begin{matrix} 0 & 1 & & \\ & 0 & & \\  & & 0 & 1 \\ & & & 0 \end{matrix} \right] C_4 + C_4 \left[\begin{matrix} 0 & 1 & & \\ & 0 & & \\  & & 0 & 1 \\ & & & 0 \end{matrix} \right] + I +    (I+\widetilde{B}_{3,(3,4)})T_2(I+ B_5)  T_1(I+\widetilde{B}_{1,(1,2)}) \doteq
\] 
\[
\left[\begin{matrix} 0 & 0 & c^4_{2,3} & c^4_{24}+c^4_{1,3} \\ &  & & c^4_{2,3} \\  & &  & 0 \\ & & & 0 \end{matrix} \right] + I + \left[\begin{matrix} 1 &  &  &  \\ & 1  & 1 &  \\  & & 1 &  \\ & & & 1 \end{matrix} \right]  \left[\begin{matrix} 1 & b^5_{1,2}  & b^5_{1,3}  & b^5_{14}  \\ & 1  & 0 & b^5_{24} \\  & & 1 & b^5_{34}  \\ & & & 1 \end{matrix} \right] \left[\begin{matrix} 1 &  &  &  \\ & 1  & 1 &  \\  & & 1 &  \\ & & & 1 \end{matrix} \right]. 
\]
In particular, we can use the following entry-by-entry evaluations to cancel all of the $c^4_{i,j}$ along with $b^{5}_{1,2}$ and $b^{5}_{34}$ as indicated:
\begin{align*}
 \framebox{$\partial c^4_{1,2}$} \doteq \framebox{$b^5_{1,2}$}\,; \, \quad \quad \quad \quad & \framebox{$\partial c^4_{34}$} \doteq \framebox{$b^5_{34}$}\,; \\
 \framebox{$\partial c^4_{1,3}$} \doteq \framebox{$c^4_{2,3}$}+ x_b\,; \quad \quad  & \framebox{$\partial c^4_{14}$} \doteq \framebox{$c^4_{24}$}+ c^4_{1,3} + y_b \doteq \framebox{$c^4_{24}$} + y_b\,;
\end{align*}
where the terms $x_b$ and $y_b$ are in the subalgebra generated by entries of $B_5$.  (Hence, $x_b$ and $y_b$ satisfy the condition of the term $w$ from the statement of Theorem \ref{thm:Alg}.)

At this point, equation (\ref{eq:pC3}) has simplified to 
\[
\partial C_3 = \left[\begin{matrix} 0 & 0 & c^3_{2,3} & c^3_{24}+c^3_{1,3} \\ &  & & c^3_{2,3} \\  & &  & 0 \\ & & & 0 \end{matrix} \right] + I + \left[\begin{matrix} 1 &  &  &  \\ & 1  & 1 &  \\  & & 1 &  \\ & & & 1 \end{matrix} \right]  \left[\begin{matrix} 1 & b^5_{1,3}  & 0  & b^5_{14}  \\ & 1  & 0 & \\  & & 1 & b^5_{24}  \\ & & & 1 \end{matrix} \right] \left[\begin{matrix} 1 &  &  &  \\ & 1  & 1 &  \\  & & 1 &  \\ & & & 1 \end{matrix} \right]. 
\]
We cancel
\begin{align*}
 \framebox{$\partial c^3_{1,2}$} \doteq \framebox{$b^5_{1,3}$}\,;  \quad \quad & \framebox{$\partial c^3_{34}$} \doteq \framebox{$b^5_{24}$}\,; \\
 \framebox{$\partial c^3_{1,3}$} \doteq \framebox{$c^3_{2,3}$}\,; \quad \quad  & \framebox{$\partial c^3_{14}$} \doteq \framebox{$c^3_{24}$}+ b^5_{14}\,.
\end{align*}
The only remaining generator is $b^5_{14}$ which has $|b^5_{14}|=2$.  In the current quotient, $b^5_{1,2} \doteq b^5_{13} \doteq b^5_{24} \doteq b^5_{34} \doteq 0$, and   $a^1_{1,2} \doteq a^2_{1,2} \doteq 1$, and it follows that (\ref{eq:pB5}) simplifies to $\partial B_5 \doteq 0$.  In particular,
\[
\partial b^5_{14} = 0.
\]
\end{proof}


\subsection{Cone points}
\label{ssec:ConePoints}
In this subsection, we extend the definition of the cellular DGA to allow for Legendrians whose front projections have cone point singularities.

The standard cone $x_1^2+x_2^2=z^2$ is the (non-generic) front projection of a Legendrian cylinder in $J^1\R^2$.     A point $p$ on the front projection of a Legendrian $L \subset J^1S$ is called a {\bf cone point} if there is a diffeomorphism of a neighborhood of $p$ in $S \times \R$ onto a neighborhood of $(0,0,0) \in \R^3$ that takes the front projection of $L$ to the standard cone.  The inverse image in $L$ of a cone point singularity is an $S^1$, so that neighborhoods of cone points in $L$ are topologically cylinders.  

Consider a Legendrian $L \subset J^1S$ with generic base and front projections except for the presence of finitely many cone points whose base projections are disjoint from the image of cusps and crossings.  Write $\Sigma_{\mathit{cone}} \subset S$ and $\Sigma \subset S$ for the base projection of cone points, and the remaining singular set of $L$.   Let $\mathcal{E}$ be a polygonal decomposition of $\pi_x(L) \subset S$ that 
 contains $\Sigma$ (resp. $\Sigma_{\mathit{cone}}$) in its $1$-skeleton (resp. $0$-skeleton).

We associate a  DGA, $(\mathcal{A},\partial)$, to $L$ as in Section \ref{sec:DefDGA} using $\mathcal{E}$ with the following modification at a cone point.  Assume that near a zero cell, $e^0_\alpha \in \Sigma_{\mathit{cone}}$, sheets of $L$ are labeled $S_1, \ldots, S_n$ with the cone point connecting sheets $S_k$ and $S_{k+1}$.  Note that a Maslov potential $\mu$ on $L$ must have $\mu(S_k) = \mu(S_{k+1}) +1$; see the desingularization of the cone point in Figure \ref{fig:Octagon}.
\begin{enumerate}
\item  The generators associated to $e^0_\alpha$ are $a^\alpha_{i,j}$ with $1 \leq i < j \leq n$ and $(i,j) \neq (k,k+1)$ with
\[
|a^\alpha_{i,j}| = \mu(S_i)-\mu(S_{j})-1.
\]  The corresponding upper triangular matrix, $A$, (the $(k,k+1)$-entry is $0$) satisfies $\partial A = A^2$.  This same matrix is used when $e^0_\alpha$ occurs as an initial or terminal vertex of a bordering $1$-cell or $2$-cell.  
\item  Choose a $2$-cell $e^2_\beta$ that borders $e^0_\alpha$, and label one of the occurrences of $e^0_\alpha$ as a vertex along $\partial e^2_\beta$ with a $U$.  When computing $\partial C$ for $e^2_\beta$, we add an additional edge to $\partial e^2_\beta$ at this vertex, and insert the matrix
\begin{equation}
\label{eq:ConePointU}
U = I +\sum_{\ell < k} a_{\ell, k+1} E_{\ell, k} + \sum_{k+1< \ell} a_{k,\ell} E_{k+1,\ell}
\end{equation}
into the product of edges that appears in $\partial C$.
\end{enumerate}

\begin{proposition} \label{prop:coneDGA} The DGA $(\mathcal{A},\partial)$ is equivalent to the cellular DGA of $L$.
\end{proposition}
\begin{proof}
 By a small Legendrian isotopy, $L$ becomes front and base generic with the cone point replaced by $4$ swallowtail points, connected by $4$-cusp edges, and $2$-crossing arcs, cf. \cite[Section 3.1]{DimitroglouRizell11}.    In the base projection, the image of the cusp edges bounds a square, $X$, whose vertices are swallowtail points.  Two opposite corners of the square are upward swallowtails on the lower sheet of the cone point, and the other two opposite corners are downard swallowtails on the upper sheet.  Crossing arcs appear as diagonals of the square.  See Figure \ref{fig:Octagon}.  There are $n+2$ sheets of $L$ above the interior of $X$.

Produce from $\mathcal{E}$ an $L$-compatible polygonal decomposition (for this generic perturbation of $L$) by replacing the single $0$-cell, $C$, that was the cone point with the natural decomposition of the square $X$ into 4 triangles from Figure \ref{fig:Octagon}.  Modify the $1$-cells that had previously had endpoints at $C$ so that their endpoints are vertices of $X$ in such a way that the corner labeled with $U$ is replaced with some sequence of edges around the square starting with $B_0$.  

To prove the proposition we produce the cone point DGA from the cellular DGA via repeated application of Theorem \ref{thm:Alg}.  As a first step, we orient the boundary edges of $X$ consistently, and then cancel the generators associated to $3$ of the boundary edges of $X$ with the generators of their terminal vertex.  The generators and differentials of these edges are simply
\[
\partial \framebox{$B$} = \framebox{$A_+$}(I+B) + (I+B) A_-
\]
where all $B$ and $A$ matrices are $n\times n$ with diagonal entries given by the appropriate $b_{i,j}$ or $a_{i,j}$.  (There are no $0$ or $1$ entries above the diagonal.)
In the resulting quotient, generators of the three $1$-cells have become $0$, and generators associated to the $4$ vertices of $X$ are now equal.  We use subscripts $A_0$ and $B_0$ for the matrices associated to the vertices of $X$ and the $1$ non-zero edge.  See Figure \ref{fig:Octagon}.

\begin{figure}

\labellist
\small
\pinlabel $S_k=S_{k+1}$ [l] at 230 212
\pinlabel $S_{k+1}=S_{k+2}$ [l] at 230 74
\pinlabel $A_0$ [l] at 248 124 
\pinlabel $0$ [bl] at 180 168
\pinlabel $A_0$ [b] at 124 248 
\pinlabel $0$ [br] at 72 168
\pinlabel $A_0$ [t] at 124 -2 
\pinlabel $0$ [tl] at 180 76
\pinlabel $A_0$ [r] at -2 124 
\pinlabel $B_0$ [tr] at 72 76
\endlabellist
\centerline{\includegraphics[scale=.5]{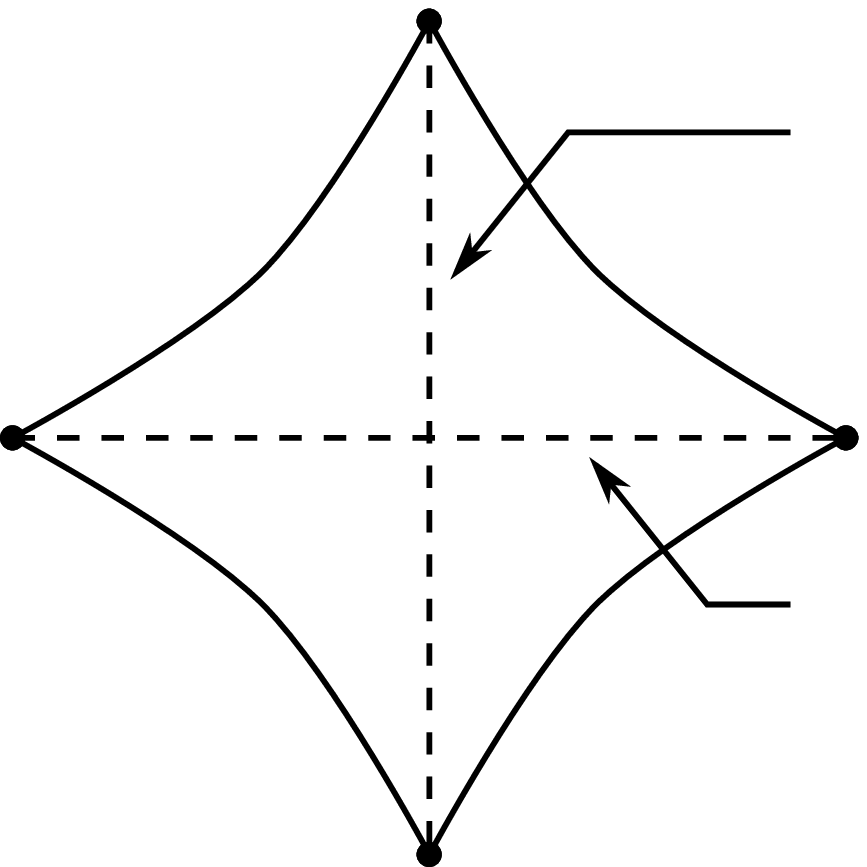} \quad \quad \quad \quad \quad \quad \quad \quad \quad \quad 
\labellist
\small
\pinlabel $T_1$ [r] at -2 98
\pinlabel $S_1$ [r] at -2 144
\pinlabel $S_3$ [l] at 244 98 
\pinlabel $T_3$ [l] at 244 144
\pinlabel $S_2$ [t] at 98 -2 
\pinlabel $T_2$ [t] at 144 -2
\pinlabel $S_4$ [b] at 144  244 
\pinlabel $T_4$ [b] at 98 244
\pinlabel $B_1$ [t] at 54 118 
\pinlabel $B_2$ [l] at 124 52
\pinlabel $B_3$ [b] at 184  124 
\pinlabel $B_4$ [r] at 118 182
\pinlabel $C_1$ at 70 66 
\pinlabel $C_2$  at 178 66
\pinlabel $C_3$  at 178 176 
\pinlabel $C_4$  at 70 176
\pinlabel $A_1$ [tl] at 126 116
\pinlabel $B_0$ [tr] at 32 44
\endlabellist
 \includegraphics[scale=.5]{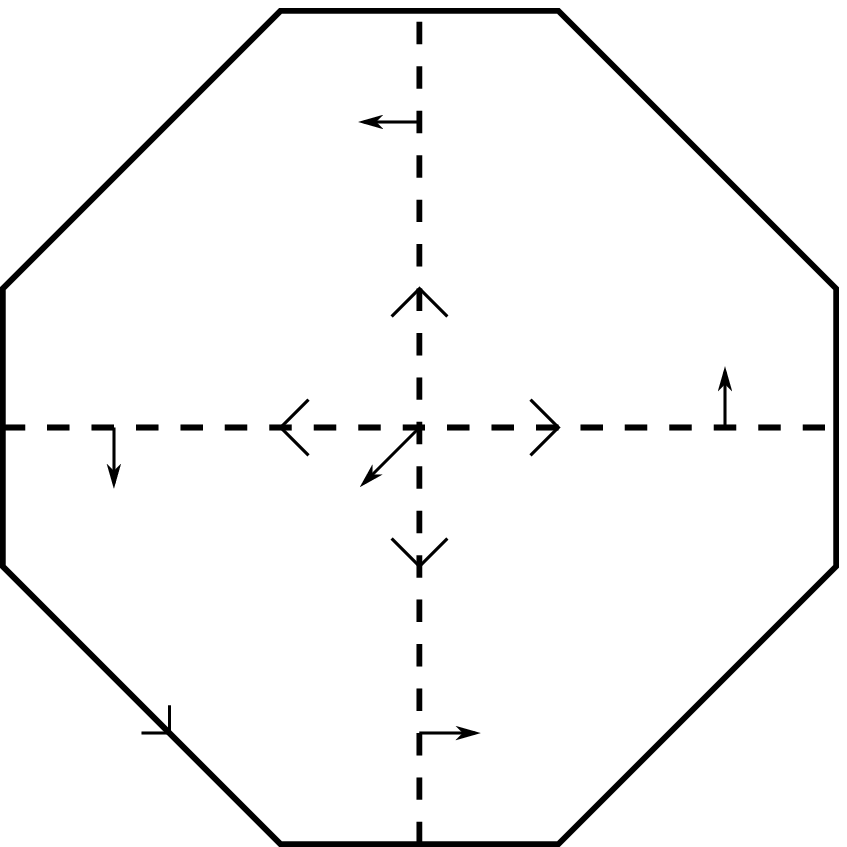} } 
\caption{(left) The square $X$ in the base projection of a resolved cone point.  (right) Notation for matrices associated to cells of $X$ and corners at swallowtail points.}
\label{fig:Octagon}
\end{figure}

Notations for cells in the interior of $X$ are indicated in Figure \ref{fig:Octagon}.  Above the $1$-cells and $0$-cells in the interior of $X$, we label sheets from $1$ to $n+2$ as they appear above the neighboring $2$-cell indicated by the small arrows in Figure \ref{fig:Octagon}.  This ordering specifies the notation for generators.  \emph{In this proof, the notation for matrices is fixed,} so that, eg., the matrix $B_i$ always has rows and columns ordered according to the ordering of sheets above $C_i$, and $A_1$ is always ordered as above $C_1$.  The crossing locus results in the following $0$ entries above the diagonal:

\centerline{
\begin{tabular}{c|c}  
Matrix & $0$ entries \\
\hline 
$B_1, B_3$  & $(k+2,k+3)$ \\
$B_2, B_4$  & $(k,k+1)$ \\
$A_1$ &  $(k,k+1)$ and $(k+2,k+3)$
\end{tabular}
}
  
As in Section \ref{sssec:sw2cell}, for computing differentials of $2$-cells containing swallowpoints, we add additional edges at corners labeled with $S$ or $T$.  We notate the matrices associated to these edges
as illustrated in Figure \ref{fig:Octagon}.  These matrices are
\[
\begin{array}{ll} \displaystyle S_1 = S_3 = I+E_{k+2,k+3} + \sum_{l < k+1} a^0_{l,k+1} E_{l,k+1}, \quad & T_1=T_3 = I+ E_{k+2,k+3} \\
 & \\
\displaystyle S_2 = S_4 = I+E_{k,k+1} + \sum_{k+2<l} a^0_{k,l-2} E_{k+2,l}, \quad & T_2=T_4 = I + E_{k,k+1}, \\
\end{array}
\]

We denote the $A_T$ matrices (as in (\ref{eq:Atdef}) from \ref{sec:matrices}) associated to the endpoints of the $1$-cell $B_i$ as 
\[
(A_0)_{T_i} = T_i \widehat{(A_0)}_{(k+1,k+2)} T_i
\]

For the $2$-cells $C_1$ (resp. $C_i$, $2 \leq i \leq 4$), we choose initial and terminal vertices to be the endpoints of $B_0$ (resp. of $B_i$).   We have differentials
\begin{align*}
\partial B_1 = & (A_0)_{T_1}(I+B_1) + (I+B_1) A_1, \\
\partial C_1 = & \widehat{(A_0)}_{k+1,k+2} C_1 + C_1 \widehat{(A_0)}_{k+1,k+2} + (I+\widehat{B_0}_{k+1,k+2}) + S_2 Q_{(0\,1)}(I+B_2)Q_{(0\,1)} (I+B_1)^{-1} T_1, \\
\partial C_2 = & (A_0)_{T_2}C_2 + C_2 Q_{(0\,1)}A_1 Q_{(0\,1)} + (I+B_2) + T_2 S_3 Q_{(2\,3)}(I+B_3)Q_{(2\,3)}, \\ 
\partial C_3 = & (A_0)_{T_3}C_3 + C_3 Q_{(0\,1) (2\,3)}A_1 Q_{(0\,1)(2\,3)} + (I+B_3) + T_3 S_4 Q_{(0\,1)}(I+B_4)Q_{(0\,1)}, \\
\partial C_4 = & (A_0)_{T_4}C_4 + C_4 Q_{(2\,3)}A_1 Q_{(2\,3)} + (I+B_4) + T_4 S_1 Q_{(2\,3)}(I+B_1)Q_{(2\,3)}, \\
\end{align*}
where the matrices $Q_{(0 \, 1)}$ and $Q_{(2\,3)}$ are permutation matrices for the transpositions $(k \, k+1)$ and $(k+2 \, k+3)$, and $Q_{(0\,1)(2\,3)}=Q_{(0 \, 1)} Q_{(2\,3)}$.

We begin to cancel generators.  First, observe that
\[
\partial \framebox{$c^4_{k,k+1}$} = 1+ a^0_{k,k+1} + \framebox{$b^1_{k,k+1}$},
\]
and use Theorem \ref{thm:Alg} to cancel the boxed generators.  [The product of strictly upper triangular matrices has all $(i,i+1)$-entries $0$, so the $(k,k+1)$-entry of $\partial C_4$ is the sum of the $(k,k+1)$-entries from $I+B_4$, $T_4$, $S_1$, and $Q_{(2\,3)}(I+B_1)Q_{(2\,3)}$.]
The remaining non-zero entries of $C_4$ are in correspondence with those of $B_4$, so we cancel
\[
\partial \framebox{$C_4$} = \framebox{$B_4$}+ \cdots. 
\]
[The condition on the ordering of generators required to apply Theorem \ref{thm:Alg} is verified as in proof of Theorem \ref{thm:sub2cell}.]
Note that in the quotient, we have 
\[
b^4_{k+2,k+3} \doteq 1,
\]
so we can compute and cancel
\[
\partial \framebox{$c^3_{k+2,k+3}$} = 1+ a^0_{k,k+1}+ b^4_{k+2,k+3} \doteq \framebox{$a^0_{k,k+1}$}.
\]
Record that
\[
a^0_{k,k+1} \doteq 0,   \Rightarrow b^1_{k,k+1} \doteq 1.
\]
With non-zero entries of $C_3$ and $B_3$ now in correspondence, cancel
\[
\partial \framebox{$C_3$} = \framebox{$B_3$}+ \cdots.
\]
This gives $b^3_{k,k+1} \doteq 1$, so that
\[
\partial c^2_{k,k+1} = 1+a^0_{k,k+1}+b^3_{k,k+1} \doteq 0. 
\]
Hence, we can modify our usual orderings of generators (arising as in Lemma \ref{lem:precA}) by moving $c^2_{k,k+1}$ to be the first generator.   
This allows us to cancel all entries of $C_2$ \emph{other than $c^2_{k,k+1}$} (the $(k,k+1)$ entry of $B_2$ is $0$) by
\[
\partial \framebox{$C_2$} = \framebox{$B_2$} + \cdots.
\]
With $b^1_{k,k+1}$ already canceled, the remaining generators in $B_1$ are in correspondence with those from $A_1$; cancel
\[
\partial \framebox{$B_1$} = \framebox{$A_1$} + \cdots.
\]

To summarize where we are, the only remaining generators associated to the cells of $X$ are $c^2_{k,k+1}$ and the entries of $B_1$, $C_1$, and $A_0$ except for $a^0_{k,k+1}$ which has $a^0_{k,k+1} \doteq 0$.  We have 
\[
C_4 \doteq 0,  \quad C_3 \doteq 0, \quad C_2 \doteq c^2_{k,k+1} E_{k,k+1}, \quad (I+B_1) \doteq (I+ E_{k,k+1}).
\]
Using $\partial C_4 \doteq 0$ gives
\[
I+B_4 \doteq T_4 S_1 Q_{(2\,3)}(I+E_{k,k+1})Q_{(2\,3)} = S_1.
\]
Then, $\partial C_3 \doteq 0$ gives
\[
I+B_3 \doteq T_3 S_4 Q_{(0\,1)}(I+B_4)Q_{(0\,1)} \doteq T_3S_4 Q_{(0\,1)}S_1Q_{(0\,1)} \doteq 
\]
\[
(I+E_{k+2,k+3})(I+E_{k,k+1} + \sum_{k+3 <l} a^0_{k,l-2} E_{k+2,l}) ( I + E_{k+2,k+3} + \sum_{l<k} a^0_{l,k+1}E_{l,k}) =
\]
\[
I+E_{k,k+1} + \sum_{k+3 <l} a^0_{k,l-2} E_{k+2,l} + \sum_{l<k} a^0_{l,k+1}E_{l,k}.
\]
[The range of $l$ in the summations is different from in the definitions of $S_1$ and $S_4$ because $a^0_{k,k+1} \doteq 0$.]
Next, compute
\[
I+B_2 \doteq F + G
\]
where
\begin{align*}
F =&  T_2 S_3 Q_{(2\,3)}(I+B_3)Q_{(2\,3)}  \\ 
 \doteq & (I+E_{k,k+1})(I+E_{k+2,k+3} + \sum_{l < k} a^0_{l,k+1} E_{l,k+1})(I+E_{k,k+1} + \sum_{k+3 <l} a^0_{k,l-2} E_{k+3,l} + \sum_{l<k} a^0_{l,k+1}E_{l,k}) \\
 =& I + E_{k+2,k+3} + \sum_{l<k} a^0_{l,k+1} (E_{l,k} + E_{l,k+1}) + \sum_{k+3<l}a^0_{k,l-2}(E_{k+2,l}+E_{k+3,l}), 
\end{align*}
and
\begin{align*}
G=& (A_0)_{T_2}C_2 + C_2 Q_{(0\,1)}A_1 Q_{(0\,1)}. 
\end{align*}
Using that $\partial B_1 = 0$ and $B_1 \doteq (I+E_{k,k+1})$ shows that 
\[
A_1 \doteq (I+B_1)^{-1}(A_0)_{T_1}(I+B_1) \doteq (I + E_{k,k+1}+ E_{k+2,k+3})\widehat{A_0}_{(k+1,k+2)}(I + E_{k,k+1}+ E_{k+2,k+3}),
\]
so we compute
\begin{align*}
C_2 Q_{(0\,1)}A_1 Q_{(0\,1)}\doteq&  (c^2_{k,k+1} E_{k,k}) A_1 Q_{(0\,1)} \\
 \doteq& \left(\sum_{k+1 < l} c^2_{k,k+1} a^0_{k,l} E_{k,l+2}\right) + c^2_{k,k+1} E_{k,k+2} + c^2_{k,k+1} E_{k,k+3}; \mbox{and} \\
(A_0)_{T_2}C_2 \doteq & \sum_{l<k} a^0_{l,k} c^2_{k,k+1}E_{l,k+1}.
\end{align*}
Combining the previous computations gives
\begin{align*}
I+B_2 \doteq& I + E_{k+2,k+3} + \sum_{l<k} a^0_{l,k+1} (E_{l,k} + E_{l,k+1}) + \sum_{k+3<l}a^0_{k,l-2}(E_{k+2,l}+E_{k+3,l}) + \\
 &  c^2_{k,k+1} E_{k,k+2} + c^2_{k,k+1} E_{k,k+3} +\sum_{k+1 < l} c^2_{k,k+1} a^0_{k,l} E_{k,l+2} + \sum_{l<k} a^0_{l,k} c^2_{k,k+1}E_{l,k+1}.
 \end{align*}
Observe that 
 $Q_{(0\,1)}(I+B_2) Q_{(0\,1)}$ is upper triangular with $(k+1,k+2)$-entry $c^2_{k,k+1}$.  

Next, since the $(k+1,k+2)$-entries of $S_2$, $(I+B_1)^{-1}$, and $T$ are all $0$, the formula for $\partial C_1$ shows
\[
\partial\framebox{$c^1_{k+1,k+2}$} = \framebox{$c^2_{k,k+1}$}.
\]  
After making this cancellation, 
\[
Q_{(0\,1)}(I+B_2)Q_{(0\,1)} \doteq I + E_{k+2,k+3} + \sum_{l<k} a^0_{l,k+1} (E_{l,k} + E_{l,k+1}) + \sum_{k+3<l}a^0_{k,l-2}(E_{k+2,l}+E_{k+3,l}),
\]
and another careful matrix multiplication shows
\[
S_2 Q_{(0\,1)}(I+B_2)Q_{(0\,1)} (I+B_1)^{-1} T_1 = I+ \sum_{l<k} a^0_{l,k+1}E_{l,k} + \sum_{k+3<l} a_{k,l-2} E_{k+3,l}. 
\]
This matrix has all above diagonal entries in the $k+1$ and $k+2$ rows and columns equal to $0$.  Thus, the entries of $\partial C^1$ in these rows and columns agree with the entries of $\widehat{(A_0)}_{k+1,k+2} C_1 + C_1 \widehat{(A_0)}_{k+1,k+2}$, and we have, for $i<k+1$ and $k+2 < j$, 
\[
\partial \framebox{$c^1_{i,k+2}$} \doteq \sum_{i<m<k+1} a^0_{i,m}c^1_{m,k+2} + \framebox{$c^1_{i,k+1}$};  \quad  \partial \framebox{$c^1_{k+1,j}$} \doteq \framebox{$c^1_{k+2,j}$} + \sum_{k+2<m<j} c^1_{k+1,m} a^0_{m-2,j-2}. 
\]
Here, we cancel beginning with $i=k$ (resp. $j=k+3$), and then with $i$ decreasing (resp. $j$ increasing), so that inductively when we cancel the equation  simplifies to 
\[
\partial \framebox{$c^1_{i,k+2}$} \doteq \framebox{$c^1_{i,k+1}$};  \quad  \partial \framebox{$c^1_{k+1,j}$} \doteq \framebox{$c^1_{k+2,j}$}. 
\]
Thus, all of the $c^1_{i,k+1}$, $c^1_{i,k+2}$, $c^1_{k+1,j}$, $c^1_{k+1,j}$, and $c^1_{k+1,k+2}$ are $0$ in the quotient. 
The remaining generators from $C^1$ are in correspondence with the $B_0$ generators, and we cancel
\[
\partial \framebox{$C_1$} = \framebox{$\widehat{B_0}_{(k+1,k+2)}$}+ \cdots.
\]
In the quotient, we have $C_1 \doteq 0$ and
\[
I+\widehat{B_0}_{(k+1,k+2)} = I+ \sum_{l<k} a^0_{l,k+1}E_{l,k} + \sum_{k+3<l} a_{k,l-2} E_{k+3,l},
\]
which gives the relations
\[
I+B_0 \doteq I +\sum_{l < k} a^0_{l, k+1} E_{l, k} + \sum_{k+1< l} a^0_{k,l} E_{k+1,l}=U.
\]

At this point all generators above the square $X$ have been canceled except for  $a^0_{i,j}$ with $(i,j) \neq (k,k+1)$.  Moreover, the matrix associated to each of the $4$ corners of $X$ is $A_0$, and to all of the outer edges is $0$ except for the edge labeled $B_0$ in Figure \ref{fig:Octagon} which has $I+B_0=U$ with $U$ as in the definition of the DGA with cone points.  Thus, the appearances of $a^0_{i,j}$ in the differential of bordering cells matches the definition of the cone point DGA.
\end{proof}

\subsection{1-cells with more than one crossing or cusp arc}
\label{ssec:MultipleCrossings}
Often it is possible to relax the requirement that at most $1$ crossing or cusp arc lies above each $1$-cell.  We do not attempt to give the most general statement here, but restrict ourselves to an illustrative example that will be useful for the class of examples in Section \ref{sec:Lnsigma}.

For $L \subset J^1S$, suppose that $\mathcal{E}_C$ is a polygonal decomposition of $\pi_x(L)$ that is $L$-compatible except near a $2$-sided simple closed curve $C$ contained in the $1$-skeleton of $\mathcal{E}_C$.  In a neighborhood $N \cong S^1 \times[0,1]$ of $C \cong S^1 \times\{1/2\}$, there appear several crossing arcs that project to small shifts of $C$ in the normal direction, and no other crossing or cusp arcs.  See Figure \ref{fig:CCrossings}.   Above $N$, $L$ is a union of sheets that project homeomorphically to $N$.  Suppose further that {\it no two sheets of $L$ cross each other more than once above $N$.}

\begin{figure}

\labellist
\small
\pinlabel $C$ [tr] at 46 58
\endlabellist
\centerline{ \includegraphics[scale=.5]{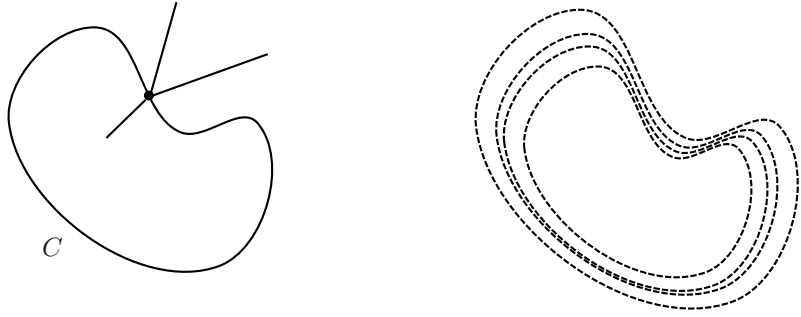} } 
\caption{(left) The curve $C$ in the decomposition $\mathcal{E}_C$.  (right) Crossing arcs appearing in the neighborhood $N \cong S^1\times[0,1]$ of $C$.}
\label{fig:CCrossings}
\end{figure}

With this set-up, we associate a variant of the cellular DGA $(\mathcal{A}_C,\partial)$ to $\mathcal{E}_C$.  For definiteness of notation, choose one side of $C$, and label the sheets of $L$ as $S_1, \ldots, S_n$ in the order they appear above this side of $N$.  Let $\sigma: \{1, \ldots, n\} \rightarrow \{1, \ldots, n \}$ be the permutation such that $S_i$ appears in position $S_{\sigma(i)}$ on the other side of $N$.  Note that for $i<j$, $S_i$ and $S_j$ cross above $S$ if and only if $\sigma(i) > \sigma(j)$.  To each $1$-cell (resp. $0$-cell) of $C$, assign generators $b_{i,j}$  (resp. $a_{i,j}$) for all $i<j$ with $\sigma(i) < \sigma(j)$.  
(So multiple pairs of crossing sheets contribute to multiple 0s to the upper triangle, as in the single crossing case.)
The differential is defined as in the case of the cellular DGA:  $0$-cells in $C$ have $\partial A = A^2$; $1$-cells in $C$ have $\partial B = A_+(I+B) + (I+B)A_-$; when a $1$-cell or $0$-cell appear in the boundary of a $2$-cell, we place the generators into matrices $B_i$ or $A_{v_i}$ being sure to order rows and columns using the ordering of sheets above the $2$-cell.  For instance, above $2$-cells on the side of $C$ where sheets of $L$ above $N$ appear in order $S_1, \ldots, S_n$ (resp. $S_{\sigma^{-1}(1)}, \ldots, S_{\sigma^{-1}(n)}$) generators associated to a bordering $1$-cell in $C$ are placed as $B_i = (b_{i,j})$ (resp.  $B_i= Q^{-1}_{\sigma} (b_{i,j}) Q_{\sigma}= (b_{\sigma^{-1}(i), \sigma^{-1}(j)})$ where $Q_{\sigma}$ is the permutation matrix $Q_{\sigma} = \sum E_{\sigma(i), i}$.)

\begin{proposition} \label{prop:ExtraCrossing}
The DGA $(\mathcal{A}, \partial)$ is stable tame isomorphic to the cellular DGA of $L$.
\end{proposition}
\begin{proof}
 For simplicity, we consider the case when the decomposition of $C$ consists of a single $0$-cell at $0=1 \in S^1 = \R/\Z$ and $1$-cell at $(0,1)$ in detail, with a similar argument applying in general.  
We produce a $L$-compatible polygonal decomposition $\mathcal{E}_L$ from $\mathcal{E}_C$ by replacing the decomposition of the single closed curve, $C$, with a decomposition as in Figure \ref{fig:Dec};  using coordinates $(x,y) \in S^1 \times[0,1] \cong N$, we take the crossing arcs in $L$ to be at $y=y_0, y= y_1, \ldots, y=y_r$, each decomposed into a $0$-cell, $\{(0,y_i)\}$ and a $1$-cell $(0,1) \times\{y_i\}$.  Additional $1$-cells are given by $\{0\}\times(y_{m-1},y_m)$ running perpendicularly through the annulus  connecting the $0$-cells of the different crossing arcs.  See Figure \ref{fig:Dec}

Denote the Cellular DGA of $\mathcal{E}_L$ by $(\mathcal{A}_L,\partial)$.  
We place generators of $\mathcal{A}_L$ corresponding to the translates of the original $0$- and $1$-cell of $C$ into matrices $A_m$, $B_m$, for $m=0, \ldots, r$, and place the generators corresponding to vertical product cells of the form $\{0\} \times (y_{m-1},y_{m})$ (resp. $(0,1) \times (y_{m-1},y_{m})$ into matrices $B_m'$, $C_m$, for $m=1, \ldots, r$.   For the $A_m$ and $B_m$ we take the ordering of sheets (and labeling of generators) to agree with the ordering in the bordering region that has smaller $y$-coordinate.  {\it In this proof, these notations for matrices are fixed} (including the ordering of rows and columns).

\begin{figure}

\labellist
\small
\pinlabel $A_0$ [r] at 472 4
\pinlabel $A_1$ [r] at 472 36
\pinlabel $A_2$ [r] at 472 68
\pinlabel $A_3$ [r] at 472 100
\pinlabel $B'_1$ [l] at 482 20 
\pinlabel $B'_2$ [l] at 482 52 
\pinlabel $B'_3$ [l] at 482 84 
\pinlabel $C_1$ [l] at 536 20 
\pinlabel $C_2$ [l] at 536 52 
\pinlabel $C_3$ [l] at 536 84 
\pinlabel $B_0$ [t] at 598 4 
\pinlabel $B_1$ [t] at 598 36 
\pinlabel $B_2$ [t] at 598 68 
\pinlabel $B_3$ [t] at 598 102

\endlabellist
 \includegraphics[scale=.7]{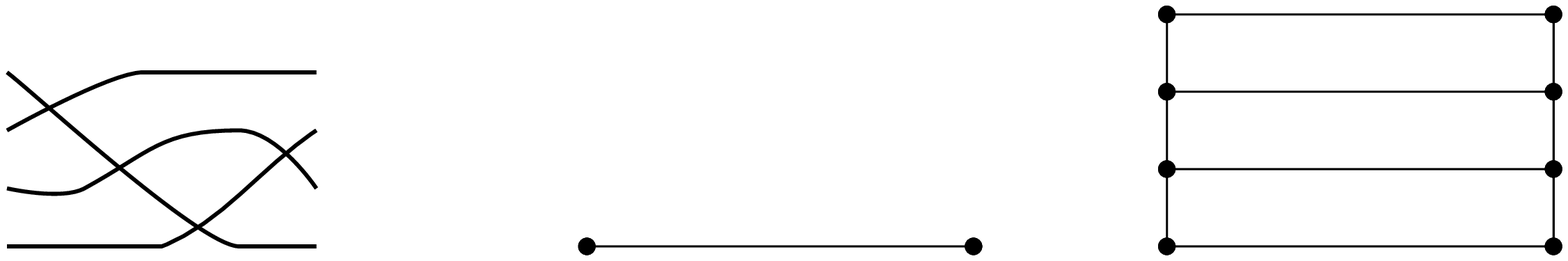}  
\caption{(Far left) A cross-section of the front projection of $L$ above $\{0\} \times[0,1] \subset N$, with $\sigma_0 = (1\,2)$, $\sigma_1 = (2 \, 3)$, $\sigma_2 = (3 \, 4)$, $\sigma_3 = (2 \, 3)$, and $\sigma = (1 \, 4\,2)$.  (Middle)  The decomposition of  $C$ in $\mathcal{E}_C$.  (Right) The decomposition $\mathcal{E}_L$ in the neighborhood $N$.}
\label{fig:Dec}
\end{figure}

To describe the differential, let $\sigma_0= (k_0,k_0+1), \ldots, \sigma_r = (k_r,k_r+1)$ be the sequence of transpositions corresponding to the crossing arcs as $y$ increases from $0$ to $1$.  So, if the crossings occur at $y=y_0, \ldots, y=y_r$, then the sheet in position $i$ immediately before $y_m$ is in position $\sigma_m(i)$ after $y_m$.  Notice that the products
\[
\widetilde{\sigma}_m = \sigma_m \sigma_{m-1} \cdots \sigma_0
\] 
then have the property that the sheet labeled $S_i$ before $y_0$ is labeled $S_{\widetilde{\sigma}_m(i)}$ after $y_m$.
In particular, the permutation $\sigma$ defined prior to the statement of the proposition is $\sigma = \widetilde{\sigma}_r$.

For the subalgebra generated by the $A_m$ and $B'_m$ generators (above $\{0\} \times [0,1]$) we have
\[
\partial B'_m = A_{m}B'_m + B'_{m} Q_{\sigma_{m-1}}A_{m-1}Q^{-1}_{\sigma_{m-1}}.
\]
[The $(i,j)$-entry of $Q_{\sigma_{m-1}}A_{m-1}Q^{-1}_{\sigma_{m-1}}$ is $a^{m-1}_{\sigma^{-1}_{m-1}(i), \sigma^{-1}_{m-1}(j)}$.]
We cancel all of the $B'_m$ and $A_m$ generators with $1 \leq m \leq r$ and some of the $A_0$ generators via the following inductive procedure.

To begin, for $(i,j) \neq (k_1,k_1+1)$, cancel $b^1_{i,j}$ with $a^1_{i,j}$, and for $(i,j) = (k_1,k_1+1)$, cancel $b^1_{k_1,k_1+1} = a^{0}_{\sigma^{-1}_{0}(k_1), \sigma^{-1}_{0}(k_1+1)}$ 
[We apply Theorem \ref{thm:Alg} repeatedly, so that $|j-i|$ increases as we go.  Thus, any $b_{p,q}$ terms that appear in $\partial b_{i,j}$ are already $0$ when the theorem is applied.]

For $m \geq 2$, inductively suppose that we have canceled all of the generators of $B'_l$ with $l<m$, and $A_l$ with $0< l <m$, as well as all entries $a^0_{i,j}$ with the property that the sheets labeled $S_{i}$ and $S_{j}$ before $y_0$ cross one another at some $y \in [y_0, y_m)$.  Moreover, suppose that in the quotient, we have 
\[
A_{m-1} \doteq Q_{\widetilde{\sigma}_{m-2}} A_0 Q_{\widetilde{\sigma}_{m-2}}^{-1}, 
\]
and $a^0_{i,j} \doteq 0$ when $S_i$ and $S_j$ cross before $y_m$.  

Then, the differential becomes
\[
\partial B'_m = A_{m} (I+B'_m) + (I+B'_m) Q_{\sigma_{m-1}} Q_{\widetilde{\sigma}_{m-2}} A_0 Q_{\widetilde{\sigma}_{m-2}}^{-1} Q_{\sigma_{m-1}}^{-1}.
\]
Note that the final term is $Q_{\widetilde{\sigma}_{m-1}} A_0 Q_{\widetilde{\sigma}_{m-1}}^{-1}$;  it is upper-triangular (because of the entries that are equal to $0$.)
For $(i,j) \neq (k_{m}, k_{m}+1)$, we cancel $b^m_{i,j}$ and $a^m_{i,j}$.  For $(i,j) = (k_{m}, k_{m}+1)$, we cancel
\[
\partial b^m_{k_m,k_m+1} = a^0_{\widetilde{\sigma}^{-1}_{m-1}(k_m), \widetilde{\sigma}^{-1}_{m-1}(k_m+1)}.
\] 
[This entry has not yet been canceled because of the assumption that no two sheets cross each other more than once in $N$.]
In the quotient, since $B'_m \doteq 0$, it follows that 
\[
0 \doteq \partial B'_m = A_m + Q_{\widetilde{\sigma}_{m-1}} A_0 Q_{\widetilde{\sigma}_{m-1}}^{-1}  \quad \Rightarrow \quad  A_m = Q_{\widetilde{\sigma}_{m-1}} A_0 Q_{\widetilde{\sigma}_{m-1}}^{-1}, 
\]
and $a^0_{\widetilde{\sigma}^{-1}_{m-1}(k_m), \widetilde{\sigma}^{-1}_{m-1}(k_m+1)} \doteq 0$ so that the induction proceeds.

At the conclusion of the procedure, the only remaining generators are those entries of $A_0$ corresponding to sheets that do not cross one another above $N$, and we have
\begin{equation} \label{eq:Arq}
A_r \doteq Q_\sigma A_0 Q_\sigma^{-1}.
\end{equation}

At this point, differentials of the $C_m$ have simplified as 
\begin{align*}
\partial C_m &= A_{m}C_m +C_m Q_{\sigma_{m-1}}A_{m-1}Q^{-1}_{\sigma_{m-1}} + (I+B_m)(I+B'_m) + (I+B_m')(I + Q_{\sigma_{m-1}}B_{m-1}Q^{-1}_{\sigma_{m-1}}) \\
  & \doteq A_{m}C_m +C_m Q_{\sigma_{m-1}}A_{m-1}Q^{-1}_{\sigma_{m-1}} + B_m + Q_{\sigma_{m-1}}B_{m-1}Q^{-1}_{\sigma_{m-1}}.
\end{align*}
Then, using a similar inductive procedure, we can cancel all of the $C_m$ and $B_m$ for $1 \leq m \leq r$, as well as those entries of $B_0$ corresponding to sheets $S_i$ and $S_j$ that cross somewhere in $N$.  In the quotient, we have the relation
\begin{equation} \label{eq:Arq2}
B_r \doteq Q_\sigma B_0 Q_\sigma^{-1}.
\end{equation}

The remaining generators in this final quotient of $(\A_L,\partial)$ are exactly as in the description of $(\A_C, \partial)$, and the differentials agree as well in view of the relations (\ref{eq:Arq}) and (\ref{eq:Arq2}).


\end{proof}

\subsubsection{Extension to multiple cone points}
\label{sssec:MultipleConePoints}

The DGA with cone points from Proposition \ref{prop:coneDGA} has a similar easy extension to the case where $L$ has several cone points projecting to the same point $x_0 \in S$.  Suppose that above $x_0$ there are $r$ cone points joining the pairs of sheets $(k_1, k_1+1), \ldots, (k_r,k_r+1)$.  In this case, form a DGA as in Proposition \ref{prop:coneDGA} with the two modifications:
\begin{enumerate}
\item The $0$-cell at $x_0$ has generators $a_{i,j}$ with $i<j$ and $(i,j) \neq (k_p,k_p+1)$ for all $1\leq p \leq r$.
\item The matrix $U= U_r U_{r-1} \cdots U_1$ where   
\begin{equation}
\label{eq:MultipleU}
U_p = I + \sum_{\ell < k_p} a_{\ell, k_p+1} E_{\ell, k_p} + \sum_{k_p+1< \ell} a_{k_p,\ell} E_{k_p+1,\ell}.
\end{equation}
\end{enumerate}

\begin{proposition}  \label{prop:ExtraCone}
The resulting DGA is equivalent to the DGA from Proposition \ref{prop:coneDGA}.  
\end{proposition}

\begin{proof}
To see this, isotope the front projection of $L$ so that the cone points appear at distinct points of $S$ in a small neighborhood of $x_0$.  Then, form a cellular decomposition where the single $0$-cell at $x_0$ is expanded to a string of $0$-cells, $A_1, \ldots, A_{r}$ at cone points connected by a sequence of $1$-cells, $B_1, \ldots, B_{r-1}$, all located in the corner of $x_0$ that was labeled with $U$.  All of the generators for the $B_1, \ldots, B_{r-1}$, and $A_2, \ldots, A_{r}$, as well as the $(k_p,k_p+1)$-entries of $A_{1}$ are canceled by a sequence of applications of Theorem \ref{thm:Alg}.  (At the inductive step, cancel $B_m$ with all of the generators of $A_{m+1}$ and the $(k_{l},k_{l}+1)$ entries of $A_m$ for $l\geq m+1$.) When computing the differential of $2$-cell that contains the string of $0$-cells and $1$-cells, the contribution to the product of edges from this corner is the product of the individual $U_p$ matrices
\[
U_p = I + \sum_{\ell < k_p} a_{\ell, k_p+1} E_{\ell, k_p} + \sum_{k_p+1< \ell} a_{k_p,\ell} E_{k_p+1,\ell},
\]     
and in the quotient $a^p_{i,j} = a^q_{i,j}$ for all $1 \leq p < q \leq r$.
\end{proof}

\subsection{Front spinning a Legendrian knot}
\label{ssec:FrontSpinning}

In this subsection, we show how the DGA of a surface which is the front-spin of a \dr{1-dimensional} Legendrian knot can be derived from the DGA of the knot. The result agrees with known computations when the axis of rotation and the knot do not intersect, and provides new examples when the axis and knot do intersect.

\subsubsection{The suspension of a DGA}
\label{sssec:AlgebraSpin}
We begin with some algebraic preliminaries, which are motivated by a homotopy theory for DGAs  (or more generally for differential graded operad algebras where the operad might not be the associative one).
Suppose $(\mathcal{A}, \partial_\mathcal{A})$ is a DGA freely generated by $a_1, \ldots, a_{k},$ triangular with respect to these generators (see Section \ref{ssec:STI}), possibly unital, and with an arbitrary coefficient ring.
Define the {\bf{suspension}} of $(\mathcal{A}, \partial_\mathcal{A}),$ $(\hat{\mathcal{A}}, \partial_{\hat{\mathcal{A}}}),$ as follows:
\begin{itemize}
\item 
 $\hat{\mathcal{A}}$ is freely generated by $a_i,$ $\hat{a}_i;$ 
\item
the grading $|a_i|$ of $a_i \in \hat{\mathcal{A}}$ is the same as it is in $\mathcal{A},$ while $|\hat{a}_i| = |a_i|+1;$
\item
$\partial_{\hat{\mathcal{A}}}(a_i) = \partial_\mathcal{A}(a_i)$ and $\partial_{\hat{\mathcal{A}}}(\hat{a}_i) = \Gamma(\partial_\mathcal{A}(a_i))$ where
$\Gamma: \mathcal{A} \rightarrow \hat{\mathcal{A}}$ is a derivation defined by its image of the generators $\Gamma(a_i) = \hat{a}_i.$
(So for example, $\Gamma(1) = 0$ if $\mathcal{A}$ is unital, and $\Gamma(a_1 a_2) = \hat{a}_1 {a}_2 + (-1)^{|a_1|} {a}_1 \hat{a}_2.$)
\end{itemize}

We assume $\hat{\mathcal{A}}$ is unital if ${\mathcal{A}}$ is.

We also consider a  variation of this construction. 
Let $(\mathcal{L}, \partial_\mathcal{L})$ be a sub-DGA
of $(\mathcal{A}, \partial_\mathcal{A})$ generated by some subset of the generating set.
Define the {\bf{suspension of}}  $\mathcal{A}$ {\bf{relative to}}  $\mathcal{L}$ to be the DGA $\hat{\mathcal{A}}$ as above except that for all elements $l \in \mathcal{L},$ $\hat{l} = 0.$

\begin{lemma}
\label{lem:SuspensionSTI}
If $(\mathcal{A},\partial_\mathcal{A})$ and $(\mathcal{B},\partial_\mathcal{B})$ are stable tame isomorphic then so too are
$(\hat{\mathcal{A}}, \partial_{\hat{\mathcal{A}}})$ and $(\hat{\mathcal{B}}, \partial_{\hat{\mathcal{B}}}).$
If $(\mathcal{L}, \partial_\mathcal{L})$ is a sub-DGA as above of both $(\mathcal{A}, \partial_\mathcal{A})$ and $(\hat{\mathcal{B}}, \partial_{\hat{\mathcal{B}}})$ and the \dr{elementary automorphisms in the definition of} stable-tame isomorphism fixes $\mathcal{L},$ then the suspensions of $\mathcal{A}$ and $\mathcal{B}$  {{relative to}}  $\mathcal{L}$ are stable tame isomorphic.

\end{lemma}


\begin{proof}

It is easy to check that if $\mathcal{B} = S(\mathcal{A})$ is a stabilization with new generators $x,y$ 
then $\hat{\mathcal{B}} = S(S(\hat{\mathcal{A}}))$ with new generators $x, y, \hat{x} = \Gamma({x}), \hat{y} = \Gamma({y}).$ 
So assume $\phi: (\mathcal{A},\partial_\mathcal{A}) \rightarrow (\mathcal{B},\partial_\mathcal{B})$ is an elementary
 isomorphism, defined on generators as $\phi(a_i) = a_i + \delta^i_j w.$ (In the language of Section \ref{ssec:STI}, $\phi(a_j) = a_j +w$ while $\phi$ is the identity on all other generators.)
We claim the algebra isomorphism
$\hat{\phi}: \hat{\mathcal{A}} \rightarrow \hat{\mathcal{B}}$ defined on generators by 
\[
\hat{\phi}(a_i) = a_i + \delta^i_j w, \quad \hat{\phi}(\Gamma({a}_i)) = \Gamma({a}_i) + \delta^i_j \Gamma({w})
\]
is a (tame) DGA isomorphism.

First note that $\Gamma \phi = \hat{\phi} \Gamma.$ 
This is obvious on generators, while the inductive (on word-length) step follows from linearity and the following relation on monomials:
\begin{eqnarray*}
\hat{\phi}(\Gamma(x y)) 
&=  & \hat{\phi}(\Gamma(x) y + (-1)^{|x|} x \Gamma(y)) 
= \Gamma(\phi(x)) \phi(y) + (-1)^{|x|} \phi(x) \Gamma(\phi(y)) \\
& = & \Gamma(\phi(x)) \phi(y) + (-1)^{|\phi(x)|} \phi(x) \Gamma(\phi(y))
 = \Gamma(\phi(xy)).
 \end{eqnarray*}

First note that since $\partial_{\mathcal{B}} = \phi \partial_{\mathcal{A}} \phi^{-1}$ then for generator
$a_i \in \hat{\mathcal{A}},$ 
$\hat{\phi}\partial_{\hat{\mathcal{A}}} \hat{\phi}^{-1}(a_i) 
= 
\partial_{\hat{\mathcal{B}}}(a_i).$
For the other generators, 
\begin{eqnarray*}
\hat{\phi}\partial_{\hat{\mathcal{A}}}\hat{\phi}^{-1}(\Gamma(a_i))
& = &
\hat{\phi}\partial_{\hat{\mathcal{A}}}\left(\Gamma(a_i) - \delta^i_j \Gamma(w)\right) \\
& = &
\hat{\phi} \Gamma(\partial_\mathcal{A}(a_i)) - \delta^i_j \hat{\phi}\Gamma(\partial_\mathcal{A}(w)) \\
\partial_{\hat{\mathcal{B}}}(\Gamma(a_i)) 
& = &
\Gamma(\partial_{\mathcal{B}}(a_i)) \\
& =& 
\Gamma({\phi}\partial_{{\mathcal{A}}}{\phi}^{-1}(a_i)) \\
& = &
\Gamma({\phi}\partial_{{\mathcal{A}}}(a_i -\delta^i_j w)) \\
& = & 
\Gamma \phi (\partial_{{\mathcal{A}}} (a_i)) - \delta^i_j \Gamma \phi (\partial_{{\mathcal{A}}} (w))
\end{eqnarray*}
These are equal by the $\Gamma \phi = \hat{\phi} \Gamma$ relation above. 

%

If the suspensions are relative, the same formulas and arguments apply after recalling that $\Gamma(l) = 0$ for $l \in \mathcal{L}.$
\end{proof}

\subsubsection{Spun cellular decomposition}
\label{sssec:GeometrySpin}
Let $J^1(\R_{x_1}) = (\R^3_{x_1,y_1,z}, \mbox{ker}\{dz-y_1dx_1\})$ denote the standard contact structure.
Let $\pi_{x_1,z}: J^1(\R_{x_1}) \rightarrow J^0(\R_{x_1}) = \R^2_{x_1,z}$ be the lower-dimensional analog of the front projection $\pi_{x,z}.$
Let $\Lambda \subset J^1(\R_{x_1})$ be a Legendrian embedding of the union of (a non-negative ) number of circles sitting in $\{x_1 <0\}$ as well as a (non-negative) number of arcs in $\{x_1 \le 0\}$ whose endpoints sit on the $z$-axis. 
We impose the following condition on the front projections of the arcs in $J^0(\R_{x_1}).$
If we reflect the arcs' fronts through the $z$-axis, the arcs and their reflections must glue together to form the front of a collection of 
(smooth) Legendrian circles.
An arbitrary Legendrian embedding of circles and arcs can achieve these conditions with a Legendrian isotopy.
Embed $J^0(\R_{x_1})$ as $\{x_2 = 0\} \subset J^0(\R^2_{x_1,x_2}) = \R_{x_1,x_2,z}^3.$
Consider the rotation about the $z$-axis, $S^1 \times \pi_{x,z}(\Lambda) \subset J^0(\R^2_{x_1,x_2})$ as defined
in \cite[Section 4]{EkholmEtnyreSullivan05a}.
We say that the {\bf{front spin}} of $\Lambda$ is the Legendrian surface $L$ in standard contact $J^1(\R^2_{x_1,x_2}) = \R^5$ whose front is given by this rotation.
Note that $L$ is a Legendrian embedding of a collection of spheres and/or tori.
In the case when two arc endpoints have the same front projection in $\{x_1 = 0\} \subset J^0(\R_{x_1}),$ then the spun surface $L$ has
a cone point.
We consider Legendrians with possibly multiple pairs of matching arc-front-endpoint conditions, \dr{but we assume} no three arcs have fronts with the same endpoint.

Consider a cellular decomposition of $\pi_{x_1}(\Lambda) \subset \R^1_{x_1},$ where in the case that  $\Lambda$ has a Legendrian arc, the decomposition includes $\{x_1 =0\}$ as a 0-cell.
For each cell $e$ of $\pi_x(\Lambda),$ excluding the 0-cell $\{x_1 = 0\}$ if it appears, we associate two cells for $\pi_x(L) \subset \R^2_{x_1,x_2},$ 
which we denote by $e$ and  $\hat{e}.$
The first is simply the cell $e$ after embedding $J^0(\R_{x_1})$ as $\{x_2 = 0\} \subset J^0(\R^2_{x_1,x_2}).$
The second is the embedded $e$ spun about the $z$-axis; that is, a radially symmetric map of $e \times [-\pi, \pi]$ with the ends $e \times \{-\pi\}$ and $e \times \{\pi\}$ mapping to the same cell as $e$.
Note that the crossing/cusp 0-cells become analogous 1-cells  under spinning.

If $e = \{x_1 = 0\}$ is a 0-cell of $\Lambda,$ then we denote by $e$ again the 0-cell $\{x_1=0=x_2\}$ of $L.$
As mentioned above, if $e =\{x_1 = 0\}$ represents a crossing 0-cell of $\Lambda,$ then the corresponding cell in $L$ sits below a cone point.

\subsubsection{The cellular DGA of a front spin}

\begin{prop}
\label{prop:1dimLCH_DGA}
The LCH DGA $(\mathcal{A}_{LCH, \Lambda}, \partial_{LCH, \Lambda})$ of the $1$-dimensional Legendrian $\Lambda$ 
is stable-tame isomorphic to  the cellular DGA of $\Lambda$,  $(\mathcal{A}_\Lambda, \partial_\Lambda)$,  as specified in Section \ref{sec:DefDGA} equations (\ref{eq:Adef}) and (\ref{eq:Bdef}).
\end{prop}

\begin{proof}
This can be deduced from a (greatly shortened) argument along the lines of the isomorphism between the Cellular DGA and LCH proved for $2$-dimensional Legendrians in \cite{RuSu2}.  

However, this is not necessary, since in the $1$-dimensional case the Cellular DGA is almost identical to the LCH DGA as computed after the addition of {\it splashes} or {\it dips}, \cite{Fu}, \cite{Sab}.  To give a definite isomorphism, we use the version of splashes from \cite{FuRu} where the generators and differentials for the LCH DGA of a $1$-dim Legendrian $\Lambda \subset J^1\R$ are stated in Propsition 4.1 of \cite{FuRu}.  There $\Lambda$ is broken into a product of elementary tangles containing either a single left cusp, right cusp, or crossing.    To the right of the $m$-th elementary tangle, a collection of splashes is given (for all $m$), and this produces two upper triangular matrices of generators labeled in \cite{FuRu} as $x^{+}_{m;i,j}$ and $x^{-}_{m;i,j}$.  In addition, when there is a crossing (resp. right cusp) in the $m$-th elementary tangle, there is a single generator $y_m$ (resp. $z_m$).

A $\Lambda$-compatible polygonal decomposition of $\pi_{x}(\Lambda)$ arises in an obvious way by taking the $0$-cells to occur at cusps and crossings.  Label the $0$- and $1$-cells, so that from left to right they appear as $A_1, B_2, A_2, B_3, A_3, \ldots, B_{N+1}, A_{N+1}$.  If we identify the Cellular DGA generators $a^m_{i,j}$ and $b^m_{i,j}$ with $x^+_{m;i,j}$ and $x^-_{m;i,j}$ respectively, then the differentials almost agree.  The differences are the following:
\begin{enumerate}
\item At $0$-cells with crossings, there are $2$ more generators, $y_m$ and $x^+_{m-1;k,k+1}$, in the LCH DGA.
\item At $0$-cells with right cusps, there are more generators, $z_m$, $x^+_{m-1;k,k+1}$, $x^+_{m-1;i,k+1}$, $x^+_{m-1;i,k}$,  $x^+_{m-1;k,j}$, and $x^+_{m-1;k+1,j}$.
\end{enumerate}
Repeated applications of Theorem \ref{thm:Alg} (that the reader may by now find routine) cancel these generators in pairs.  In this quotient of the LCH DGA, the differential agrees precisely with the Cellular DGA differential.

\end{proof}



We can now state the main proposition of Subsection \ref{ssec:FrontSpinning}

\begin{prop}
\label{prop:FrontSpin}
\begin{enumerate}
\item
Suppose $\Lambda$ has no arc components.
The DGA $(\mathcal{A}_L, \partial_{\mathcal{A}_L})$ is the suspension 
of $(\mathcal{A}_\Lambda, \partial_{\mathcal{A}_\Lambda}).$
\item
Suppose $\Lambda$ has arc components whose front projections have distinct endpoints.
The DGA $(\mathcal{A}_L, \partial_{\mathcal{A}_L})$ is the suspension 
of $(\mathcal{A}_\Lambda, \partial_{\mathcal{A}_\Lambda})$ relative to the sub-DGA  associated to the $\{x_1 = 0 = x_2\}$ 0-cell.
\item
Suppose $\Lambda$ has arc components whose front projections have $r$ pairs of matching endpoints.
The DGA $(\mathcal{A}_L, \partial_{\mathcal{A}_L})$ is the suspension 
of $(\mathcal{A}_\Lambda, \partial_{\mathcal{A}_\Lambda})$ relative to the sub-DGA associated to the $\{x_1 = 0 = x_2\}$ 0-cell,
with the following modification.
Let $A$ be the matrix of generators above the $\{x_1 = 0 = x_2\}$ 0-cell.
We replace $\Gamma(A) =0$ with $\Gamma(A) = U -I$ from equation (\ref{eq:MultipleU}).
\end{enumerate}

\end{prop}


\begin{proof}

As we see from Sections \ref{sssec:0cells} and \ref{sec:ddef1cell} that for each generator of type $a_{i,j}$ and $b_{i,j}$ in $\mathcal{A}_\Lambda,$ the cellular  $\mathcal{A}_L$ has a pair of generators generators $a_{i,j}, \hat{a}_{i,j}$ and  $b_{i,j}, \hat{b}_{i,j}.$
A quick check shows $|\hat{x}_{i,j}| = |x_{i,j}| +1$ for either type of generator.
(If $A = (a_{i,j})$ sits above the 0-cell $\{x_1  = 0\}$ there are no new generators of type $\hat{a}_{i,j}.$)

Let $\hat{\partial}$ denote the suspension of $\partial := \partial_\Lambda,$ which a priori may not equal $\partial_L.$
We first verify that $\Gamma \partial = \hat{\partial} \Gamma.$ 
This holds by definition in case (1) and by a tautology in case (2), so consider case (3).
Recall $A = (a_{i,j})$ is the matrix of elements sitting over the  0-cell $\{x_1  = 0\}.$
\begin{eqnarray*}
\Gamma \partial (A) &= &\Gamma(A^2) = A \Gamma(A) + \Gamma(A) A = AU + UA,\\
\hat{\partial} \Gamma(A) & = & \hat{\partial}(U).
\end{eqnarray*}
Recall $U = U_r U_{r-1} \cdots U_1.$
Assume we have shown that $AU+UA = \hat{\partial}U$ if $r=1.$ 
Then the general case follows from induction: 
\[
\hat{\partial} (U_{r+1} U) = \hat{\partial}(U_{r+1}) U + U_{r+1} \hat{\partial}(U)  = AU_{r+1} U + U_{r+1} A U + U_{r+1} AU +U_{r+1} U A = A(U_{r+1} U) + (U_{r+1} U)A.
\]
Since $AI + IA = 0$ and $\hat{\partial} I = 0,$ from (\ref{eq:ConePointU}) it suffices to show
\[
\sum_{\ell < k} A E_{\ell, k}a_{\ell, k+1} + \sum_{k+1< \ell} A  E_{k+1,\ell} a_{k,\ell} + 
\sum_{\ell < k}  a_{\ell, k+1} E_{\ell, k} A+ \sum_{k+1< \ell} a_{k,\ell} E_{k+1,\ell}  A = 
\sum_{\ell < k}  \partial a_{\ell, k+1} E_{\ell, k} + \sum_{k+1< \ell} \partial a_{k,\ell} E_{k+1,\ell} .
\]
Recall $a_{k,k+1} = 0.$ 
Let $(M)_{ij}$ denote the $(i,j)$-entry of the $n \times n$ matrix $M$ and $\delta^i_j$ denote the Kronecker delta.
We compute each entry of each term in each of the four left hand side summands.
\begin{eqnarray*}
(A E_{\ell, k})_{ij} a_{\ell, k+1} & = & \delta^j_k a_{i, \ell} a_{\ell, k+1} \,\,\mbox{where} \,\, 1 \le i < \ell < k, \\
(A  E_{k+1,\ell})_{ij} a_{k,\ell} & = & \delta^j_\ell a_{i,k+1} a_{k,\ell} \,\,\mbox{where} \,\, 1 \le i < k, \,\, k+1 < \ell \le n, \\
a_{\ell, k+1} (E_{\ell, k} A)_{ij} & = & \delta^i_\ell a_{\ell, k+1} a_{k,j} \,\,\mbox{where} \,\, 1 \le \ell < k, \,\, k+1 <  j \le n,\\
a_{k,\ell} (E_{k+1,\ell}  A)_{ij} & = & \delta^i_{k+1} a_{k, \ell} a_{\ell, j}\,\,\mbox{where} \,\, k+1 < \ell < j \le n.
\end{eqnarray*}
So the second and third summands cancel, while the first and fourth match the first and second on the right hand side.

Let $X = (x_{i,j})$ and $\hat{X} = (\hat{x}_{i,j})$ denote the strictly upper-triangular matrices of generators as in Section \ref{ssec:diff},
modified appropriately if a crossing or cusp pair is involved.
Then $\partial_L A$ is given by (\ref{eq:Adef}), $\partial_L B, \partial_L \hat{A}$ are given by (\ref{eq:Bdef}), and $\partial_L \hat{B}$ is given by 
(\ref{eq:Cdef}).
We verify the identity for $\partial_L \hat{B},$ as $\partial_L \hat{A}$ is easier and $\partial_L A, \partial_L B$ are immediate.

Let $0 \ge t_1 > \ldots > t_M \in \R^1_{x_1<0} = \R^1_{r>0}$ denote the $0$-cells used to define the cellular DGA of $\Lambda.$
We use polar coordinates $(r,\theta) \in \R^1_+ \times [-\pi, \pi]$ for
$\R^2_{x_1,x_2}$ to refer to the $1$-cells (polar lines) and $2$-cells (polar rectangles) which define the cellular DGA of $L.$ 
For example, all 1-cells that contain cusp or crossing loci are on $\{r = \mbox{const}\}$ lines.
If a $1$-cell is indicated to lie on $\{\theta = +\pi\}$ (resp.  $\{\theta = -\pi\}$), this means that we consider the $1$-cell as the boundary of a $2$-cell lying counter-clockwise (resp. clockwise) to it.
We orient all $1$-cells in the positive polar coordinate direction.
Fix a $2$-cell $\{t_m \le r \le t_{m+1}\}$ with initial and terminal vertices $v_m = (t_m, -\pi),$ $v_{m+1}=(t_{m+1}, \pi).$
We use below that $B_{\{\theta = -\pi\}} =  B_{\{\theta = \pi\}},$  $\hat{B} = \Gamma(B_{\theta = \pm \pi})$ and $\hat{A}_{v_l} = B_{\{r= t_l\}}.$
Also, for the relative suspension case  (so $t_1 = 0$),  $B_{\{r= t_1\}} = \Gamma({A}_{v_1}) $ equals the 0 matrix when there are no cone points in case (2), and equals $U-I$ when there are in case (3).
\begin{eqnarray*}
\partial_L \hat{B} & = & A_{v_{m+1}} \hat{B} + \hat{B} A_{v_m} + (I + B_{\{r=t_{m+1}\}})(I + B_{\{\theta = -\pi\}}) + (I + B_{\{\theta = \pi\}})(I + B_{\{r = t_m\}})\\
& = & A_{v_{m+1}} \hat{B} + \hat{B} A_{v_m} + B_{\{r=t_{m+1}\}}B_{\{\theta = -\pi\}} + B_{\{\theta = \pi\}}B_{\{r = t_m\}} + B_{\{r=t_{m+1}\}}+ B_{\{r = t_m\}} \\
& = & \Gamma \left( A_{v_{m+1}} B_{\{\theta = -\pi\}} + B_{\{\theta = \pi\}}A_{v_m}  + A_{v_{m+1}} + A_{v_m} \right)\\
& = & \Gamma\left( \partial_\Lambda B_{\{\theta = \pi\}}\right).
\end{eqnarray*}
\end{proof}

\subsubsection{A comparison with older front-spun computations}

The first case (where there are no Legendrian arcs) is already known in the literature.
The second two cases (Legendrian arcs spinning to possibly induce cone points, or not) are new. 

\begin{corollary}
\label{cor:EkholmKalman}
The cellular DGA computation in the first example agrees with pre-existing computations of the LCH DGA, up to stable tame isomorphism.
\end{corollary}

\begin{proof}
A partial computation of the LCH DGA of this example was done, up to certain quadratic terms, in 
\cite[Proposition 4.17]{EkholmEtnyreSullivan05a}.
The complete LCH computation was done in \cite[Theorem 1.1]{EkholmKalman08}, which states that up to stable tame isomorphism,
the LCH DGA of $L$ is the suspension of the LCH DGA of $\Lambda.$
The result then follows from Lemma \ref{lem:SuspensionSTI} and Propositions \ref{prop:1dimLCH_DGA} and \ref{prop:FrontSpin}.
\end{proof}

\subsection{A family of examples} \label{sec:Lnsigma}

The following generalizes a family of Legendrian spheres considered in \cite{BST}.

We introduce Legendrians, $L_{n, \sigma, \mathbf{a}}$, where $n \geq 1$,
\begin{itemize}
\item $\sigma: \{ 1,2, \ldots, 2n \} \rightarrow \{1,2, \ldots, 2n\}$ is a permutation with the property that
\[
i<j \mbox{  and  } \sigma(i)>\sigma(j) \quad \Rightarrow \quad \mbox{$i$ even  and  $j$ odd},
\] 
\item and $\mathbf{a} = (a_1, a_2, \ldots, a_n) \in \Z^n$ satisfies $a_1 > a_2 > \ldots > a_n = 0$.
\end{itemize} 

To construct $L_{n, \sigma, \mathbf{a}}$, start with $n$ copies of the standard $1$-dimensional Legendrian placed vertically above one another so that their left and right cusps all share common $x$-coordinates.  For some pairs of unknots, the lower strand of the upper unknot crosses the upper strand of the lower unknot twice to form a Hopf link.  This is specified by the permutation $\sigma$ as follows:  Labeling the strands near left cusps in with descending $z$-coordinate as $1,2 \ldots, 2n$, for $i<j$ the $i$-th strand crosses the $j$-th strand twice if and only if $\sigma(i) > \sigma(j)$.  Arrange that the front diagram is symmetric, so that when we spin the front diagram around the midpoint (in the $x$-direction) it produces a stack of two dimensional unknots with one circular crossing arc for each $i<j$ with $\sigma(i)>\sigma(j)$.  Call this spun front $L_1$.

Next, to the right of the previously constructed front, we create one more Legendrian sphere, $L_2$.  Begin with a single $1$-dimensional unknot; apply $a_{\ms{1}}$ consecutive Type 1 Reidemeister moves to produce a front with $a_{\ms{1}} +1$ left cusps and $a_{\ms{1}}$ crossings. Then, spin around the midpoint to produce a $2$-dimensional Legendrian sphere that is $a_{{\ms{1}}}+1$ standard $2$-dimensional unknots joined together by a sequence of $a_{\ms{1}}$ cone points.

Finally, to arrive at $L_{n, \sigma, \mathbf{a}}$, we connect these two pieces by joining the $i$-th unknot from $L_1$ to the $a_i$-th cusp edge from $L_2$ by a tube whose cross-sections are Legendrian unknots.  Here, we label the unknots of $L_1$ with descending $z$-coordinate, but label the cusp edges of $L_2$ with ascending $z$-coordinate, starting with $0$.  See Figure \ref{fig:LSigma}.

\begin{figure}

\centerline{ \includegraphics[scale=.5]{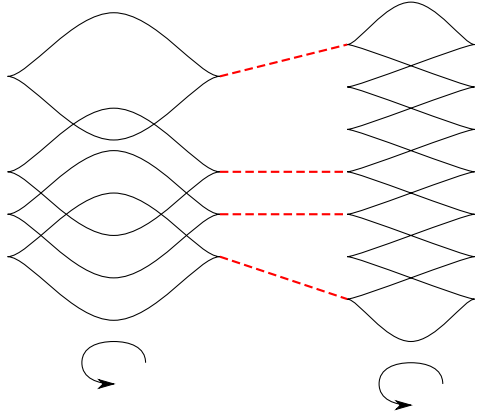} } 

\caption{Construction of $L_{n, \sigma, \mathbf{a}}$ when $n = 4$,  
$\sigma(1) = 1,$ $\sigma(2) = 3,$ $\sigma(3) = 2,$ $\sigma(4) = 6,$ $\sigma(5) = 4,$ $\sigma(6) = 7,$ $\sigma(7) = 5,$ $\sigma(8) = 8$
and $\mathbf{a} = (a_1, a_2, a_3, a_4) = (6, 3, 2, 0)$. The dotted red lines, are tubes with cusp edges.}

\label{fig:LSigma}
\end{figure}

\begin{theorem}  \label{thm:Lnsigma}
The DGA of $L_{n, \sigma, \mathbf{a}}$ is equivalent to a DGA with generators
\[
z, \, b_{i,j}, c_{\sigma(j),\sigma(i)},  \quad \mbox{for $i<j$ with $\sigma(i)>\sigma(j)$}
\]
and differentials
\begin{align*}
\partial z = & \sum_{\stackrel{i<j}{\sigma(i)>\sigma(j)}} \left( b_{i,j}c_{\sigma(j),\sigma(i)} + c_{\sigma(j),\sigma(i)} b_{i,j} \right) \\
\partial b_{i,j} = &  \sum_{\begin{array}{c} i<m<j \\ \sigma(i) > \sigma(m-1) \\ \sigma(m) > \sigma(j) \end{array}} 
b_{i,m-1}b_{m,j} \\
\partial c_{\sigma(j),\sigma(i)} = &  \sum_{\begin{array}{c} i<j<m \\ \sigma(j) < \sigma(m) < \sigma(i) \end{array}} b_{j+1,m}c_{\sigma(m),\sigma(i)} 
+ \sum_{\begin{array}{c} m<i<j \\  \sigma(j) < \sigma(m) < \sigma(i) \end{array}} c_{\sigma(j),\sigma(m)}b_{m,i-1}.
\end{align*}
The grading is given by
\[
|z| = 2, \quad |b_{i,j}| = 2(a_k-a_l)-1, \quad |c_{\sigma(j),\sigma(i)}| = 2(a_l-a_k) +2
\]
where $i = 2k$ and $j = 2l -1$.
\end{theorem}

The crossing configurations that lead to terms in $\partial$ are depicted in Figure \ref{fig:Config}.
{\ms{We prove Theorem \ref{thm:Lnsigma} at the end of this subsection.}}

\begin{figure}

\labellist
\small
\pinlabel $\partial~z=$ [r] at -5 344
\pinlabel $i$ [l] at 64 320
\pinlabel $j$ [l] at 64 368
\pinlabel $\big(b_{i,j}c_{j,i}+c_{j,i}b_{i,j}),$ [l] at 90 344
\pinlabel $\partial~b_{i,j}=$ [r] at 470 344
\pinlabel $i$ [l] at 586 320
\pinlabel $j$ [l] at 586 352
\pinlabel $m-1$ [l] at 586 392
\pinlabel $m$ [l] at 586 280
\pinlabel $b_{i,m-1}b_{m,j},$ [l] at 614  344
\pinlabel $\partial~c_{j,i}=$ [r] at 4 94
\pinlabel $j$ [l] at 130 152
\pinlabel $i$ [l] at 130 82
\pinlabel $j+1$ [l] at 130 42
\pinlabel $m$ [l] at 130 114
\pinlabel $b_{j+1,m}c_{m,i}+$ [l] at 194 94
\pinlabel $i$ [l] at 482 40
\pinlabel $j$ [l] at 482 0
\pinlabel $m$ [l] at 482 82
\pinlabel $i-1$ [l] at 482 152
\pinlabel $c_{j,m}b_{m,i-1},$ [l] at 532  94

\endlabellist
\centerline{ \includegraphics[scale=.4]{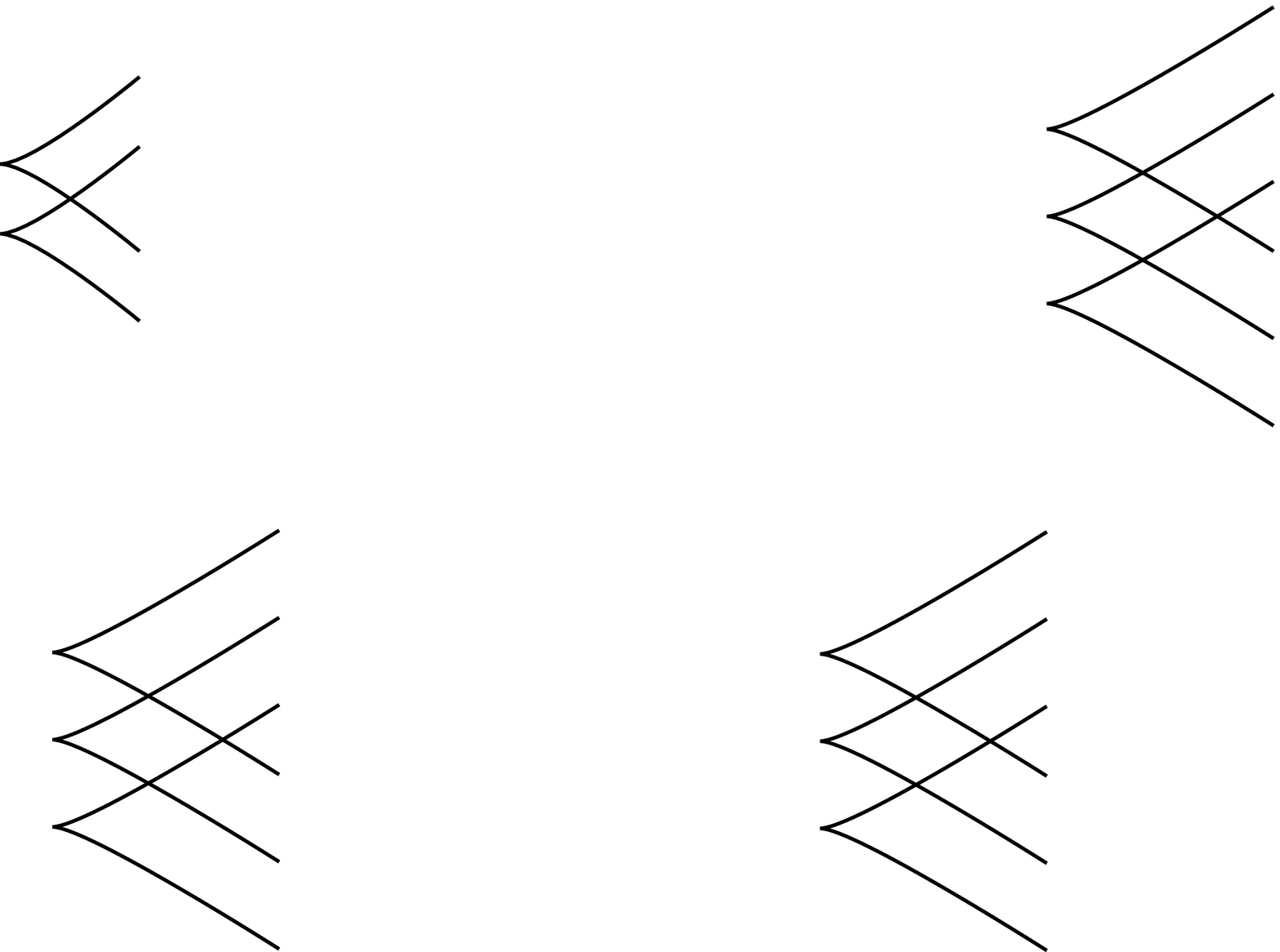} } 
\caption{The differential for $L_{n,\sigma, \mathbf{a}}$.}
\label{fig:Config}
\end{figure}

Among the $L_{n,\sigma, \mathbf{a}}$, there are many examples of pairs of Legendrians whose linearized homology groups have identical ranks, but are distinguished by the ring structure on linearized cohomology.  We briefly recall the definition of the ring structure.  Let $\epsilon: (\A,\partial) \rightarrow (\Z/2,0)$ be a graded augmentation, and write $\A = \oplus_{n=0}^\infty V^{\otimes n}$ where $V$ denotes the linear $\Z/2$-span of the generators of $\A$. 
The conjugated differential
\[
\partial^\epsilon = \Phi_\epsilon \partial (\Phi_\epsilon)^{-1}
\]
has the form
\[
\partial^\epsilon|_{V} = d_1 + d_2 +d_3 +\cdots,  \quad \mbox{with $d_{n}:V \rightarrow V^{\otimes n}$.}
\]
Dualizing gives operations
\[
m_n: (V^*)^{\otimes n} \rightarrow V^*.
\]
The linearized cohomology complex is $(V^*, m_1)$, and the operation $m_2$ descends to a well defined associative product of degree $+1$ on the linearized cohomology
\[
LCH^*(L, \epsilon) := H^*(V^*,m_1).
\]
The set of isomorphism types of the rings 
\[
\{LCH^*(L,\epsilon) \, | \, \epsilon \mbox{ any graded augmentation of $\A$} \}
\]
depends only on the stable tame isomorphism type of $(\A, \partial)$ and is hence a Legendrian isotopy invariant \cite{CEKSW}.

Observe that for any $L_{n, \sigma, \mathbf{a}}$ there is at least one augmentation, $\epsilon_0$, such that $m_1 =0$.  (Take $\epsilon_0$ to vanish identically on generators.)  Moreover, for $L_{n,\sigma, \mathbf{a}}$, the operation $m_2$ is independent of $\epsilon$ (due to the absence of monomials of degree $3$ or more in $\partial$).  Thus, the ring structure on $LCH^*(L,\epsilon_0)$ is independent of the choice of $\epsilon_0$ (with $m_1=0$).  This ring $LCH^*(L_{n,\sigma, \mathbf{a}},\epsilon_0)$ is an invariant of $L_{n,\sigma, \mathbf{a}}$.  

For example, consider the case where $\sigma = \sigma_0$ is 
\[
\sigma_0(1)=1; \quad \sigma_0(2k) = 2k+1;  \quad \sigma_0(2k+1) = 2k;\,\, 1\leq k \leq n-1, \quad \sigma_0(2n) = 2n.
\]
the Poincare polynomial of $LCH^*(L_{n,\sigma_0, \mathbf{a}},\epsilon_0)$ is 
\[
P(t)= \sum_{k \in \Z} (\dim LCH^k(L_{n,\sigma_0, \mathbf{a}},\epsilon_0))t^k = t^2 + \sum_{l=1}^{n-1} \left( t^{2(a_l-a_{l+1})-1} + t^{2(a_{{\ms{l}}+1}-a_{\ms{l}})+2}\right). 
\]
By choosing the sequence $\mathbf{a}= (a_1,a_2, \ldots, a_n)$ appropriately, this can be an arbitrary polynomial of the form
\begin{equation} \label{eq:Poincare}
P(t) = t^2 + \sum_{l=1}^N \left(t^{2c_l-1} + t^{2-2c_l}\right)
\end{equation}
for any $c_1, \ldots, c_N >0$.

\begin{proposition}
\begin{enumerate}
\item For arbitrary $L_{n, \sigma, \mathbf{a}}$, there exists $n'$ and $\mathbf{a}'$ such that the groups
\[
LCH^*(L_{n, \sigma, \mathbf{a}}, \epsilon_0)  \quad \mbox{and} \quad LCH^*(L_{n',\sigma_0, \mathbf{a}'},\epsilon_0)
\]
have the same Poincare polynomial.  

\item If there exists $i<j$ with $|i-j| \geq 2$ such that $\sigma(i) > \sigma(j)$, then these groups are not isomorphic as rings.  In particular, $L_{n, \sigma, \mathbf{a}}$ and $L_{n', \sigma_0, \mathbf{a}'}$ {\ms{are not Legendrian isotopic}}.
\end{enumerate}
\end{proposition}
 
{\ms{Recall that for all $2$-dimensional Legendrian spheres in standard contact $\R^5,$ such as $L_{n, \sigma, \mathbf{a}},$ the ``classical" Legendrian invariants do not contain any information: there is only one smooth isotopy class,  the Thurston-Bennequin number is 1, and the rotation class vanishes \cite{EkholmEtnyreSullivan05a}.}}

\begin{proof}
(1) follows since for arbitrary $L_{n, \sigma, \mathbf{a}}$ the Poincare polynomial has the form given in (\ref{eq:Poincare}).

(2) follows since in $LCH^*(L_{n',\sigma_0, \mathbf{a}'},\epsilon_0)$, the image of $m_2$ is $1$-dimensional.  (It is the span of $z^*$.)  However, when $|i-j| \geq 2$ has $i<j$ and $\sigma(i)>\sigma(j)$, there is some $i<m-1<m<j$ such that $\sigma(m-1)<\sigma(j)<\sigma(i)<\sigma(m)$.  This implies that $m_2(b^*_{i,m-1},b^*_{m,j})= b^*_{i,j}$, so that the image of $m_2$ has dimension $2$ or more. 
\end{proof}

As a specific example where (2) applies, consider 
\[
\begin{array}{ccl} n=3, & \sigma = (2 \, 4 \, 5 \, 3),  & \mbox{and $\mathbf{a} = (2,1,0)$;} \\
n'=4, & \sigma_0 = (2 \, 3) (4\, 5) (6 \, 7), & \mbox{and $\mathbf{a}'= (4,2,1,0)$}.
\end{array}
\] 

\begin{remark}
In the case of $\sigma = \sigma_0$ and $n=2,$ the Legendrians $L_{2,\sigma_0, \mathbf{a}}$ are considered in \cite[Section 6]{BST} with the  generating family homology computed.  After correcting for a small error in the computation in \cite{BST} of the Maslov potential, the generating family homology has the same Poincare polynomial as $LCH^*(L_{2,\sigma_0, \mathbf{a}},\epsilon_0)$.   
\end{remark}

\begin{proof}[Proof of Theorem \ref{thm:Lnsigma}]
We use the extended version of the cellular DGA from Propositions \ref{prop:ExtraCrossing} and \ref{prop:ExtraCone}, so we work with a decomposition of $\pi_x(L_{n, \sigma, \mathbf{a}})$ where all of the crossing arcs (resp. cone points) of $L_1$ (resp. $L_2$)  correspond to a single circle (resp. single $0$-cell) in the $1$-skeleton of $\mathcal{E}$.  Moreover, by methods similar to those used in the proof of Proposition \ref{prop:ExtraCrossing}, we can position the cusp arcs so that together their projection is as in Figure \ref{fig:LnSigma}.  
In more detail, the cusp edges of the $L_2$ part of $L_{n, \sigma, \mathbf{a}}$ all project to a circle in $\R^2$; the cusp edges that are part of $L_{n, \sigma, \mathbf{a}}$  to the left of $L_2$ (i.e.,  $L_1$ together with the tubes connecting $L_1$ to $L_2$) project to an arc with two end points on the $L_2$ circle.
This leads to the cellular decomposition pictured in Figure \ref{fig:LnSigma}.  Since the cusp edges above any given edge all involve distinct pairs of sheets, the only modification to the usual definition of the DGA is that when computing the differential of bordering cells, at $0$-cells (resp. $1$-cells)  we insert one $N = \left[\begin{array}{cc} 0 & 1 \\ 0 & 0 \end{array} \right]$ block (resp. $2\times 2$ $0$ block) for every pair of sheets that cusp above the $0$-cell.  In particular, above the $0$-cells $v_0$ and $v_1$ the associated matrices have $N$ blocks all along the diagonal.    

\begin{figure}

\labellist
\small
\pinlabel $C_1$  at 80 82
\pinlabel $C_2$ at 186 106
\pinlabel $C_3$  at 330 116 
\pinlabel $B_1$ [bl] at 180 62
\pinlabel $A_1$ [r] at 116 82 
\pinlabel $v_0$ [tr] at 240 36
\pinlabel $v_1$ [b] at 232 128
\pinlabel $B_2$ [tr] at 60 45 
\pinlabel $A_2$ [l] at 310 82
\pinlabel $B_3$ [l] at 228 92
\pinlabel $B_4$ [l] at 280 62

\endlabellist
\centerline{ \includegraphics[scale=.7]{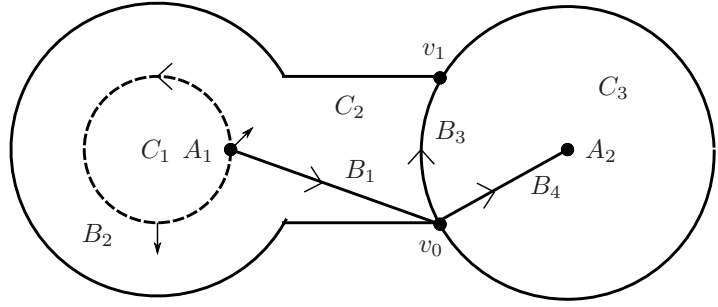} } 
\caption{The polygonal decomposition for $L_{n,\sigma, \mathbf{a}}$.}
\label{fig:LnSigma}
\end{figure}

\medskip

\noindent {\bf Step 1.}  {\it Canceling generators in the $L_2$ part of $L$.}

\medskip

Above $C_3$, $L$ has $2a_1+2$ sheets.  We have differentials
\begin{align}
\label{eq:B4diff} \partial B_4 & = N (I+B_4) +(I+B_4) A_2 \\
\label{eq:C3diff} \partial C_3 & = N C_3 + C_3 N + (I + \widehat{B}_3) + (I+B_{4}) U (I+B_4)^{-1}  
\end{align}
where
\begin{equation} \label{eq:Udiff}
U = U_{a_1}U_{a_1-1} \cdots U_1, \quad U_p = I +  \sum_{l< 2p} a^{\ms{2}}_{l,2p+1} E_{l,2p} + \sum_{2p+1 < l} a^{\ms{2}}_{2p,l} E_{2p+1,l}, \quad N = \sum_{k=1}^{a_1+1} E_{2k-1,2k}.
\end{equation}
That is, $N$ is block diagonal with the $2\times 2$ blocks $\left[ \begin{array}{cc} 0 & 1 \\ 0 & 0 \end{array} \right]$ along the diagonal.  

Note that the matrix $A_2$ has $0$'s in the $(2,3), (4,5), \ldots, (2a_1,2a_1+1)$ entries.  Therefore, we compute from (\ref{eq:B4diff})
\[
\partial b^4_{2k,2k+1} = 0, \quad \mbox{for $ 1\leq k \leq a_1$.}
\]
Thus, we can adjust our usual triangular partial ordering of generators (as in Lemma \ref{lem:precA}), $\prec_\A$, by taking these $b^4_{2k,2k+1}$ to be the smallest generators.  With this adjustment to the ordering, for all $i<j$ with $(i,j) \neq (2k,2k+1)$, we can cancel
\[
\partial \framebox{$b^4_{i,j}$} = \framebox{$a^2_{i,j}$} + o(a^2_{i,j})
\]
where (canceling with $j-i$ increasing), the $o(a^2_{i,j})$ term belongs to the sub-algebra generated by generators that are strictly smaller than $a^2_{i,j}$ with respect to $\prec$.  In the quotient, we have
\[
B_4 \doteq \sum_{k=1}^{a_1} b^4_{2k,2k+1} E_{2k,2k+1}; \quad  \partial B_4 \doteq 0,
\]
and the latter equality shows that
\begin{equation}  \label{eq:A2IB4}
A_2 \doteq (I+B_4)^{-1} N(I+B_4) \doteq (I+\sum_{k=1}^{a_1} b^4_{2k,2k+1} E_{2k,2k+1}) \left(\sum_{l=1}^{a_1+1} E_{2l-1,2l} \right) (I+\sum_{k=1}^{a_1} b^4_{2k,2k+1} E_{2k,2k+1}) 
\end{equation}

From (\ref{eq:Udiff}) and (\ref{eq:A2IB4}), we observe that the $(2l-1,2l)$-entries of $U$ are given by
\[
\begin{array}{lll}
(U)_{1,2} & = a^2_{1,3} & \doteq b^4_{2,3}, \\
(U)_{2l-1,2l} & = a^2_{2l-1,2l+1} +a_{2l-2,2l} & \doteq b^4_{2l, 2l+1} + b^4_{2l-2,2l-1}, \\
(U)_{2a_1+1,2a_1+2} & = a^2_{2a_1,2a_1+2} & \doteq b^4_{2a_1,2a_1+1}.
\end{array}
\]

Thus, for $1 \leq l \leq a_1$ we can cancel 
\[
\partial \framebox{$c^3_{1,2}$} \doteq \framebox{$ b^4_{2,3}$} + (\widehat{B_3})_{1,2},
\]
\[
\partial \framebox{$c^3_{2l-1,2l}$} \doteq \framebox{$ b^4_{2l,2l+1}$} + b^4_{2l-2,2l-1} + (\widehat{B_3})_{2l-1,2l}.
\]
We DO NOT cancel $c^3_{2a_{1}+1,2a_{1}+2}$.  Instead, set $z = c^3_{2a_{1}+1,2a_{1}+2}$, and observe that
\begin{equation} \label{eq:Zdiff}
\partial z \doteq (\widehat{B_3})_{1,2}+(\widehat{B_3})_{3,4}+ \cdots+  (\widehat{B_3})_{2a_{1}+1,2a_{1}+2}.
\end{equation}

Finally, we cancel all remaining $C_{\ms{3}}$ generators,  with themselves; note that (\ref{eq:C3diff}) gives for $i<2l<2l+1 < j$,
\[
\partial \framebox{$c^{\ms{3}}_{i,2l+1}$} = \framebox{$c^{\ms{3}}_{i,2l}$} + X;  \quad \partial \framebox{$c^{\ms{3}}_{2l,j}$} = \framebox{$c^{\ms{3}}_{2l+1,j}$} + Y
\]
where $X$ belongs to the sub-algebra generated by generators $x$ satisfying $x \prec_\mathcal{A} c^{\ms{3}}_{i,2l}$, and similar for $Y$.  

To summarize, the only generator remaining from the matrices $A_2, B_4, C_3$ is $z= c^3_{2a_{1}+1,2a_{1}+2}$ whose differential is given by (\ref{eq:Zdiff})

\medskip

\noindent {\bf Step 2.}  {\it Canceling generators in the $L_1$ part of $L$.}

\medskip

Above $L_1$, $L$ has $2n$ sheets.  The matrices $B_1,B_3, C_2, C_3$ have generators in all of there entries with $1 \leq i< j \leq n$.  The matrices $A_1$ and $B_2$  have generators $a^1_{i,j}$ and $b^2_{i,j}$ in the entries with $i<j$ and $\sigma(i) < \sigma(j)$, with the remaining entries $0$ (as specified Proposition \ref{prop:ExtraCrossing}).  
{\ms{There are no generators associated to the vertices $v_0$ and $v_1.$}}
As indicated in Figure \ref{fig:LnSigma}, we label sheets above these cells as they appear above $C_3$.  The notation for these matrices is fixed.  We have differentials
\begin{align}
\partial B_1 & = A_1 (I+B_1) + (I+B_1) N, \\
\partial C_1 & = Q_\sigma A_1 Q^{-1}_\sigma C_1 + C_1  Q_\sigma A_1 Q^{-1}_\sigma +  Q_\sigma (I+B_2) Q^{-1}_\sigma + I,  \label{eq:DiffC11}\\
\partial C_2 & = N C_2 + C_2 N + (I+B_3) + (I+B_1)(I+B_2)(I+B_1)^{-1}.
\end{align}
Recall $Q_\sigma = \Sigma E_{\sigma(i), i}$ so that the $(i,j)$-entry of $Q_\sigma A_1 Q^{-1}$ is $a^1_{\sigma^{-1}(i), \sigma^{-1}(j)}.$

Start by canceling
\[
\partial \framebox{$C_2$} = \framebox{$B_3$} + \cdots.
\]
Observe, that in the quotient
\begin{equation} \label{eq:doteq}
(B_3)_{i,i+1} \doteq (B_2)_{i,i+1}
\end{equation}
(because the $(i,i+1)$-entry of  $NC_3$ and $C_3N$ is $0$ since both factors are upper triangular, and the $(i,i+1)$-entry of $(I+B_1)(I+B_2)(I+B_1)^{-1}$ is the sum of the $(i,i+1)$-entries of the three factors).

Next, we will cancel as many of the entries of $B_1$ with $A_1$ as we can.  When doing so, in order to meet the hypothesis of Theorem \ref{thm:Alg}, we will need to modify the ordering of generators.  For the entries of $B_1$ and $A_1$, we require that $x_{i',j'} \prec y_{i,j}$ if and only if $j'-i' < j-i$ (where each of  $x$ and $y$ has the form $a^1$ or $b^1$).  Remaining generators are all larger than those from $B_1$ and $A_1$ and are ordered as in Lemma \ref{lem:precA}.  

With this ordering, we can cancel
\[
\partial \framebox{$b^1_{i,j}$} = \framebox{$a^1_{i,j}$} + \cdots, \quad \forall i<j \mbox{ with } \sigma(i) < \sigma(j).
\]
In the quotient, the values of $a^1_{i,j}$ are as follows, depending on the parity of $i$ and $j$:
\begin{equation} \label{eq:parity}
\begin{array}{lcl}
\mbox{$i$ odd, $j$ even:}&  \quad a^1_{i,j}  & \doteq \left\{\begin{array} {cr} 1, &  (i,j) = (2l-1,2l) \\ 0, & \mbox{else.} \end{array} \right. \\
\mbox{$i$ even, $j$ even:}&   \quad a^1_{i,j} & \doteq b^{\ms{1}}_{i,j-1} \\
\mbox{$i$ odd, $j$ odd:}&  \quad a^1_{i,j}   &  \doteq b^{\ms{1}}_{i+1,j} 
\end{array}
\end{equation}
[These formula are verified using induction on $j-i$ and the relation
\[
a^1_{i,j} \doteq \sum_{i<m<j} a_{i,m}^1 b_{m,j}^1 + \left( B_1 N \right)_{i,j}.
\]
In doing so, it is useful to note that if $b^{\ms{1}}_{m,j}$ is not $0$ in the quotient then, since $\sigma(m) > \sigma(j)$, we must have $m$ even and $j$ odd.]

The generators $b^1_{i,j}$ with $i<j$ and $\sigma(i) > \sigma(j)$ will not be canceled.  They have differentials
\begin{equation}  \label{eq:b1ijdiff}
\partial b^1_{i,j} = \sum_{i<m<j} a^1_{i,m}b^1_{m,j} \doteq \sum_{\begin{array}{c} i<m<j \\ \sigma(i) > \sigma(m-1) \\ \sigma(m) > \sigma(j) \end{array}} 
b^{\ms{1}}_{i,m-1}b^{\ms{1}}_{m,j}. 
\end{equation}
[At the first equality, note that there is no $a^1_{i,j}$ generator, and for the second equality use that $b^1_{m,j}=0$ unless $m$ is even.]

Finally, we cancel as many entries of $C_1$ as we can with the entries of $B_2$, using
\[
\partial \framebox{$C_1$} = Q^{-1}_\sigma \framebox{$B_2$} Q_\sigma + ...
\]
In doing so, it is convenient to rewrite (\ref{eq:DiffC11}) as
\[
{\ms{\partial Q^{-1}_{\sigma} C_1 Q_{\sigma} = A_1 Q^{-1}_{\sigma} C_1 Q_\sigma + Q^{-1}_{\sigma} C_1 Q_\sigma A_1 +B_2}},
\]
giving the entry-by-entry formula
\[
\partial c^1_{\sigma(i),\sigma(j)} = \sum_{i<m, \sigma(i) < \sigma(m)<\sigma(j)} a^1_{i,m} c^1_{\sigma(m), \sigma(j)} + \sum_{m<j, \sigma(i) < \sigma(m)<\sigma(j)} c^1_{\sigma(i),\sigma(m)} a^1_{m,j} + \left\{\begin{array}{cr} b^2_{i,j}, &  i<j \\ 0, & i>j. \end{array}\right. 
\]
Thus, for all $i<j$ with $\sigma(i) < \sigma(j)$, we cancel
\[
\partial \framebox{$c^1_{\sigma(i),\sigma(j)}$} = \framebox{$b^2_{i,j}$} + \cdots.
\]
[To apply Theorem \ref{thm:Alg}, we order so that all entries of $B_2$ are greater than all entries of $B_1$ (the $A_1$ generators belong to the subalgebra generated by the $B_1$ generators in the current quotient), and so that $c_{\sigma(i),\sigma(j)} \prec b^2_{i',j'}$ if and only if ${\ms{\sigma(j)-\sigma(i)}} < j'-i'$.]
All $c^1_{\sigma(i),\sigma(j)}$ generators with $j<i$ and $\sigma(i) < \sigma(j)$ remain in the quotient.  For such generators, $j$ is even and $i$ is odd so we use (\ref{eq:parity}) to compute
\begin{align}
\partial c^1_{\sigma(i), \sigma(j)} & =  \sum_{i<m, \sigma(i) < \sigma(m)<\sigma(j)} a^1_{i,m} c^1_{\sigma(m), \sigma(j)} + \sum_{m<j, \sigma(i) < \sigma(m)<\sigma(j)} c^1_{\sigma(i),\sigma(m)} a^1_{m,j} \label{eq:c1ijdiff} \\
 & \doteq \sum_{i<m; j<m; \sigma(i) < \sigma(m)<\sigma(j)} b^1_{i+1,m} c^1_{\sigma(m), \sigma(j)} +  \sum_{m<j; m<i; \sigma(i) < \sigma(m)<\sigma(j)} c^1_{\sigma(i),\sigma(m)} b^1_{m,j-1}. \notag
\end{align}
Note that these formulas agree with those from the statement of the theorem with the role of $i$ and $j$ interchanged.

To summarize, in Step 1 and Step 2 we have now cancelled all of the generators besides $z$ and those $b^1_{i,j}$ and $c^1_{\sigma(j), \sigma(i)}$ with $i<j$ and $\sigma(i) > \sigma(j)$.  The differentials of the $b^1_{i,j}$ and $c^1_{\sigma(j),\sigma(i)}$ have been computed in (\ref{eq:b1ijdiff}) and (\ref{eq:c1ijdiff}).  It remains only to calculate $\partial z$.  Using that 
\begin{align*}
b^2_{2l-1,2l} & \doteq \sum_{2l-1<m, \sigma(2l-1) < \sigma(m)<\sigma(2l)} a^1_{2l-1,m} c^1_{\sigma(m), \sigma(2l)} + \sum_{m<2l, \sigma(2l-1) < \sigma(m)<\sigma(2l)} c^1_{\sigma(2l-1),\sigma(m)} a^1_{m,2l} \\
 & \doteq \sum_{2l<m, \sigma(m)<\sigma(2l)} b^1_{2l,m} c^1_{\sigma(m), \sigma(2l)} + \sum_{m<2l-1, \sigma(2l-1) < \sigma(m)} c^1_{\sigma(2l-1),\sigma(m)} b^1_{m,2l-1}.
\end{align*}
[At the second equality, we used that the condition on the first (resp. second) summation implies $m$ crosses $2l$ (resp. $m$ crosses $2l-1$), so $m$ is odd (resp. even).]  Since 
\begin{align*}
\partial z &  \doteq (\widehat{B_3})_{1,2}+(\widehat{B_3})_{3,4}+ \cdots+  (\widehat{B_3})_{2a_{1}+1,2a_{1}+2} \\ 
 & = b^3_{1,2}+b^3_{3,4}+ \cdots+  b^3_{2l-1,2l}  \\
 & \doteq \sum_{l=1}^n b^2_{2l-1,2l} \\
 \end{align*}
 the result follows. [The second equality uses that $(\widehat{B_3})_{2r-1, 2r}$ is $0$ unless there is a tube connecting one of the unknots on the left to the $r$-th cusp edge.  The final inequality uses (\ref{eq:doteq}).]



\end{proof}



\begin{thebibliography}{99}

\bibitem{ArnoldGuseinZadeVarchenko}
V.I. Arnold, S. M. Gusein-Zade, A.N. Varchenko.
\newblock Singularities of differentiable maps. Volume 1.
{\em Modern Birkhauser Classics}, Monogr. Math., 82, Birkh�user Boston, Boston, MA, 1985

\bibitem{BourgeoisEkholmEliashberg12}
Frederic Bourgeois, Tobias Ekholm and Yakov Eliashberg.
\newblock Effect of Legendrian surgery.
\newblock  Geom. Topol. 16 (2012), no. 1, 301--389

\bibitem{BST} F. Bourgeois, J. Sabloff, L. Traynor.
\newblock Lagrangian cobordisms via generating families: construction and geography. 
\newblock Algebr. Geom. Topol. 15 (2015), no. 4, 2439–2477

\bibitem{Chantraine15}  B. Chantraine
\newblock Lagrangian concordance is not a symmetric relation
\newblock Quantum Topology 6  (2015), No. 3, 451-474.

\bibitem{Chekanov02}
Yuri Chekanov.
\newblock Differential algebra of Legendrian links.
{\em Invent. Math.}, 150 (2002), no. 3, 441--483.

\bibitem{CEKSW}
Gokhan Civan, John B. Etnyre, Paul Koprowski, Joshua M. Sabloff, Alden Walker. 
\newblock Product structures for Legendrian contact homology. 
Mathematical Proceedings of the Cambridge Philosophical Society (2011)  150:02, 291--311.

\bibitem{DimitroglouRizell11}
Georgios Dimitroglou Rizell.
\newblock{Knotted Legendrian surfaces with few Reeb chords},
\newblock {\em Alg. Geom. Top.} 11 (2011) 2903--2936.


\bibitem{DimitroglouRizell12}
Georgios Dimitroglou Rizell.
\newblock{Legendrian Ambient Surgery and Legendrian Contact Homology },
\newblock ar$\chi$iv: 1205.5544

\bibitem{DimitroglouRizellGolovko14}
Georgios Dimitroglou Rizell and Roman Golovko.
\newblock{Estimating the number of Reeb chords using a linear representation of the characteristic algebra}
\newblock ar$\chi$iv:1409.6278

\bibitem{Ekholm07}
Tobias Ekholm.
\newblock Morse flow trees and Legendrian contact homology in 1-jet spaces.
\newblock {\em Geom. Topol.}, 11:1083--1224, 2007.

\bibitem{EEMurphySmith}
T. Ekholm, Y. Eliashberg, E. Murphy, and I. Smith. 
\newblock{Constructing exact Lagrangian immersions with few double points}
\newblock {\em GAFA}, 23 (2013), no. 6, 1772--1803.


\bibitem{EESabloff}  
T. Ekholm, J. Etnyre, and J. Sabloff.
\newblock{A duality exact sequence for Legendrian contact homology}.
\newblock {\em Duke Mathematical Journal.}  150 (2009), no. 1, 1--75 






\bibitem{EkholmEtnyreNgSullivan11}
Tobias Ekholm, John Etnyre, Lenny Ng, and Michael Sullivan.
\newblock{Knot contact homology}.
\newblock {\em Geom. Topol.} 17 (2013), no. 2, 975--1112

\bibitem{EkholmEtnyreSullivan05b}
Tobias Ekholm, John Etnyre, and Michael Sullivan.
\newblock The contact homology of {L}egendrian submanifolds in {${\mathbb R}\sp
  {2n+1}$}.
\newblock {\em J. Differential Geom.}, 71(2):177--305, 2005.

\bibitem{EkholmEtnyreSullivan05a}
Tobias Ekholm, John Etnyre, and Michael Sullivan.
\newblock Non-isotopic {L}egendrian submanifolds in {$\mathbb R\sp {2n+1}$}.
\newblock {\em J. Differential Geom.}, 71(1):85--128, 2005.

\bibitem{EkholmEtnyreSullivan05c}
Tobias Ekholm, John Etnyre, and Michael Sullivan.
\newblock Orientations in Legendrian contact homology and exact Lagrangian immersions. 
\newblock Internat. J. Math. 16 (2005), no. 5, 453--532.

\bibitem{EkholmEtnyreSullivan07}
Tobias Ekholm, John Etnyre, and Michael Sullivan.
\newblock Legendrian contact homology in {$P\times\mathbb R$}.
\newblock {\em Trans. Amer. Math. Soc.}, 359(7):3301--3335 (electronic), 2007.

\bibitem{EkholmKalmanHonda12}
Tobias Ekholm, Tamas Kalman and Ko Honda.
\newblock  Legendrian knots and exact Lagrangian cobordisms.
\newblock ar$\chi$iv:1212.1519.

\bibitem{EkholmKalman08}
Tobias Ekholm and Tamas Kalman.
\newblock Isotopies of Legendrian 1-knots and Legendrian tori.
\newblock {\em J. Symp. Geom.}, 2008, no 4, 407--460.

\bibitem{EkholmNgShende16}
Tobias Ekholm, Lenny Ng, and Vivek Shende.
\newblock  A complete knot invariant from contact homology.
\newblock ar$\chi$iv:1606.07050 


\bibitem{Eliashberg00}
Yakov Eliashberg.
\newblock Invariants in contact topology,
\newblock  {\em Proceedings of the International Congress of Mathematicians}, 
Vol.~II (Berlin, 1998).  Doc. Math.  1998,  Extra Vol. II, 327--338.

\bibitem{EliashbergGiventalHofer00}
Yakov Eliashberg, Alexander Givental, and Helmut Hofer.
\newblock Introduction to symplectic field theory,
\newblock {\em Geom. Funct. Anal.}, 2000,  Special Volume, Part II, 560--673.

\bibitem{Fu} D. Fuchs, The Chekanov-Eliashberg invariant of Legendrian knots: Existence of augmentations. \textsl{J. Geom. Phys.} {\bf 47 } (2003), no. 1, 43-65.



\bibitem{FuRu}   D. Fuchs, D. Rutherford. 
\newblock Generating families and Legendrian contact homology in the standard contact space.
\newblock {\em J. Topol.} 4 (2011) 190�226.

\bibitem{GordonLidman15}
Cameron Gordon and Tye Lidman.
\newblock Knot contact homology detects cabled, composite, and torus knots. 
\newblock ar$\chi$iv:1509.01642 


\bibitem{HarperSullivan}
John Harper and Michael Sullivan.
\newblock{A bordered Legendrian contact algebra}
\newblock {\em J. Symp. Geom.}, 2014, no 2, 237-255.

\bibitem{HatcherWagoner73}
Allen Hatcher and John Wagoner.
\newblock Pseudo-isotopies of compact manifolds. 
\newblock Asterisque, No. 6. Soc. Math. de France, Paris, 1973. 275 pp.


\bibitem{HenryRutherford13}
Michael B. Henry and Daniel Rutherford.
 \newblock A combinatorial DGA for Legendrian knots from generating families. 
 \newblock Commun. Contemp. Math. 15 (2013), no. 2, 1250059, 60 pp.

\bibitem{HenryRutherford15}
Michael B. Henry and Daniel Rutherford.
 \newblock Ruling polynomials and augmentations over finite fields.
\newblock \emph{J. Topology} 8 (2015), no. 1, 1-37 

\bibitem{Lowell15}
Mark Lowell.
\newblock A Seifert-van Kampen Theorem for Legendrian Submanifolds and Exact Lagrangian Cobordisms.
\newblock ar$\chi$iv:1511.03375.

\bibitem{Ng03}
Lenny Ng.
\newblock Computable Legendrian invariants. 
\newblock Topology 42 (2003), no. 1, 55--82.

\bibitem{Ng08}
Lenny Ng.
\newblock Framed knot contact homology. 
\newblock Duke Math. J. 141 (2008), no. 2, 365--406

 
\bibitem{NRSSZ}
Lenhard Ng, Dan Rutherford, Vivek Shende, Steven Sivek and Eric Zaslow
\newblock Augmentations are Sheaves.
\newblock ar$\chi$iv:1502.04939  



\bibitem{Rutherford06}  D. Rutherford, \textsl{The Thurston-Bennequin number, Kauffman polynomial, and ruling invariants of a Legendrian link: The Fuchs conjecture and beyond, Int. Math. Res. Not.} (2006), Art. ID 78591.



\bibitem{RuSu2}
Dan Rutherford and Michael Sullivan.
\newblock{Cellular computation of Legendrian contact homology for surfaces, Part II.}
\newblock Preprint.

\bibitem{RuSu3}
Dan Rutherford and Michael Sullivan. In preparation.


\bibitem{Sab} J. M. Sabloff, Augmentations and rulings of
  Legendrian knots, \textit{Int. Math. Res. Not.} \textbf{2005},
  no. 19, 1157--1180.


\bibitem{Sivek11}
Steven Sivek.
\newblock{A bordered Chekanov-Eliashberg algebra}
\newblock {\em J. Topology} 4 (2011), no. 1, 73--104.





\end{thebibliography}
\end{document}